\documentclass[11pt]{article}
\usepackage[utf8]{inputenc}
\usepackage[T1]{fontenc}
\usepackage[normalem]{ulem}
\usepackage{parskip}
\usepackage[margin=1in]{geometry}
\usepackage{amsmath}
\usepackage{amssymb}
\usepackage{amsthm}
\usepackage{amsfonts}
\usepackage{mathtools}
\usepackage{braket}
\usepackage{bbm}
\usepackage{graphicx}
\usepackage{caption}
\usepackage{subcaption}
\usepackage{listings}
\usepackage{algorithm}
\usepackage{algpseudocode}
\usepackage{xcolor}
\usepackage{mymacros}
\usepackage{pgf,interval}
\usepackage{tabularx}
\usepackage{booktabs}
\usepackage[inline]{enumitem}
\usepackage[%
backref=true,%
natbib=true,%
maxbibnames=99,%
doi=false,%
url=false,%
giveninits=true]{biblatex}
\usepackage{hyperref}
\usepackage{pgfplots}
\usepackage[title]{appendix}

\DefineBibliographyStrings{english}{backrefpage={page}, backrefpages={pages}}
\intervalconfig{soft open fences}
\setitemize{itemsep=1pt}
\newcommand{\seti}{\mathcal{X}}

\renewcommand{\tilde}[1]{\widetilde{#1}}
\renewcommand{\hat}[1]{\widehat{#1}}

\newcommand{\hfcn}{h}
\newcommand{\fcnl}{\varphi}
\newcommand{\txi}{\tilde{\xi}}

\newcommand{\benders}{\text{BD}}
\newcommand{\masterp}{\text{MP}}

\newcommand{\dsubp}{\text{DSP}}

\newcommand{\cvar}{\mathbf{CVaR}}

\newcommand{\Probt}{\mathop{\mathbb{P{}}_{\text{true}}}}
\newcommand{\ProbQ}{\mathop{\mathbb{Q{}}}}

\newcommand{\setcross}{\mathcal{S}}
\newcommand{\CCP}{\text{CCP}}
\newcommand{\SAA}{\text{SAA}}

\newcommand{\DRCC}{\text{DR-CCP}}

\newcommand{\DRCCbC}{\text{DRC}}
\newcommand{\hProb}{\hat{\Prob{}}}

\newcommand{\CVaR}{\textbf{CVaR}}
\newcommand{\VaR}{\textbf{VaR}}
\newcommand{\Path}{\textsc{Path}}
\newcommand{\OOS}{\text{OOS}}

\newcommand{\ltag}[1]{\text{(#1)}}

\newcommand{\edits}[1]{{\color{black} #1}}
\newcommand{\editsII}[1]{{\color{black} #1}}

\title{ {\bf Chance-Constrained Set Covering with Wasserstein Ambiguity}
\author{\normalsize {\bf Haoming Shen} and {\bf Ruiwei Jiang} \\
{\small Department of Industrial and Operations Engineering}\\
{\small University of Michigan, Ann Arbor, MI 48109}\\[1mm]
{\small Email: \{hmshen, ruiwei\}@umich.edu}\\
}
}
\date{}
\hypersetup{
pdfauthor={Haoming Shen},
pdftitle={Distributionally Robust Chance-Constrained Set Covering},
pdfkeywords={},
pdfsubject={},
pdfcreator={hmshen},
pdflang={English}}

\begin{document}

\maketitle

\begin{abstract}
We study a generalized distributionally robust chance-constrained set covering
problem (\DRCCbC{}) with a Wasserstein ambiguity set, where both decisions and
uncertainty are binary-valued. We establish the NP-hardness of \DRCCbC{} and
recast it as a two-stage stochastic program, which facilitates decomposition
algorithms. Furthermore, we derive two families of valid inequalities. The first
family targets the hypograph of a ``shifted'' submodular function, which is
associated with each scenario of the two-stage reformulation. We show that the
valid inequalities give a complete description of the convex hull of the
hypograph. The second family mixes inequalities across multiple scenarios and
gains further strength via lifting. Our numerical experiments demonstrate the
out-of-sample performance of the \DRCCbC{} model and the effectiveness of our
proposed reformulation and valid inequalities.

\vspace{0.25in}

\end{abstract}


\section{Introduction}%
\label{sec:intro}
%
We consider a set covering model, in which one selects a subset of \(n\)
elements to cover \(I\) targets with a minimum cost. If we employ
\(x \in \binaries^n\) to denote a binary decision vector such that, for all
\(j \in [n] := \{1, \ldots, n\}\), \(x_j = 1\) if element \(j\) is selected and
\(x_j = 0\) otherwise, then we can formulate such coverage requirement as a set
of linear inequalities:
\begin{align}
  x^{\top} \xi_i \geq 1, \quad \forall i \in [I],
  \label{eqn:sc-constr}
\end{align}
where \(\xi_i := [\xi_{i1}, \ldots, \xi_{in}]^{\top}\) denotes a vector of
binary parameters such that \(\xi_{ij} = 1\) if element \(j\) is able to cover
target \(i\) and \(\xi_{ij} = 0\) otherwise. Set covering models find a variety
of real-world applications, including scheduling~\cite{smith-1988-bus-crew},
production planning~\cite{vasko-1989-set-cover}, facility
location~\cite{gunawardane-1982-dynam-version}, and vehicle
routing~\cite{bramel-1997-effec-set}, etc. Consider, for example, we wish to
build medical facilities among \(n\) locations to cover \(I\) target residential
regions. In this context, \(\xi_i\) describes the connection between target
\(i\) and all candidate locations and it may depend on their distances, e.g.,
\(\xi_{ij} = 1\) if target \(i\) is near location \(j\). Accordingly,
constraints~\eqref{eqn:sc-constr} ensure that every target is within the
neighborhood of some open facilities. This model can be generalized to
incorporate backup coverage (see~\cite{yang-2005-gener-weigh}). For example,
crucial targets with higher priority (e.g., nursing homes) may be covered by
multiple open facilities. This generalizes constraints~\eqref{eqn:sc-constr} to
\(x^{\top} \xi_i \geq v_i, \forall i \in [I]\), where \(v_i \geq 1\) denotes an
integer constant that represents the coverage level of target \(i\).

In emergency (\eg{}, natural disasters), the connection between target \(i\) and
open facilities may be randomly disrupted. In that case, it is convenient to
model \(\xi_i\) as Bernoulli random variables, denoted by \(\txi_i, i \in [I]\),
and formulate the generalized set covering constraints probabilistically:
\begin{align}
  \Probt{}\Set{x^{\top} \txi_i \geq v_i, \forall i \in [I]} \geq 1 - \epsilon,
  \label{eqn:ccbc-constr}
\end{align}
where \(\Probt{}\) denotes the joint probability distribution of \(\txi\) and
\(\epsilon\), often chosen to be small such as \(0.1\) and \(0.05\), denotes a
pre-specified risk level. Intuitively, a selection \(x\) that
satisfies~\eqref{eqn:ccbc-constr} can maintain the desired coverage level of
\emph{all} targets with high probability. Chance
constraint~\eqref{eqn:ccbc-constr} is called single if \(I = 1\) and joint if
\(I \geq 2\).

In reality, our knowledge on \(\Probt{}\) is usually ambiguous. For
example, the historical data of \(\txi\) may be limited because of the
infrequency of natural disasters. Due to such ambiguity on
\(\Probt{}\), we adopt a distributionally robust
perspective. Specifically, we assume the access to a set of \(N\)
independent and identically distributed (\iid{}) samples drawn from
\(\Probt{}\), denoted as \(\{\hat{\xi}^{j}\}_{j \in [N]}\). This gives
rise to an empirical distribution
\(\hProb{}_{\tilde{\xi}} := (1/N)\sum_{j=1}^N \Delta_{\hat{\xi}^{j}}\)
and provides an approximation to \(\Probt{}\), \edits{where
  $\Delta_{\hat{\xi}^{j}}$ denotes the Dirac measure on the singleton
  $\{\hat{\xi}^{j}\}$}. Unfortunately, as we will demonstrate in
Section \ref{sec:exps-oos}, simply replacing \(\Probt{}\) with
\(\hProb{}_{\tilde{\xi}}\) in chance
constraint~\eqref{eqn:ccbc-constr} may not produce a feasible
solution. That is, a solution to such empirical approximation has low
confidence of satisfying the actual chance constraint
\eqref{eqn:ccbc-constr} under \(\Probt{}\). As an alternative,
distributionally robust approaches consider all distributions lying in
a neighborhood of \(\hProb{}_{\tilde{\xi}}\). We denote this
neighborhood as an ambiguity set \(\mathcal{P}\), which can be defined
based on various discrepancy measures between probability
distributions, e.g., Hellinger distance, Prokhorov metric, Wasserstein
distance, etc. In this paper, we adopt the Wasserstein distance (see,
e.g.,~\cite{mohajerin-2018-data-driven}), which measures the
discrepancy between distributions \(\Prob{}_1\) and \(\Prob{}_2\) by
\begin{gather*}
  d_W(\Prob{}_1, \Prob{}_2)
  := \inf_{\ProbQ{} \sim (\Prob{}_1, \Prob{}_2)} \Expect{}_{\ProbQ{}} \left[
    \norm{\tilde{X}_1 - \tilde{X}_2}_p \right],
\end{gather*}
where \(\tilde{X}_1\), \(\tilde{X}_2\) are random variables following
distributions \(\Prob{}_1\), \(\Prob{}_2\) respectively,
\(\ProbQ{} \sim (\Prob{}_1, \Prob{}_2)\) denotes that \(\ProbQ{}\) is a joint
distribution of \(\tilde{X}_1\) and \(\tilde{X}_2\) with marginals
\(\Prob{}_{1}\) and \(\Prob{}_{2}\), and \(\norm{\cdot}_p\) denotes the
\(p\)-norm. Intuitively, \(\ProbQ{}\) is a plan of transporting the probability
masses of \(\Prob{}_1\) to make it coincide with \(\Prob{}_2\), and
\(d_W(\Prob{}_1, \Prob{}_2)\) equals the minimum transportation cost evaluated
under the \(p\)-norm. Accordingly, we consider the following Wasserstein
ambiguity set:
\begin{gather}
  \mathcal{P}
  := \Set{
    \Prob \in \mathcal{P}_0(\binaries^{I \times n}) \colon
    d_W(\Prob, \hProb{}_{\tilde{\xi}}) \leq \delta}, \label{eq:intro-wass-ambiguity}
\end{gather}
where \(\mathcal{P}_0(\binaries^{I \times n})\) denotes the set of all
distributions supported on \(\binaries^{I \times n}\) and \(\delta > 0\) denotes
the radius of the ambiguity set. This leads to the following generalized
distributionally robust chance-constrained set covering problem:
\begin{flalign}
  \ltag{\textbf{\DRCCbC{}}} && \min_x ~~
  & c^\top x, && \nonumber{} \\
  && \st{} ~~
  & \inf_{\Prob \in \mathcal{P}} \Prob \Set{ x^{\top} \txi_i \geq v_i , \forall i \in [I] }
    \geq 1 - \epsilon, && \label{eqn:drccbc-constr} \\
  && & x \in \binaries^n, && \nonumber{}
\end{flalign}
where \(c \in \reals^n\) denotes a deterministic cost vector and constraint
\eqref{eqn:drccbc-constr} requires to satisfy the chance constraint under all
distributions within \(\mathcal{P}\). If \(\delta\) is sufficiently large that
\(\Probt{} \in \mathcal{P}\) then any feasible solution \(x\) to (\DRCCbC{})
satisfies the ``true'' chance constraint \eqref{eqn:ccbc-constr}. In a
data-driven context, it has been shown that as the data size \(N\) increases,
the confidence of \(\Probt{} \in \mathcal{P}\) increases to one exponentially
fast (see, e.g.,~\cite{fournier-2014-rate-conver}
and~\cite{gibbs-2002-choos-bound}).

We remark that the Wasserstein distance is equivalent to various other
discrepancy measures, including Hellinger distance and Prokhorov metric. For
example, one can bound the Hellinger distance between \(\Prob\) and
\(\hProb{}_{\tilde{\xi}}\), from both above and below, by
\(d_W(\Prob, \hProb{}_{\tilde{\xi}})\) (see, e.g.,~\cite{gibbs-2002-choos-bound}).
This implies that we can employ constraint~\eqref{eqn:drccbc-constr} to
approximate distributionally robust chance constraints defined under other
discrepancy measures.


\subsection{Literature Review}%
\label{sec:intro-review}

Chance-constrained programming (\CCP{}) arises in a wide range of applications
including power systems~\cite{zhang-2010-probab-analy},
transportation~\cite{chen-2004-stoch-trans}, facility
location~\cite{miranda-2006-simul-inven}, and wireless
communication~\cite{hsiung-2005-power}. Dating back
to~\cite{charnes-1958-cost-horiz, miller-1965-chanc-const,
  prekopa-1970-on-proba}, \CCP{}s are considered very challenging to solve
because
\begin{enumerate*}
\item[(i)] checking the feasibility of a given solution \(x\) demands
multivariate integral, which is difficult to calculate, and
\item[(ii)] \CCP{} produces a nonconvex and even disconnected feasible region in general.
\end{enumerate*}
To mitigate this challenge, prior work has studied convex conservative
approximation~\cite{nemirovski-2007-convex-approx, ben-tal-2009-robus} and
sample average approximation
(\SAA{})~\cite{luedtke-2008-sampl-approx,calafiore-2004-uncer-convex-progr},
both of which can efficiently search for feasible solutions with guarantee.
Nevertheless, the former may remain challenging to solve if decision variables
are discrete (like in this paper), and the latter often demands the capability
of drawing as many samples from \(\Probt{}\) as needed. In addition,
distributionally robust chance-constrained models (\DRCC{}) have received
increasing attention in recent
years~\cite{ghaoui-2003-worst-case,calafiore-2006-distr-robus,vandenberghe-2007-gener-cheby,nemirovski-2007-convex-approx,zymler-2011-distr-robus,xu-2012-optim-under,jiang-2015-data-driven,duan-2018-distr-robus,chen-2018-data-driven,yang-2018-wasser-distr,li-2019-ambig-risk}.
There are many successful developments on the tractability of single and joint
chance constraints with moment ambiguity sets, which characterize
\(\mathcal{P}\) based on moment information of
\(\Probt{}\)~\cite{calafiore-2006-distr-robus,zymler-2011-distr-robus,yang-2014-distr-robus,hanasusanto-2015-distr-robus,xie-2016-deter-refor,hanasusanto-2017-ambig-joint,li-2019-ambig-risk}.
Nevertheless, moment ambiguity sets become more conservative than their
counterparts based on discrepancy measures (e.g., a Wasserstein ambiguity set)
when more data samples are available. On the other hand, \DRCC{} with a
Wasserstein ambiguity set is not polynomially solvable in
general~\cite{xie-2020-bicrit-approx}. Prior work has developed exact
reformulations and valid
inequalities~\cite{xie-2019-distr-robus,ji-2018-data-driven,chen-2018-data-driven,ho-nguyen-2020-distr-robus,ho-nguyen-2020-stron-formul}
for solving this model. For
example,~\cite{xie-2019-distr-robus,ji-2018-data-driven,chen-2018-data-driven}
derived mixed-integer linear or conic reformulations. When decision variables
are purely binary,~\cite{xie-2019-distr-robus} exploited the submodularity of
the reformulation to produce extended polymatroid inequalities. In
addition,~\cite{ji-2018-data-driven} derived precedence valid inequalities among
the binary variables they employed to indicate constraint satisfaction in each
scenario.~\cite{ho-nguyen-2020-distr-robus}
and~\cite{ho-nguyen-2020-stron-formul} focused on \DRCC{} with right-hand-side
(\RHS{}) and left-hand-side (\LHS{}) Wasserstein ambiguity, respectively.
By exploring the connection between \DRCC{} and the \SAA{}
  formulation of \CCP{}, they employed the mixing scheme to produce
  valid inequalities (see,
  e.g,~\cite{guenluek-2001-mixin-mixed,kuecuekyavuz-2012-mixin-sets,luedtke-2014-branc-and}).
Nevertheless, prior work has paid less attention to problems with binary
decision variables and discrete uncertainty. In this paper, we study joint
\DRCC{} with \LHS{} Wasserstein ambiguity in generalized set
covering, where both decision and random variables are purely binary.

Chance-constrained integer
programs~\cite{beraldi-2009-exact-approac,song-2013-branc-and,luedtke-2014-branc-and,song-2014-chanc-const,wu-2019-probab-partial,wang-2019-solut-approac}
are stochastic variants of combinatorial optimization problems when uncertainty
arises. In~\cite{luedtke-2014-branc-and}, the authors developed a general
decomposition framework for solving chance-constrained programs and derived
strong mixed valid inequalities by combining ``base'' inequalities from each
scenario. In addition,~\cite{song-2014-chanc-const} proposed an efficient
coefficient strengthening method and lifted probabilistic cover inequalities for
chance-constrained bin packing problems.
Recently,~\cite{wang-2019-solut-approac} studied a chance-constrained assignment
problem and its \DRCC{} variant, for which they derived strong lifted cover
inequalities with efficient separation heuristics. In this work, we exploit the
special structures of set covering and Wasserstein ambiguity to derive two
families of valid inequalities. The first family produces the convex hull of the
hypograph of a ``shifted'' submodular function. To the best of our knowledge,
such a convex hull was not discovered in prior work. The second family
mixes inequalities across multiple scenarios, bearing a resemblance to the
derivation in~\cite{luedtke-2014-branc-and}. Nevertheless, we take advantage of
the binary nature of our decision and random variables and strengthen the mixed
inequalities further via lifting.

Chance-constrained set covering models fall into two main categories. In the
first category, uncertainty arises on the \RHS{}~\cite{beraldi-2002-probab-set,
  saxena-2010-mip-refor}. In~\cite{beraldi-2002-probab-set}, the authors
developed a specialized branch-and-bound algorithm based on the enumeration of
\(p\)-efficient points, which were initially introduced
by~\cite{prekopa-1990-dual-method}. Later,~\cite{saxena-2010-mip-refor}
simplified the enumeration approach and derived polarity cuts to improve the
computational performance. The second category models \LHS{}
uncertainty~\cite{fischetti-2012-cuttin-plane,ahmed-2013-probab-set,wu-2019-probab-partial},
as in this paper. For example,~\cite{fischetti-2012-cuttin-plane} studied single
chance constraints, where all components of the Bernoulli random vector
\(\txi_i\) are independent, and developed efficient cutting plane approaches. In
addition,~\cite{wu-2019-probab-partial} proposed an exact approach for solving
\CCP{}s when there exists an oracle to retrieve the probability of any events
under \(\Probt{}\). They demonstrated this approach on chance-constrained
partial set covering problems with either independence or linear threshold
assumptions. To the best of our knowledge,~\cite{ahmed-2013-probab-set} is the
only prior work on \DRCC{} for set covering, but they studied single chance
constraints under a moment ambiguity set. They derived a compact equivalent
reformulation and exploited its supermodularity to derive strong valid
inequalities. Different from~\cite{ahmed-2013-probab-set}, we study joint
\DRCC{} under Wasserstein ambiguity. In addition, the supermodularity of their
reformulation stems from the correlation among \(\txi_{ij}\)'s, while our
``shifted'' submodular function is a result of the generalized set covering and the Wasserstein ambiguity.

\subsection{Contributions}%
\label{sec:intro-contrib}

We derive exact reformulations and two families of valid inequalities for (\DRCCbC{}). Our main contributions include
\begin{enumerate}
\item We establish the NP-hardness of (\DRCCbC{}) and show that (\DRCCbC{}) admits a
    deterministic two-stage reformulation, which can be solved efficiently by
    decomposition algorithms.
\item We derive a complete description of the convex hull of a basic
    mixed-integer set, which stems from the hypograph of a ``shifted''
    submodular function and arises in each scenario of the two-stage
    reformulation.
\item We derive a family of cross-scenario inequalities by lifting the mixed
    valid inequalities obtained from multiple scenarios.
\item We conduct extensive numerical experiments to demonstrate (i) the
      out-of-sample performance of (\DRCCbC{}) and (ii) the effectiveness of our
      two-stage reformulation and valid inequalities.
\end{enumerate}

The rest of this paper is organized as follows. Section~\ref{sec:reform} shows
the NP-hardness and develops an exact two-stage reformulation for (\DRCCbC{}).
Section~\ref{sec:valid_ineqs} derives the single- and cross-scenario valid
inequalities. \edits{Section~\ref{sec:extension-knaps} extends the reformulation and valid inequalities to a distributionally robust knapsack chance constraint.} Finally, Section~\ref{sec:exps} demonstrates the effectiveness of
the proposed model and solution approaches.

\textit{Notation:} \(\mathbb{Z}_+\) denotes the set of nonnegative integers. For
integers \(m\) and \(n\), \([n] := \set{1, \ldots, n}\), \(\ones_{m\times n}\)
denotes an \(m\times n\) matrix of all ones, \(\ones\) denotes a vector of
all ones with suitable dimension, and \(\basevec_m\) denotes the \(m\)th standard basis vector with suitable dimension. For \(x \in \reals\),
\(\posp{x} := \max \set{x, 0}\).
For set \(E\), \(\abs{E}\) denotes its cardinality, \(\Co(E)\) denotes its convex hull, and the indicator function
\(\Ind{x \in E} := 1\) if \( x \in E\) and \(\Ind{x \in E} := 0\) otherwise. 


\section{Two-Stage Reformulation}%
\label{sec:reform}
We derive a reformulation of (\DRCCbC{}) based on the conditional value-at-risk (\CVaR{}) in Section \ref{sec:reform-cvar} and show its NP-hardness in Section \ref{sec:reform-nphard}. Then, in Section \ref{sec:reform-two-stage}, we further recast it as a two-stage stochastic program.

\subsection{Conditional Value-at-Risk Reformulation}%
\label{sec:reform-cvar}


We recall the definitions of value-at-risk
(\VaR{}) and \CVaR{}~\cite{rockafellar-1999-optim-condit}. For a random variable
\(\tilde{X}\) following its induced probability distribution
\(\Prob{}_{\tilde{X}}\), the \((1 - \epsilon)\)-\VaR{} of \(\tilde{X}\) is
defined as
\begin{gather*}
\VaR_{1 - \epsilon} (\tilde{X}) := \inf
\Set{x : \Prob{}_{\tilde{X}} \set{ \tilde{X} \leq x }\geq 1 - \epsilon},
\end{gather*}
and its \((1 - \epsilon)\)-\CVaR{} is defined as:
\begin{align}
\CVaR_{1 - \epsilon} (\tilde{X}) =
  \min_{\gamma} \Set{\gamma + \frac{1}{\epsilon} \Expect{} \posp{\tilde{X} - \gamma} }.
  \label{eqn:cvar-def}
\end{align}

Previously,~\cite{xie-2019-distr-robus} derived a \CVaR{} reformulation for \DRCC{} when \(\txi\) is supported in a normed vector space. Since our \(\txi\) is binary-valued, we adapt the framework of~\cite{xie-2019-distr-robus} to obtain a slightly different reformulation, which also admits a \CVaR{} interpretation.

\begin{proposition}[Adapted from Theorem~\(1\) in~\cite{xie-2019-distr-robus}]%
\label{prop:reform-mono}
Let \(Z\) represent the feasible region produced by constraint \eqref{eqn:drccbc-constr} in (\DRCCbC{}). Then, it holds that
\begin{align}
  Z = \Set{x \in \binaries^n \colon
  \begin{aligned}
    & \exists \ \gamma \in \reals_+, z \in \reals^N_-:\\
    & \delta - \gamma \epsilon \leq \frac{1}{N} \sum_{j \in [N]} z_j, \\
    & z_j + \gamma \leq \min_{i \in [I]} \left( \posp{x^{\top} \hat{\xi}^j_i - v_i + 1} \right)^{1/p},
    \forall j \in [N]
  \end{aligned}}.\label{eqn:reform-mono}
\end{align}
In addition, \(Z\) admits the following \CVaR{} interpretation:
\begin{gather}
Z = \Set{x \in \binaries{}^{n}: \frac{\delta}{\epsilon} + \cvar{}_{1 - \epsilon}
  \left[ - g(x, \hat{\xi} ) \right] \leq 0 }, \label{eqn:reform-cvar}  \\
  \mathrm{where} \quad g(x, \xi) = \min_{i \in [I]} \left( \posp{x^{\top} \xi_i - v_i + 1} \right)^{1/p}. \nonumber{}
\end{gather}
\end{proposition}
%

\begin{proof}[Proof (see Theorem~\(1\) in~\cite{xie-2019-distr-robus}
and Theorem~\(3\) in~\cite{blanchet-2017-quant-distr})] We first
rewrite constraint~\eqref{eqn:drccbc-constr} using complement:
\begin{align*}
  \sup_{\Prob{} \in \mathcal{P}} \Prob{} \Set{\exists \, i \in [I]: \, x^{\top} \txi_i \leq v_i - 1} \leq \epsilon.
\end{align*}
\editsII{By Theorem~\(3\) in~\cite{blanchet-2017-quant-distr}, we expand the above supremum to obtain
\begin{align*}
  \sup_{\Prob \in \mathcal{P}} \Prob \left\{ \exists \, i \in [I] \colon x^{\top} \tilde{\xi}_i \leq v_i - 1 \right\}
  = \min_{\lambda \geq 0} \Set{ \lambda \delta + \frac{1}{N} \sum_{j \in [N]} \Big(1 - \lambda \cdot g(x, \hat{\xi}^j)\Big)^+ },
\end{align*}
where
\begin{align*}
  g(x, \hat{\xi}^j)
  & := \inf_{\xi \in \Xi} \Set{ \norm{\xi - \hat{\xi}^j}_p \colon \exists \, i \in [I] ~\st{}~ x^{\top} \xi_i \leq v_i - 1 } \\
  & = \min_{i \in [I]} \inf_{\xi \in \Xi} \Set{ \norm{\xi - \hat{\xi}^j}_p \colon x^{\top} \xi_i \leq v_i - 1 }.
\end{align*}
Then, the \textbf{CVaR} interpretation  follows from Corollary~\(1\) in~\cite{xie-2019-distr-robus}. In what follows, we recast \(g(x, \hat{\xi}^j)\). To this end,} let \(\mathcal{T}\) be the
index set of \(x\), i.e., \(\mathcal{T} := \{k \in [n]: x_k = 1\}\). For any
\(i \in [I]\), we have
\begin{align*}
  \min_{x^{\top} \xi_i \leq v_i - 1} \norm{\xi - \hat{\xi}^j}_p
   & = \min_{x^{\top} \xi_i \leq v_i - 1} \left(
   \sum_{k = 1}^n \abs[\Big]{\xi_{ik} - \hat{\xi}^j_{ik}}^p \right)^{1/p} \\
  & = \min_{x^{\top} \xi_i \leq v_i - 1} \left(
    \sum_{k \not \in \mathcal{T}} \abs[\Big]{\xi_{ik} - \hat{\xi}^j_{ik}}^p +
    \sum_{k \in \mathcal{T}} \abs[\Big]{\xi_{ik} - \hat{\xi}^j_{ik}}^p \right)^{1/p} \\
  & = \min_{x^{\top} \xi_i \leq v_i - 1} \left(
    \sum_{k \in \mathcal{T}} \abs[\Big]{\xi_{ik} - \hat{\xi}^j_{ik}}^p \right)^{1/p} \\
  & = \left( \posp{x^{\top} \hat{\xi}^j_i - (v_i - 1)} \right)^{1/p}.
\end{align*}
\editsII{It follows that \(\displaystyle g(x, \hat{\xi}^j) = \min_{i \in [I]}\big((x^{\top}\hat{\xi}^j_i - v_i + 1)^+\big)^{1/p}\). This completes the proof.
}
\end{proof}
%
\edits{Although (\DRCCbC{}) admits the same \CVaR{} interpretation as
  in~\citet{xie-2019-distr-robus}, the reformulation $Z$ of
  (\DRCCbC{}) is inherently different from that
  of~\cite{xie-2019-distr-robus}. This is because the uncertain
  parameters $\tilde{\xi}$ are binary-valued in (\DRCCbC{}) while
  those in~\cite{xie-2019-distr-robus} are supported in a vector
  space. As a result, projecting a vector onto a half-space, which is
  a key step of reformulating Wasserstein chance constraints, differs
  in discrete and vector spaces. Specifically, in the proof of
  Proposition~\ref{prop:reform-mono},
\begin{equation*}
\min_{x^{\top} \xi_i \leq v_i - 1} \|\xi - \hat{\xi}^j\|_p \ = \
\left\{\begin{array}{ll}
\left((x^{\top}\hat{\xi}^j_i-v_i+1)^+\right)^{1/p} & \mbox{if $\Xi = \{0, 1\}^n$ (as in the proof of Proposition~\ref{prop:reform-mono}),} \\[0.5cm]
\frac{(x^{\top}\hat{\xi}^j_i - v_i + 1)^+}{\|x\|_{p/(p-1)}} & \mbox{if $\Xi = \mathbb{R}^n$ (as in~\cite{xie-2019-distr-robus})}.
\end{array}\right.
\end{equation*}
}
{\color{black}
\begin{remark}
The reformulation of \(Z\) in~\eqref{eqn:reform-mono} is of independent interest because it remains valid beyond the current setting of set covering. For example, (\DRCCbC) admits the same reformulation when \(x \in \binaries^n\), \(\Xi = \integers^{I\times n}_+\), and \(p = 1\). In addition, the reformulation remains valid when some target cannot be covered by certain elements. Formally, suppose that \(\xi\) is supported on the set
\[
\left\{\xi \in \binaries^{I \times n}: \xi_{ik} = 0 \quad \forall (i, k) \in S\right\}
\]
for a subset \(S \subseteq [I]\times [n]\), and accordingly \(\hat{\xi}^j_{ik} = 0\) for all \((i,k) \in S\). Then, the reformulation of \(Z\) in~\eqref{eqn:reform-mono} remains valid.
\end{remark}
}


\subsection{NP-hardness}%
\label{sec:reform-nphard}
%

%
%
We establish the NP-hardness of solving (\DRCCbC{}) in the following proposition.
\edits{
\begin{proposition}
\label{prop:reform-nphard}
(\DRCCbC{}) is NP-hard to solve for any given fixed risk level
\(\epsilon \in (0, 1)\).
\end{proposition}
\begin{proof}
We show that (\DRCCbC{}) is equivalent to the following problem when
\(N = 1\) and \(v_i = 1\) for all \(i \in [I]\):
\begin{align*}
  \min_{x} ~~
  & c^{\top} x, \tag{DRC'}\\
  \st ~~
  & \frac{\delta}{\epsilon}
    \leq \min_{i \in [I]} \left( x^{\top} \hat{\xi}^1_i \right)^{1/p}, \\
  & x \in \binaries^n.
\end{align*}
Let \((x_1, \gamma_1, z_1) \in Z\) be a feasible solution to (\DRCCbC{}), then
\(x_1\) is feasible to (\DRCCbC{}') because
\begin{align*}
  \frac{\delta}{\epsilon}
  \leq \frac{z_1}{\epsilon} + \gamma_1 \leq z_1 + \gamma_1
  \leq \min_{i \in [I]} \left( x^{\top} \hat{\xi}^1_i \right)^{1/p}.
\end{align*}
On the other hand, if \(x_2\) is feasible to (\DRCCbC{}'), then
together with \(\gamma_2 := \delta / \epsilon, z_2 := 0\) they are feasible to (\DRCCbC{}) because
\begin{align*}
  \delta - \gamma_2 \epsilon
  & = 0 \leq z_2, \\
  z_2 + \gamma_2
  & = \delta / \epsilon
    \leq \min_{i \in [I]} \left( x^{\top}_2 \hat{\xi}^1_i \right)^{1/p}.
\end{align*}
Therefore, (\DRCCbC{}) and (\DRCCbC{}') are equivalent. Since restricting
\(\delta\) to be \(\epsilon\) in (\DRCCbC{}') recovers the set cover problem, we
conclude that (\DRCCbC{}) is NP-hard to solve.
\end{proof}
One may be tempted to obtain the convex hull of the following mixed-integer set $\mathcal{Q}$ arising from the reformulation~\eqref{eqn:reform-mono},
\begin{align*}
  \mathcal{Q} := \Set{
  (\theta, x) \in \reals \times \binaries^n \colon
  \theta \leq \min_{i \in [I]} \left( x^{\top} \xi_i \right)^{1/p}}
\end{align*}
for given $\xi_i \in \binaries{^n}$, $i \in [I]$. Nonetheless, we show that optimizing a linear function over \(\mathcal{Q}\) is
NP-hard as well. The proof of this proposition is given in Appendix~\ref{appendix:np-hard-sep-pf}.
\begin{proposition}\label{reform:np-hard-sep}
The following problem is NP-hard to solve for any given fixed \(p \in \reals, p \geq 1\):
\begin{align}
  \min_{x, \theta} ~~
  & c^\top x - \theta^{1/p} \label{eqn:reform-nphard} \\
  \st{} ~~
  & \theta \leq x^{\top} \xi_i, \forall i \in [I], \nonumber{} \\
  & \theta \in \reals, x \in \binaries^n. \nonumber{}
\end{align}
\end{proposition}
}

\subsection{Two-Stage Reformulation}%
\label{sec:reform-two-stage}
The reformulation \eqref{eqn:reform-mono} of \(Z\) is potentially
challenging to optimize over, particularly because function
\(g(x, \xi)\) \editsII{defined in Proposition~\ref{prop:reform-mono}} is non-concave in \(x\).
\edits{Although
  one can linearize $g$ with the help of auxiliary binary variables
  and big-M coefficients, the formulation thus obtained is
  computationally ineffective. \editsII{Another linearized formulation without big-M coefficients can be obtained by following~\cite{xie-2019-distr-robus} if \(g(x, \xi)\) is supermodular in \(x\). Unfortunately, this is not the case even when \(I = 1\) (see Section~\ref{sec:valid_ineqs_single}).} As an alternative, we exploit the (hidden) submodularity of $g$ to obtain a two-stage reformulation without
  additional binary variables or big-M coefficients.} We review key
concepts of submodularity in Section \ref{sec:reform-poly} and present
the reformulation in Section \ref{sec:reform-linearize}.

\subsubsection{Polyhedral Results for Submodular Functions}
\label{sec:reform-poly}

\begin{definition}
A function \(\phi \colon 2^{[n]} \to \reals\) is submodular if for any \(\mathcal{R}, \mathcal{T} \subseteq [n]\), we have
\begin{align*}
  \phi(\mathcal{R}) + \phi(\mathcal{T})
  \geq \phi(\mathcal{R} \cup \mathcal{T}) + \phi(\mathcal{R} \cap \mathcal{T}).
\end{align*}
In addition, \(\phi\) is supermodular if \(-\phi\) is submodular.\qed
\end{definition}
For ease of exposition, we use \(\phi(\mathcal{T})\) and \(\phi(x_{\mathcal{T}})\) interchangeably, where \(x_{\mathcal{T}} \in \binaries^n\) is the indicating vector of \(\mathcal{T}\) such that \(x_k = 1\) if and only if \(k \in \mathcal{T}\). An example of submodular function follows.
\begin{lemma}[\cite{topkis-1978-minim-submod,xie-2019-distr-robus}]%
\label{lem:reform-neg-posp-submod}
For fixed \(\alpha \in \reals^n_+, \alpha_o \in \reals\) function
\(\phi(x) := \min\{-\alpha^{\top} x + \alpha_o, 0\}\) is submodular.
\end{lemma}
Next, consider the epigraph \(\epi(\phi)\) of a submodular function \(\phi\):
\begin{align*}
  \epi(\phi) := \Set{
  (\theta, x) \in \reals \times \binaries^n \colon \phi(x) \leq \theta}.
\end{align*}
Then, the convex hull of \(\epi(\phi)\) is fully described by the extended polymatroid inequalities (EPI)~\cite{schrijver-2003-combin}, i.e.,
\begin{align*}
  \Co{}(\epi(\phi))
  := \Set{(\theta, x) \in \reals \times [0, 1]^n:
  \phi(\varnothing) +
  \sum_{k = 1}^n \left[ \phi(\mathcal{T}_k) - \phi(\mathcal{T}_{k - 1}) \right] x_{\sigma_k}
  \leq \theta, \forall \sigma \in \Sigma},
\end{align*}
where \(\Sigma\) denotes all permutations of \([n]\) and
\(\mathcal{T}_k := \Set{\sigma_1, \ldots, \sigma_k}, \mathcal{T}_0 := \varnothing\).
The separation of \(\Co{}(\epi(\phi))\) is very efficient even though it
involves \(n!\) number of constraints.
\begin{theorem}[%
\edits{Proposition~\(1\) in~\cite{atamtuerk-2008-polym-mean}, %
Theorem~\(44.3\) in~\cite{schrijver-2003-combin}, %
Section~\(3\) in~\cite{lovasz-1983-submod-funct-convex}}
]\label{thm:reform-epi-sep}
If \((\hat{\theta}, \hat{x}) \not \in \Co{}(\epi(\phi))\), then it violates constraint
\begin{align*}
  \phi(\varnothing) + \sum_{k = 1}^n \left[
  \phi(\mathcal{T}_k) - \phi(\mathcal{T}_{k - 1}) \right] x_{\sigma_k} \leq \theta,
\end{align*}
where \(\sigma \in \Sigma\) is the permutation such that \(\hat{x}_{\sigma_1} \geq \hat{x}_{\sigma_2} \geq \cdots \geq \hat{x}_{\sigma_n}\).
\end{theorem}

\subsubsection{Reformulation} \label{sec:reform-linearize} We first
exchange the order of applying \(\posp{\cdot}\) and \((\cdot)^{1/p}\)
in the definition of \(g\). Since \(x^{\top}\hat{\xi}^j - v_i + 1\)
can be negative, we extend the domain of function \((\cdot)^{1/p}\) in
the following.
\begin{proposition}
  \label{prop:reform-domain-extension}
  For \(z \in \integers{}\), define
  \( \bar{f}(z) := z^{1/p} \cdot \Ind{z \geq 0} + z \cdot \Ind{z \leq -1} \). Then,
  \begin{align*}
    g(x, \xi) \equiv \min_{i \in [I]} \left( \posp{x^{\top} \hat{\xi}^j_i - v_i + 1} \right)^{1/p}
    = \posp{\bar{f}\left(\min_{i \in [I]} \left\{x^{\top}\hat{\xi}^j_i - v_i + 1\right\}\right)}.
  \end{align*}
\end{proposition}
\begin{proof}
  For \(x \in \binaries^n\), we have
  \begin{align*}
    \min_{i \in [I]} \left( \posp{x^{\top} \hat{\xi}^j_i - v_i + 1} \right)^{1/p}
    & = \left( \min_{i \in [I]} \posp{x^{\top} \hat{\xi}^j_i - v_i + 1}\right)^{1/p} \\
    & = \left( \posp{\min_{i \in [I]} \left\{x^{\top} \hat{\xi}^j_i - v_i + 1\right\}} \right)^{1/p} \\
    & = \posp{\bar{f} \left( \min_{i \in [I]} \left\{x^{\top}\hat{\xi}^j_i - v_i + 1\right\} \right)},
  \end{align*}
where the first equality is because \((\cdot)^{1/p}\) is monotone, the second equality is because \((\cdot)^+\) is monotone, and the third equality is because \(\bar{f}(z) \geq 0\) if and only if \(z \geq 0\).
\end{proof}

{\color{black}Although \(\bar{f}\) is defined on \(\integers\), the conclusion of Proposition~\ref{prop:reform-domain-extension} holds even when \(x^{\top}\hat{\xi}^j_i\) takes a fractional value.} Next, we linearize the nonlinear function \(\bar{f}\). \editsII{In what follows, Propositions~\ref{prop:reform-nonlinear}--\ref{prop:reform-closed-form-sol} utilize the fact that \(x^{\top} \hat{\xi}^j_i\) is an integer.}
\begin{proposition}%
\label{prop:reform-nonlinear}
For any \(x \in \binaries^n\), it holds that
\begin{align*}
  \bar{f} \left( \min_{i \in [I]} \left\{x^{\top} \hat{\xi}^j_i - v_i + 1\right\} \right) = \max_y ~~
  & c^{\top}_p y + \bar{f}(1 - v_m)\\
  \st{} ~~
  & \ones{}^{\top} y \leq \min_{i \in [I]} \left\{x^{\top} \hat{\xi}^j_i - v_i + v_m\right\}, \\
  & y \in [0, 1]^{n},
\end{align*}
where \(\displaystyle v_m := \max \set{v_i \colon i \in [I]}\) and \(c_p := [\bar{f}(2 - v_m) - \bar{f}(1 - v_m), \bar{f}(3 - v_m) - \bar{f}(2 - v_m),
\ldots, \bar{f}(n + 1 - v_m) - \bar{f}(n - v_m)]^{\top}\).
\end{proposition}

\begin{proof}
For fixed \(x \in \binaries^n\), define
\(\bar{v} := \min \set{x^{\top} \hat{\xi}^j_i - v_i + 1 \colon i \in [I]}\).
Since the above linear program is a continuous knapsack problem with an integer knapsack capacity and \(c_p \geq 0\), there exists an optimal solution \(y^* \in \binaries^n\) such that  \(\ones{}^{\top} y^* = \bar{v} + v_m - 1\).
Also note that vector \(c_p\) has non-increasing entries because \(\bar{f}\) is
a concave increasing function by construction. Hence, without loss of optimality, the first \(\bar{v} + v_m - 1\) entries of \(y^*\) equal one and the remaining entries equal zero. This yields an optimal objective value
\begin{align*}
  c^{\top}_p y^* & = \bar{f}(2 - v_m) - \bar{f}(1 - v_m) + \bar{f}(3 - v_m) - f(2 - v_m) +
  \cdots + \bar{f}(\bar{v}) - \bar{f}(\bar{v} - 1) + \bar{f}(1 - v_m) \\
  & = \bar{f}(\bar{v}) = \bar{f}\left(\min_{i \in [I]}\left\{ x^{\top} \hat{\xi}^j_i - v_i + 1\right\}\right).
\end{align*}
\end{proof}

We are now ready to present the main result of this section.
\begin{proposition}\label{prop:reform-two-stage}
(\DRCCbC{}) admits the following reformulation:
\begin{align}
(\mathrm{\bf MP}) \quad \min_{x, \gamma, z} ~~
& c^{\top} x \nonumber{}\\
\st{} ~~
& \delta - \gamma \epsilon \leq \frac{1}{N} \sum_{j \in [N]} z_{j},  \nonumber{}\\
& - z_j - \gamma \geq \min \Set{Q_j (\mu, x), 0}, \quad \forall \mu \in \mathcal{M}, \forall j \in [N],  \label{eqn:reform-submod}\\
& x \in \binaries^{n}, \gamma \in \reals_{+}, z \in \reals^{N}_{-}, \nonumber{}
\end{align}
where \(Q_j (\mu, x) := \sum_{i \in [I]} \mu_{1i} \left( x^{\top} \hat{\xi}^j_i - v_i + v_m \right) + \ones^{\top} \mu_2 - \bar{f}(1 - v_m)\) and \(\mathcal{M} := \{(\mu_1, \mu_2) \in \reals{}^I_- \times \reals{}^n_-: \ \mu_1^{\top} \ones_{I \times n} + \mu_2^{\top} \leq -c^{\top}_p\}\).
\end{proposition}


\begin{proof}
By Propositions~\ref{prop:reform-mono} and~\ref{prop:reform-nonlinear},
(\DRCCbC{}) is equivalent to:
\begin{align}
\min_{x, \gamma, z} ~~
& c^{\top} x \nonumber{}\\
\st{} ~~
& \delta - \gamma \epsilon \leq \frac{1}{N} \sum_{j \in [N]} z_{j},  \nonumber{}\\
& - z_j - \gamma \geq \min \Set{Q'_j (x), 0}, \quad \forall j \in [N], \nonumber{}\\
& x \in \binaries^{n}, \gamma \in \reals_{+}, z \in \reals^{N}_{-}, \nonumber{}
\end{align}
where \(Q'_j(x)\) is defined as:
\begin{subequations}
\begin{flalign}
(\mathrm{\bf SP}_j) && Q'_j(x) :=
\min_{y_j} ~~
& -c^{\top}_p y - \bar{f}(1 - v_m) && \nonumber\\
&& \st{} ~~
& \ones{}^{\top} y_j \leq x^{\top} \hat{\xi}^{j}_{i} - v_i + v_m, \quad \forall i \in [I], && \label{eqn:reform-note-1} \\
&& & y_j \in [0, 1]^{n}. && \label{eqn:reform-note-2}
\end{flalign}
\end{subequations}
Taking dual of the above linear program yields:
\begin{flalign*}
  (\mathrm{\bf DSP}_j) && Q'_j(x) = \max_{\mu} ~~
  & Q_j(\mu, x) \equiv \sum_{i \in [I]} \mu_{1i} \left( x^{\top} \hat{\xi}^j_i - v_i + v_m \right) +
    \ones^{\top} \mu_2 - \bar{f}(1 - v_m) && \\
  && \st{} ~~
  & \mu \in \mathcal{M}, &&
\end{flalign*}
where dual variables \(\mu_1, \mu_2\) are associated with constraints \eqref{eqn:reform-note-1} and \eqref{eqn:reform-note-2}, respectively. Then, the reformulation follows by noting that
\begin{align}
  \min \Set{ \max_{\mu \in \mathcal{M}} Q_j(\mu, x), 0}
  = \max_{\mu \in \mathcal{M}} \min \Set{Q_j(\mu, x), 0}. \label{eqn:reform-note-3}
\end{align}
\end{proof}

\begin{remark}
By Lemma~\ref{lem:reform-neg-posp-submod}, the \RHS{} of constraints  \eqref{eqn:reform-submod} is submodular in \(x\). As a result, we can replace \eqref{eqn:reform-submod} with the EPI with respect to this submodular function. This implies that (\DRCCbC{}) admits a mixed-integer \emph{linear} reformulation.
\end{remark}
Proposition \ref{prop:reform-two-stage} suggests a Benders decomposition (\benders{}) algorithm of solving (\DRCCbC{}) in an iterative manner. In each iteration, \benders{} solves
\((\masterp{})\) with constraints~\eqref{eqn:reform-submod} replaced by a relaxation, and sends an optimal solution \((\hat{x}, \hat{\gamma}, \hat{z})\) to every \((\dsubp_j)\) to check feasibility. If feasible, then \((\hat{x}, \hat{\gamma}, \hat{z})\) is optimal to (\DRCCbC{}); otherwise, we
identify a \(\hat{\mu} \in \mathcal{M}\) such that
\(-\hat{z}_j - \hat{\gamma} < \min \{Q_j(\hat{\mu}, \hat{x}), 0\}\). In the latter case, we obtain a violated EPI with respect to \(\min \{Q_j(\hat{\mu}, x), 0\}\) and add it as a Benders feasibility cut back to \((\masterp{})\). Naturally, the efficacy of this \benders{} algorithm depends on that of solving \((\dsubp{}_j)\). The next proposition gives a closed-form solution to \((\dsubp{}_j)\).
\begin{proposition}\label{prop:reform-closed-form-sol}
For fixed \(x\), an optimal solution \((\hat{\mu}_1, \hat{\mu}_2)\) to \((\dsubp{}_j)\) satisfies
\begin{gather}
\hat{\mu}_1 = - \basevec_{i^*} (c_p)_{v^*}, \nonumber{} \\
(\hat{\mu}_2)_k = \min \Set{-(c_p)_k - \ones^{\top} \hat{\mu}_1, 0}, \quad \forall k \in [n],
\label{eqn:reform-dsub-2}
\end{gather}
where \(\displaystyle v_m = \max \set{v_i \colon i \in [I]}\), 
\(\displaystyle i^* \in \argmin\{ x^{\top} \hat{\xi}^j_i - v_i + v_m: i \in [I]\}\), \(\bar{v} := \min \set{x^{\top} \hat{\xi}^j_i - v_i + v_m \colon i \in [I]}\), and \(v^* \in \set{\ceil{\bar{v}}, \floor{\bar{v}}}\).
\end{proposition}
\begin{proof}
To maximize \(Q_j(\mu, x)\), every component in \(\mu_2\) should attain its
upper bound at optimality; and hence~\eqref{eqn:reform-dsub-2} follows. \edits{Since
\begin{enumerate*}[label=(\roman*)]
\item \(\mu_1 \leq 0\),
\item \(x^{\top} \hat{\xi}^j_i - v_i + v_m \geq 0\) for all \(i \in [I]\), and
\item \(\sum_{k \in [n]} \posp{(c_p)_k + \ones^{\top}\mu_1}\) relies solely on
\(\ones^{\top}\mu_1\),
\end{enumerate*}
the support of \(\mu_1\) is a singleton at optimum, \ie{},
\(\mu_1 = s \cdot \mathbf{e}_{i^*}\) for a nonpositive scalar $s$. Accordingly, (\(\text{DSP}_j\)) reduces to a one-dimensional problem
\begin{align*}
  \max_{\mu \in \mathcal{M}} Q_j(\mu, x) = \max_{s \in \reals_-} \Set{
  s \cdot \bar{v} -
  \sum_{k \in [n]} \posp{(c_p)_k + s}} - \bar{f}(1 - v_m),
\end{align*}
and \(\ones^{\top} \mu_1 = s\). Since the entries of \(c_p\) are descending, a
maximizer to the above problem equals either \(-(c_p)_{\floor{\bar{v}}}\) or
\(-(c_p)_{\ceil{\bar{v}}}\).}
\end{proof}
Therefore, by Theorem~\ref{thm:reform-epi-sep} the Benders feasibility cut takes the following form:
\begin{align}
  & - z_j - \gamma
  \geq \phi^j(\varnothing) +
    \sum_{k = 1}^n \left[ \phi^j(\mathcal{T}_k) - \phi^j(\mathcal{T}_{k - 1}) \right] x_{\sigma_k},
    \label{eqn:reform-feasi-base}\\
  \mathrm{where} \quad & \phi^j(x) := \min\left\{\ones^{\top} \hat{\mu}_1 \left( x^{\top} \hat{\xi}^j_{i^*} - v_{i^*} + v_m \right) + \ones^{\top} \hat{\mu}_2 - \bar{f}(1 - v_m), \ 0\right\} \nonumber{} \\
  \mathrm{and} \quad & \phi^j(\mathcal{T}_k) - \phi^j(\mathcal{T}_{k - 1})
  =
  \begin{cases}
  \ones^{\top}\hat{\mu}_1 & \text{ if } \phi^j(\mathcal{T}_{k - 1}) < 0 \\
  \phi^j(\mathcal{T}_k) & \text{ if } \phi^j(\mathcal{T}_{k - 1}) = 0 \text{ and } \phi^j(\mathcal{T}_k) < 0 \\
  0 & \text{ if } \phi^j(\mathcal{T}_{k}) = 0.
  \end{cases} \nonumber{}
\end{align}
\editsII{
We compare our two-stage reformulation of (\DRCCbC{}) with that of~\cite{xie-2019-distr-robus} 
when decision variables are
binary. As pointed out after Proposition~\ref{prop:reform-mono}, our
setting is inherently different from that
of~\cite{xie-2019-distr-robus}, and as a result, the technique
Xie~\cite{xie-2019-distr-robus} used to exploit the supermodularity
does not apply to our case. To be concrete, suppose that \(\txi\) is supported in a vector space as in~\cite{xie-2019-distr-robus}. Then, applying Theorem~\(1\)
and Proposition~\(1\) of~\cite{xie-2019-distr-robus} to our set covering model yields the following reformulation of~(\DRCCbC{}):
\begin{align*}
  Z_0 = \Set{
  x \in \binaries^n \colon
  \begin{aligned}
  & \exists \, \gamma \in \reals_+, \nu \in \reals_{+}, z \in \reals^N_- \colon \\
  & \delta \nu - \gamma\epsilon \leq \frac{1}{N} \sum_{j \in [N]} z_j, \\
  & \norm{x}_{p/(p-1)} \leq \nu, \\
  & z_j + \gamma \leq \posp{x^{\top} \hat{\xi}^j_i - v_i}, \forall i \in [I], j \in [N] \\
  \end{aligned}}.
\end{align*}
In the last constraints of $Z_0$, the supermodularity of
$(x^{\top} \hat{\xi}^j_i - v_i)^+$ in $x$ appears naturally. In
contrast, in our case the term 
$\bigl((x^{\top}\hat{\xi}^j_i-v_i + 1)^+\bigr)^{1/p}$ in \eqref{eqn:reform-mono} is neither supermodular in $x$ nor in its
complements (\ie{}, replacing some entries $x_k$ with $1 - x_k$)
because the root function $(\cdot)^{1/p}$ is strictly concave for all
$p > 1$. To exploit supermodularity, we need to recast the root
function as a linear program (\ie{},
Proposition~\ref{prop:reform-nonlinear}) and then take dual.
}



\section{Valid Inequalities}%
\label{sec:valid_ineqs}
We derive two families of valid inequalities for the
reformulation~\eqref{eqn:reform-mono} of \(Z\).
Section~\ref{sec:valid_ineqs_single} provides inequalities based on a single
scenario \(j \in [N]\), Section~\ref{sec:proofs-of-propositions} consists of
detailed proofs for establishing a convex-hull result of these inequalities, and
Section~\ref{sec:valid_ineqs_cross} provides cross-scenario inequalities.

\subsection{Valid inequalities from a single scenario}%
\label{sec:valid_ineqs_single}
We focus on a nonlinear inequality
\(z_j + \gamma \leq ((x^{\top}\hat{\xi}^j_i - v_i + 1)^+)^{1/p}\) in the
reformulation~\eqref{eqn:reform-mono}, for a single scenario \(j \in [N]\) and a
single covering \(i \in [I]\). If \(v_i = 1\) then the \RHS{} of this inequality
becomes a submodular function in \(x\) because \(p \geq 1\) and
\(\hat{\xi}^j_i \geq 0\). For general \(v_i\), the \RHS{} can be viewed as the
submodular function being ``shifted'' by \((v_i - 1)\). However, after such a
shift the \RHS{} becomes neither submodular nor supermodular, preventing us from
using existing valid inequalities
(e.g.,~\cite{ahmed-2009-maxim-class,schrijver-2003-combin}). In this section, we
generalize this inequality and consider the following hypograph of a
shifted submodular function:
\begin{align*}
  \seti{} := \Set{(\theta, x) \in \reals \times \binaries^{n} :
  \theta \leq f\left(\posp{x^{\top}\xi - \beta}\right)},
\end{align*}
where \(f : \reals{}_+ \to \reals\) is a strictly concave increasing function
with \(f(0) = 0\), \(\xi \in \binaries^{n}\) is a given binary vector, and
\(\beta \in \integers{}_+\) is a given integer. Let \(\mathcal{Z}\) be the set
of indices where \(\xi\) takes \(1\). If \(\beta \geq \abs{\mathcal{Z}} - 1\)
then \(f((x^{\top}\xi - \beta)^+) \equiv f(1)(x^{\top}\xi - \beta)^+\) is
supermodular in \(x\). In this case, \(\Co(\seti)\) can be fully described by
the EPI. Hence, we assume \(1 \leq \beta \leq \abs{\mathcal{Z}} - 2\) without
loss of generality (\WLOG{}). Additionally, since \(\xi\) has at most \(\abs{\mathcal{Z}}\)
numbers of ones, all possible values for \((x^{\top}\xi - \beta)^+\) lie in the set
\[
L := \Bigl[\abs{\mathcal{Z}} - \beta\Bigr].
\]
We replace \(f\) with its inner piecewise linear approximation with integer breakpoints in $L$ in the definition of $\seti{}$. The following lemma formalizes this.
\editsII{%
\begin{lemma}
For \(\ell \in L\), let \(\varphi_{\ell}\) be a linear function
defined as \(\varphi_{\ell}(z):= a_{\ell} z + b_{\ell}\) where
\(a_{\ell} := f(\ell) - f(\ell - 1)\) and
\(b_{\ell} := f(\ell) - a_{\ell} \cdot \ell\). To avoid clutter, we also define \(a_{|\mathcal{Z}|-\beta+1} := 0\) and \(b_{|\mathcal{Z}|-\beta+1} := f(|\mathcal{Z}|-\beta)\). Then, we have
\(f(z) = \displaystyle\min_{\ell \in L} \set{\varphi_{\ell}(z)}\) for any
\(z \in L\). In particular,
\(f(\ell) = \varphi_{\ell}(\ell) = \varphi_{\ell + 1}(\ell)\) for all \(\ell \in L\).
\end{lemma}
\begin{proof}
It is straightforward that
\(f(\ell) = \varphi_{\ell}(\ell) = \varphi_{\ell + 1}(\ell)\) for all \(\ell \in L\). To show
that \(f(z) = \min_{\ell \in L}\set{\varphi_{\ell}(z)}\) for any
\(z \in L\), it is sufficient to argue that
\(\varphi_i(\ell) \geq \varphi_{\ell}(\ell)\) for every \(i \in
L\). When \(i < \ell\), \(\varphi_i(\ell) \geq \varphi_{\ell}(\ell)\)
is equivalent to
\(f(i) - f(i - 1) \geq (f(\ell) - f(i)) / (\ell - i)\), which holds
due to the concavity of \(f\). When \(i > \ell\),
\(\varphi_i(\ell) \geq \varphi_{\ell}(\ell)\) is equivalent to
\((f(i) - f(\ell))/(i - \ell) \geq f(i) - f(i - 1)\), which holds due
to the concavity of \(f\).
\end{proof}
We mention properties about \(a_{\ell}\) and \(b_{\ell}\) in the following remark.
\begin{remark}
For \(a_{\ell}\) and \(b_{\ell}\), it holds that
\begin{enumerate}[label=(\roman*)]
\item \(a_{\ell}\) are strictly positive and decreasing in $\ell$. This is because \(f\) is strictly concave and increasing.
\item \(b_{\ell} \geq 0\) for all \(\ell \in L\). This is because
  \[
  b_{\ell} = \sum_{k=1}^\ell \bigl(f(k) - f(k-1)\bigr) - \ell\cdot\bigl(f(\ell) - f(\ell-1)\bigr)
  = \sum_{k=1}^\ell \left(a_k - a_{\ell}\right) \geq 0,
  \]
  where the last inequality is because \(a_{\ell}\) are decreasing in \(\ell\).
\item \(b_{\ell} / a_{\ell}\) are strictly increasing in $\ell$. This is because
  \(a_{\ell}\) are strictly decreasing and \(b_{\ell}\) are increasing since
  \begin{align*}
  b_{\ell + 1} - b_{\ell}
  = \sum_{k=1}^{\ell+1} \left(a_k - a_{\ell}\right) - \sum_{k=1}^\ell \left(a_k - a_{\ell}\right) = \ell (a_{\ell} - a_{\ell+1}) > 0.
  \end{align*}
\end{enumerate}
\end{remark}
}
\edits{
In addition, for ease of presenting the valid inequalities we define
\begin{align*}
  L_1 & := \bigcup_{\ell \in L}\Set{(\ell, \rho) \colon \rho = \frac{b_{\ell}}{a_{\ell}}}, \\
  L_2 & := \bigcup_{\ell \in L}\Set{(\ell, \rho) \colon \rho \in [\beta] \cap
        \left(\frac{b_{\ell}}{a_{\ell}}, \frac{b_{\ell+1}}{a_{\ell+1}}\right)},
\end{align*}
}
and \(\hfcn_{(\ell, \rho)}: \reals_+ \rightarrow \mathbb{R}\),
\begin{align*}
  \hfcn_{(\ell, \rho)}(z)
  := \posp{\frac{f(\ell)}{(\ell + \rho)}\bigl( z - (\beta - \rho) \bigr)},
  \quad \forall (\ell, \rho) \in L_1 \cup L_2.
\end{align*}
\editsII{
\begin{remark}
  For \((\ell, \rho) \in L_1\), we can simplify \(\hfcn_{(\ell, \rho)}\) as
  \begin{align*}
    \hfcn_{(\ell, \rho)}(z)
    = \posp{a_{\ell} \left( z - \beta \right) + b_{\ell}}
    = \posp{\fcnl_{\ell} \left( z - \beta \right)},
  \end{align*}
  from which \(h_{(\ell, \rho)}(z)\) 
  coincides with \(f(\posp{z - \beta})\) for all \((\ell, \rho) \in L_1\) when 
  \(z \in \{\ell+\beta-1, \ell+\beta\}\). For ease of exposition, we
  use \(\posp{\fcnl_{\ell}}\) and \(\hfcn_{(\ell, \rho)}\)
  interchangeably if \((\ell, \rho) \in L_1\). For
  \((\ell, \rho) \in L_2\), \(h_{(\ell, \rho)}(z) = 0\) if
  \(z \leq \beta - \rho\) and \(h_{(\ell, \rho)}(z) > 0\) if
  \(z > \beta - \rho\). In particular,
  \(h_{(\ell, \rho)}(z) = f(\ell)\) when \(z = \ell + \beta\).
  \end{remark}
}
%
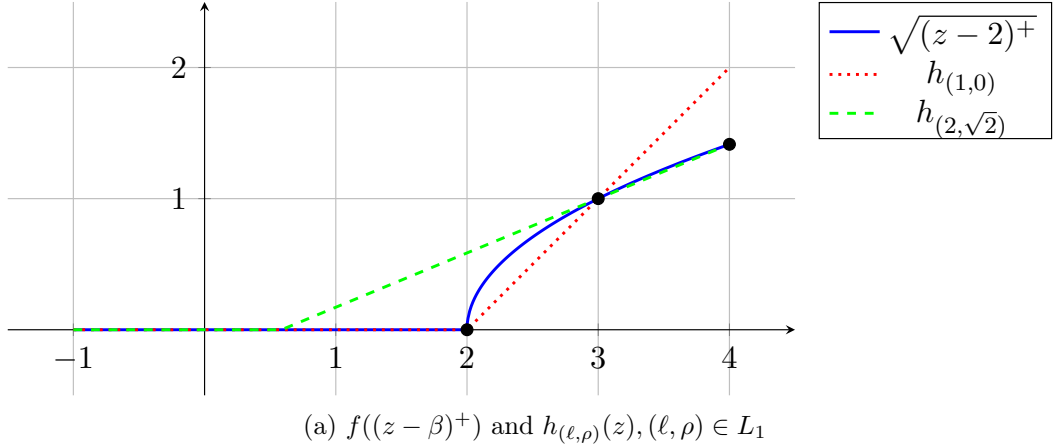
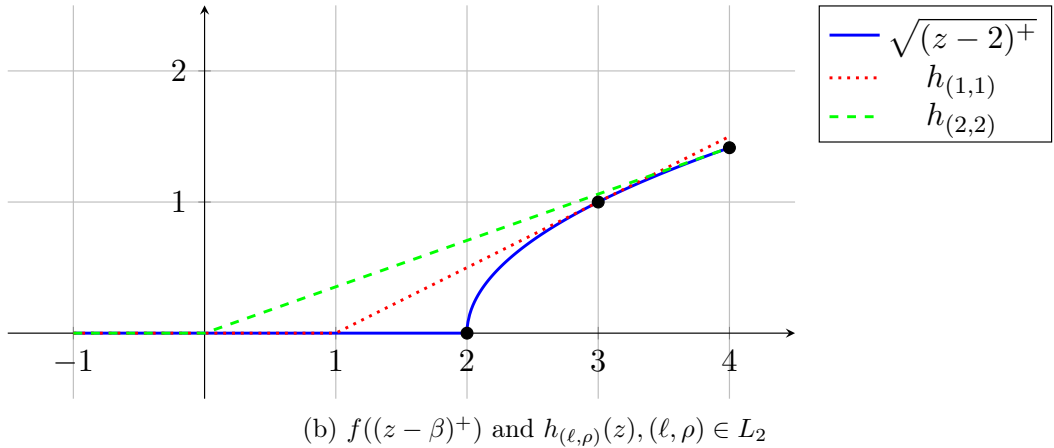
\begin{figure}[!htbp]
\begin{subfigure}{\textwidth}
  \centering
  \resizebox{0.85\textwidth}{!}{\begin{tikzpicture}
\begin{axis}[
  unit vector ratio*=1 1 1,
  width=11cm,
  grid=major,
  xtick distance=1,
  ymin=0, ymax=2,
  axis lines=middle,
  enlargelimits={abs=0.5},
  legend pos=outer north east,
]
\addplot[blue,samples=400,domain=-1:4,line width=1pt] {sqrt(max(x - 2, 0))};
\addlegendentry{$\sqrt{(z - 2)^+}$}
\addplot[red,dotted,samples=400,domain=-1:4,line width=1pt] {max(x - 2, 0)};
\addlegendentry{$h_{(1, 0)}$}
\addplot[green,dashed,samples=400,domain=-1:4,line width=1pt] {(sqrt(2) - 1) * max(x - (2 - sqrt(2)), 0)};
\addlegendentry{$h_{(2, \sqrt{2})}$}
\addplot[only marks,mark=*] table {
  x    y
  2    0
  3    1
  4    1.41421356
};
\end{axis}
\end{tikzpicture}}
  \caption{\(f((z-\beta)^+)\) and \(h_{(\ell, \rho)}(z), (\ell, \rho) \in L_1\)}%
  \label{fig:shifted-f-and-h1}
\end{subfigure}%
\par\bigskip 
\begin{subfigure}{\textwidth}
  \centering
  \resizebox{0.85\textwidth}{!}{\begin{tikzpicture}
\begin{axis}[
  unit vector ratio*=1 1 1,
  width=11cm,
  grid=major,
  xtick distance=1,
  ymin=0, ymax=2,
  axis lines=middle,
  enlargelimits={abs=0.5},
  legend pos=outer north east,
]
\coordinate (O) at (0,0);
\addplot[blue,samples=400,domain=-1:4,line width=1pt] {sqrt(max(x - 2, 0))};
\addlegendentry{$\sqrt{(z - 2)^+}$}
\addplot[red,dotted,samples=400,domain=-1:4,line width=1pt] {0.5 * max(x - 1, 0)};
\addlegendentry{$h_{(1, 1)}$}
\addplot[green,dashed,samples=400,domain=-1:4,line width=1pt] {sqrt(2) / 4 * max(x, 0)};
\addlegendentry{$h_{(2, 2)}$}
\addplot[only marks,mark=*] table {
  x    y
  2    0
  3    1
  4    1.41421356
};
\end{axis}
\end{tikzpicture}}
  \caption{\(f((z-\beta)^+)\) and \(h_{(\ell, \rho)}(z), (\ell, \rho) \in L_2\)}%
  \label{fig:shifted-f-and-h2}
\end{subfigure}
\caption{A numerical example of the shifted function \(f((z-\beta)^+)\) and \(h_{(\ell, \rho)}(z)\).} \label{fig:f-h}
\end{figure}
\begin{example} \label{exam:f-h}
Suppose that \(\abs{\mathcal{Z}} = 4\), \(\beta = 2\), and function \(f(z) = \sqrt{z}\). Then, \(L = [2]\), \(L_1 = \set{(1, 0), (2, \sqrt{2})}\), \(L_2 = \set{(1,1), (2,2)}\), and the shifted function \(f((z - \beta)^+) = \sqrt{(z-2)^+}\). Accordingly, for \((\ell, \rho) \in L_1\), the function \(h_{(\ell, \rho)}\) takes the following two forms:
\begin{align*}
    h_{(1,0)}(z) = (z-2)^+, \quad h_{(2, \sqrt{2})}(z) = \bigl(\sqrt{2}-1\bigr)\posp{z - (2-\sqrt{2})};
\end{align*}
and for \((\ell, \rho) \in L_2\), it takes the following two forms:
\begin{align*}
    h_{(1,1)}(z) = \frac{1}{2}(z-1)^+, \quad h_{(2, 2)}(z) = \frac{\sqrt{2}}{4}\posp{z}.
\end{align*}
We depict the shifted function \(f((z - \beta)^+)\) and \(h_{(\ell, \rho)}(z)\) in Fig.~\ref{fig:f-h}. In particular, from Fig.~\ref{fig:shifted-f-and-h1}, we observe that, as expected, each \(h_{(\ell, \rho)}(z)\), \((\ell, \rho) \in L_1\) coincides with \(f(\posp{z - \beta})\) when \(z \in \{\ell+\beta-1, \ell+\beta\}\). From Fig.~\ref{fig:shifted-f-and-h2}, we observe that each \(h_{(\ell, \rho)}(z)\), \((\ell, \rho) \in L_2\) has an integer \(z\)-intercept and it coincides with \(f(\posp{z - \beta})\) only at the point \((\ell+\beta, f(\ell))\).\qed
\end{example}
In what follows, we study the hypograph of \(\hfcn_{(\ell, \rho)}\) and its
relationship with \(\seti{}\). For \((\ell, \rho) \in L_1 \cup L_2\), define
\begin{align*}
  \seti_{(\ell, \rho)} :=
  \Set{ (\theta, x) \in \reals \times \binaries^{n} \colon
  \theta \leq \hfcn_{(\ell, \rho)}(x^{\top}\xi) }.
\end{align*}
First, we observe that \(\{\seti_{(\ell, \rho)}: (\ell, \rho) \in L_1\}\)
provides a reformulation for \(\seti\), while
\(\{\seti_{(\ell, \rho)}: (\ell, \rho) \in L_2\}\) provides a set of valid
inequalities.
\begin{lemma}%
  \label{lem:vis-frepr}
  For a given \(x \in \binaries^{n}\), we have
  \begin{align}
    f \left( \posp{ x^{\top} \xi - \beta} \right)
    = \min_{(\ell,\rho) \in L_1} \hfcn_{(\ell, \rho)}(x^{\top}\xi)
    \leq \min_{(\ell, \rho) \in L_2}\hfcn_{(\ell, \rho)}(x^{\top}\xi).
  \end{align}
\end{lemma}
\begin{proof}
Fix \(\hat{x} \in \binaries^n\). We first prove the equality. If
\(\hat{x}^{\top} \xi - \beta \leq 0\), then
\begin{align*}
  f \left(\posp{\hat{x}^{\top} \xi - \beta}\right)
  = 0
  = \fcnl_1\left(\posp{\hat{x}^{\top}\xi - \beta}\right)
  = \hfcn_{(1, 0)}(\hat{x}^{\top}\xi)
  \geq \min_{(\ell, \rho) \in L_1} \hfcn_{(\ell, \rho)}(\hat{x}^{\top}\xi),
\end{align*}
where the first equality is because \(f(0) = 0\), the second equality is because
\(\fcnl_1(z) = f(1)z\), and the third equality is because
\(\fcnl_1((\hat{x}^{\top}\xi - \beta)^+) = (\fcnl_1(\hat{x}^{\top}\xi - \beta))^+\).
If \(\hat{\ell} := \hat{x}^{\top}\xi - \beta \geq 1\), then
\begin{align*}
  f \left(\posp{\hat{x}^{\top}\xi - \beta}\right)
  = f(\hat{\ell})
  = a_{\hat{\ell}} \cdot \hat{\ell} + b_{\hat{\ell}}
  = \posp{\fcnl_{\hat{\ell}} (\hat{x}^{\top} \xi - \beta)}
  \geq \min_{(\ell, \rho) \in L_1} \hfcn_{(\ell, \rho)}(\hat{x}^{\top}\xi),
\end{align*}
where the last equality is because \(a_{\hat{\ell}}, b_{\hat{\ell}} \geq 0\).
Next we show that
\(\displaystyle f ( (x^{\top} \xi - \beta)^+ ) \leq \min_{(\ell,\rho) \in L_1} \hfcn_{(\ell, \rho)}(x^{\top}\xi)\).
Suppose that \((\hat{\ell}, \hat{\rho})\) is a minimizer for
\(\displaystyle \min_{(\ell, \rho) \in L_1} \hfcn_{(\ell, \rho)}(\hat{x}^{\top}\xi)\),
then
\begin{align*}
  \min_{(\ell, \rho) \in L_1} \hfcn_{(\ell, \rho)}(\hat{x}^{\top}\xi)
  & = \hfcn_{(\hat{\ell}, \hat{\rho})}(\hat{x}^{\top}\xi)
    = \posp{\fcnl_{\hat{\ell}}(\hat{x}^{\top} \xi - \beta)} \\
  & = \fcnl_{\hat{\ell}}\left( \posp{\hat{x}^{\top} \xi - \beta} \right)\Ind{\hat{x}^{\top}\xi - \beta \geq 0} +
    \posp{\fcnl_{\hat{\ell}}(\hat{x}^{\top} \xi - \beta)}
    \Ind{\hat{x}^{\top}\xi - \beta < 0} \\
  & \geq f \left(\posp{\hat{x}^{\top}\xi - \beta}\right),
\end{align*}
where the last equality is because \(a_{\hat{\ell}}, b_{\hat{\ell}} \geq 0\) and
the inequality is due to the piecewise linear representation of \(f\) and the
fact that \(f(0)=0\). Therefore,
\(f \left(\posp{x^{\top} \xi - \beta}\right) = \displaystyle \min_{(\ell, \rho) \in L_1} \hfcn_{(\ell, \rho)}(x^{\top}\xi)\).

It remains the show \(\displaystyle \min_{(\ell, \rho) \in L_1} \hfcn_{(\ell, \rho)}(x^{\top}\xi) \leq \hfcn_{{(\hat{\ell}, \hat{\rho})}} (x^{\top}\xi)\) for all \((\hat{\ell}, \hat{\rho}) \in L_2\). We note that
\begin{gather*}
  a_{\hat{\ell}+1} (\hat{\rho} + \hat{\ell})
  \leq a_{\hat{\ell}+1} \cdot \hat{\ell} + b_{\hat{\ell}+1}
  = f(\hat{\ell})
  = a_{\hat{\ell}} \cdot \hat{\ell} + b_{\hat{\ell}}
  \leq a_{\hat{\ell}} (\hat{\rho} + \hat{\ell}),
\end{gather*}
implying
\begin{gather*}
  a_{\hat{\ell}+1}
  \leq \frac{f(\hat{\ell})}{(\hat{\ell} + \hat{\rho})} \leq a_{\hat{\ell}}.
\end{gather*}
Therefore, we have
\begin{align*}
  \hfcn_{{(\hat{\ell}, \hat{\rho})}} (x^{\top}\xi)
  = & \posp{\frac{f(\hat{\ell})}{(\hat{\ell} + \hat{\rho})} \left( x^{\top} \xi - \beta + \hat{\rho} \right)}
      = \posp{f(\hat{\ell}) + \frac{f(\hat{\ell})}{(\hat{\ell} + \hat{\rho})} \left( x^{\top} \xi - \beta - \hat{\ell}\right)} \\
  = & \posp{f(\hat{\ell}) + \frac{f(\hat{\ell})}{(\hat{\ell} +
      \hat{\rho})} \left( x^{\top} \xi - \beta - \hat{\ell}\right)}
      \Ind{x^{\top} \xi - \beta \leq \hat{\ell}} + \\
    & \posp{f(\hat{\ell}) +
      \frac{f(\hat{\ell})}{(\hat{\ell} + \hat{\rho})} \left( x^{\top} \xi - \beta - \hat{\ell}\right)}
      \Ind{x^{\top} \xi - \beta \geq \hat{\ell} + 1} \\
  \geq & \posp{f(\hat{\ell}) + a_{\hat{\ell}} \left( x^{\top} \xi - \beta - \hat{\ell} \right)}
         \Ind{x^{\top} \xi - \beta \leq \hat{\ell}} + \\
    & \posp{f(\hat{\ell}) + a_{{\hat{\ell}+1}} \left( x^{\top} \xi - \beta - \hat{\ell} \right)}
      \Ind{x^{\top} \xi - \beta \geq \hat{\ell} + 1} \\
  = & \posp{\fcnl_{\hat{\ell}}\left( x^{\top}\xi - \beta \right)} \Ind{x^{\top} \xi - \beta \leq \hat{\ell}} +
      \posp{\fcnl_{{\hat{\ell}+1}}\left( x^{\top}\xi - \beta \right)} \Ind{x^{\top} \xi - \beta \geq \hat{\ell} + 1} \\
  \geq & \min_{(\ell, \rho) \in L_1} \hfcn_{(\ell, \rho)} \left( x^{\top}\xi \right).
\end{align*}
\end{proof}
The previous Lemma implies
\begin{align*}
  f \left( \posp{ x^{\top} \xi - \beta} \right)
  = \min_{(\ell, \rho) \in L_1} \hfcn_{(\ell, \rho)}(x^{\top}\xi)
  = \min_{(\ell, \rho) \in L_1 \cup L_2} \hfcn_{(\ell, \rho)}(x^{\top}\xi).
\end{align*}

\edits{Second, we notice that the EPIs arising from
  \(\set{{\color{black}\mathcal{X}_{(\ell, \rho)}} \colon (\ell, \rho) \in L_2}\) and
  \(\set{{\color{black}\mathcal{X}_{(\ell, \rho)}} \colon (\ell, \rho) \in L_1}\) are sufficient
  to describe \(\Co{}(\mathcal{X})\).}
\begin{theorem}%
\label{thm:vi-single-main}
It holds that
  \begin{align*}
    \Co(\seti) =  \bigcap_{(\ell, \rho) \in L_1 \cup L_2} \Co(\seti_{(\ell, \rho)}).
  \end{align*}
\end{theorem}
\edits{Before presenting the proof, we discuss the implication of Theorem~\ref{thm:vi-single-main} on the
  separation of \(\Co(\mathcal{X})\). To this end, by Theorem~\ref{thm:vi-single-main}, we only need to separate from
  \(\Co(\mathcal{X}_{(\ell, \rho)})\) for all \((\ell, \rho) \in L_1 \cup L_2\),
  where \(|L_1| = |\mathcal{Z}| - \beta\), \(|L_2| \leq \beta\) and so
  \(|L_1| + |L_2| \leq |\mathcal{Z}| \leq n\). In addition, for each
  \((\ell, \rho)\), \(h_{(\ell, \rho)}(x^{\top}\xi)\) is supermodular in \(x\)
  by Lemma~\ref{lem:reform-neg-posp-submod}. Hence, by
  Theorem~\ref{thm:reform-epi-sep}, we can separate from each 
  \(\Co(\mathcal{X}_{(\ell, \rho)})\) by running a sorting algorithm among the
  \(n\) entries of a given solution \(\hat{x}\). In addition, the resulting order is
  valid for all \((\ell, \rho)\). Therefore, we can separate from
  \(\Co(\mathcal{X})\) by running a \emph{single} sorting and checking the
  violation of at most \(n\) inequalities.}

\begin{proof}[Proof of Theorem~\ref{thm:vi-single-main}]
The proof is established on
Propositions~\ref{prop:vi-opt-conds-1},~\ref{prop:vi-opt-struct},
and~\ref{prop:vi-min-max}. We present their statements and relegate the proofs
to Section~\ref{sec:proofs-of-propositions}. We denote the intersection on the
\RHS{} of the claim as \(\seti^{\prime}\). By Lemma~\ref{lem:vis-frepr}, any
\((\hat{\theta}, \hat{x}) \in \seti\) satisfies
\begin{align*}
  \hat{\theta}
  \leq f \left( \posp{\hat{x}^{\top} \xi - \beta} \right)
  = \min_{(\ell, \rho) \in L_1 \cup L_2} \hfcn_{(\ell, \rho)}(\hat{x}^{\top}\xi).
\end{align*}
Then, \((\hat{\theta}, \hat{x}) \in \seti_{(\ell, \rho)}\) for all
\((\ell, \rho) \in L_{1} \cup L_2\), implying
\((\hat{\theta}, \hat{x}) \in \seti^{\prime}\) and
\(\Co(\seti) \subseteq \seti'\). It remains to prove
\(\seti' \subseteq \Co(\seti)\). Equivalently, we prove that
\((\hat{\theta}, \hat{x}) \notin \Co(\seti)\) implies
\((\hat{\theta}, \hat{x}) \notin \Co(\seti_{(\ell, \rho)})\) for some
\((\ell, \rho) \in L_1 \cup L_2\).

First, by the hyperplane separation theorem, there exists nonzero
\((\pi_0, \pi) \in \reals^{1+n}\) satisfying
\begin{align*}
\pi_0 \hat{\theta} + \pi^{\top} \hat{x}
> \pi_0 \theta + \pi^{\top} x, \quad \forall (\theta, x) \in \seti{}.
\end{align*}
If \(\hat{x} \notin [0, 1]^n\) then
\((\hat{\theta}, \hat{x}) \notin \Co(\seti_{(\ell, \rho)})\) because
\(\Co(\seti_{(\ell, \rho)}) \subseteq \mathbb{R}\times[0, 1]^n\). Hence, we can
assume \(\hat{x} \in [0, 1]^n\) \WLOG{} Since
\(\theta\) is unbounded below in \(\seti\), \(\pi_0\) should be nonnegative.
If \(\pi_0 = 0\) then \(\pi^{\top} \hat{x} > \pi^{\top} x\) for all
\(x \in \binaries^{n}\), a contradiction, because this nonzero \(\pi\) separates
\(\hat{x}\) from \([0, 1]^n\). Hence, \(\pi_0 > 0\). We can further assume
\(\pi_0 = 1\) \WLOG{} and arrive at
\begin{align*}
  \exists \pi \in \mathbb{R}^n: \quad \hat{\theta} + \pi^{\top} \hat{x}
  > \theta + \pi^{\top} x, \quad \forall (\theta, x) \in \seti{}.
\end{align*}
In other words, \(\hat{\theta}\) exceeds the optimal value of the following
formulation:
\begin{align}
  \min_{\pi \in \reals^{n}} \Set{
      \max \Set{
          \theta + \pi^{\top} x : (\theta, x) \in \seti{}
      } - \pi^{\top} \hat{x}
  }. \label{eqn:vi-main-theta-x}
\end{align}
\edits{The rest of the proof solves problem~\eqref{eqn:vi-main-theta-x}. The main idea can
be summarized as follows:
\begin{align*}
  & \min_{\pi \in \reals^{n}} \Set{
      \max_{(\theta, x) \in \seti{}} \Set{\theta + \pi^{\top} x} - \pi^{\top} \hat{x}} \\
  = & \min_{\pi \in \Pi^*} \Set{
    \max_{(\theta, x) \in \seti{}} \Set{\theta + \pi^{\top} x} -
    \pi^{\top} \hat{x}} \tag{by Propositions~\ref{prop:vi-opt-conds-1},~\ref{prop:vi-opt-struct}} \\
  = & \min_{\pi \in \Pi^*} \Set{ \min_{(\ell, \rho) \in L_1 \cup L_2} \left[
      \max_{(\theta, x) \in \seti{}_{(\ell, \rho)}} \Set{\theta + \pi^{\top} x} -
      \pi^{\top} \hat{x} \right]} \tag{by Proposition~\ref{prop:vi-min-max}},
\end{align*}
where \(\Pi^*\) is a subset of \(\reals^n\) given by the optimality conditions
derived in Propositions~\ref{prop:vi-opt-conds-1} and~\ref{prop:vi-opt-struct}. Proposition~\ref{prop:vi-min-max} relates \(\seti\) with
\(\set{\seti_{(\ell, \rho)}}_{(\ell, \rho) \in L_1 \cup L_2}\). Hence, if $\hat{\theta}$ exceeds the optimal value of~\eqref{eqn:vi-main-theta-x}, there exist \(\pi' \in \Pi^*\) and \((\ell', \rho') \in L_1 \cup L_2\) that separates
\((\hat{\theta}, \hat{x})\) from \(\seti_{(\ell', \rho')}\).}

Let \(\set{\sigma_i}_{i = 1}^{\abs{\mathcal{Z}}}\) be
a permutation of \(\mathcal{Z} \equiv \{i \in [n]: \xi_i = 1\}\) such that
\(\hat{x}_{\sigma_1} \geq \hat{x}_{\sigma_2} \geq \cdots \geq \hat{x}_{\sigma_{|\mathcal{Z}|}}\).
For the outer minimization problem in formulation~\eqref{eqn:vi-main-theta-x},
we derive the following optimality conditions.
\begin{proposition}\label{prop:vi-main-opt-conds}
There exists an optimal solution \(\pi\) to
formulation~\eqref{eqn:vi-main-theta-x} satisfying the following:
\begin{enumerate}
\item \(\pi_{\sigma_1} \geq \pi_{\sigma_2} \geq \cdots \geq \pi_{\sigma_{\abs{\mathcal{Z}}}}\);
\item \(\pi \leq 0\), \(\pi_{\sigma_{\beta+1}} \geq -a_1\), and \(\pi_i = 0\) for all \(i \not \in \mathcal{Z}\).\qed
\end{enumerate}
\end{proposition}
Hence, it suffices to consider those \(\pi \in \reals^n_-\)
with the same ordering as \(\hat{x}\) on \(\mathcal{Z}\). It follows that
\begin{align}
  \max_{(\theta, x) \in \seti{}}\left\{\theta + \pi^{\top} x\right\} \
  & = \ \max_{T \in \{0\}\cup [\abs{\mathcal{Z}}]}\max_{x^{\top}\xi = T}\left\{ f\left((x^{\top}\xi - \beta)^+\right) + \sum_{i \in \mathcal{Z}} \pi_i x_i\right\} \nonumber \\
  & = \ \max_{T \in \{0\}\cup [\abs{\mathcal{Z}}]} \Set{
    f \left( \posp{ T - \beta } \right) + \sum_{i=1}^{T} \pi_{\sigma_i}} \label{eqn:vi-inner-f-sum} \\
  & = \ \max_{T \in \{0\}\cup [\abs{\mathcal{Z}}]} \Set{
    \min_{(\ell, \rho) \in L_1 \cup L_2} h_{(\ell, \rho)}(T) + \sum_{i=1}^{T} \pi_{\sigma_i}}, \nonumber
\end{align}
where the first equality uses the definition of \(\seti\) and the optimality
condition \(\pi_i = 0\) for all \(i \notin \mathcal{Z}\), the second equality
uses the ordering of \(\pi\) on \(\mathcal{Z}\), and the last equality follows
from Lemma~\ref{lem:vis-frepr}.

Second, we show that there exists a set \(\Pi^* \subseteq \mathbb{R}^n\) of
candidate optimal solutions such that a \(\pi \in \Pi^*\) solves
problem~\eqref{eqn:vi-main-theta-x}. In addition, we show that the following
minimax-type claim holds for all \(\pi \in \Pi^*\):
\begin{align*}
  \max_{T \in \{0\}\cup [\abs{\mathcal{Z}}]} \Set{
  \min_{(\ell, \rho) \in L_1 \cup L_2} h_{(\ell, \rho)}(T) + \sum_{i=1}^{T} \pi_{\sigma_i}}
  = \min_{(\ell, \rho) \in L_1 \cup L_2}
  \max_{T \in \{0\}\cup [\abs{\mathcal{Z}}]} \left\{\hfcn_{(\ell, \rho)}(T) + \sum_{i=1}^{T} \pi_{\sigma_i}\right\},
\end{align*}
\edits{To this end, we consider the one-dimensional maximization
problem~\eqref{eqn:vi-inner-f-sum} and observe that its objective function is
nonincreasing on \(\set{0} \cup [\beta]\) since
\(f \left( \posp{T - \beta} \right)\) is always \(0\).}
In addition, if
\(T \in \{\beta+1, \ldots, |\mathcal{Z}|\}\), it has decreasing increments
\(a_{i-\beta}+\pi_{\sigma_i}\) for any \(\beta + 1 \leq i \leq |\mathcal{Z}|\)
because \(f\) is concave. Since \(a_1 + \pi_{\sigma_{\beta + 1}} \geq 0\) by
Proposition~\ref{prop:vi-opt-conds-1}, we define \(t\) to be the last occurrence
of nonnegative increment, i.e.,
\begin{align*}
  t :=
  \begin{cases}
  \max \left\{i \in \{\beta+1, \ldots, |\mathcal{Z}|-1\}:
    a_{i - \beta} + \pi_{\sigma_i} \geq 0,
    a_{i - \beta + 1} + \pi_{\sigma_{i + 1}} \leq 0\right\}
  & \text{ if  } a_{\abs{\mathcal{Z}} - \beta} + \pi_{\sigma_{\abs{\mathcal{Z}}}} \leq 0 \\
    \abs{\mathcal{Z}}
  & \text{ if  } a_{\abs{\mathcal{Z}} - \beta} + \pi_{\sigma_{\abs{\mathcal{Z}}}} > 0.
  \end{cases}
\end{align*}
Then, with any fixed \(\pi\) satisfying the optimality conditions in
Proposition~\ref{prop:vi-opt-conds-1}, a maximizer \(T^*\) of
problem~\eqref{eqn:vi-inner-f-sum} is either \(0\) or \(t\), yielding the
optimal value
\begin{equation*}
    \left( \sum_{i = 1}^{\beta} \pi_{\sigma_i} +
\sum_{i = \beta + 1}^t (a_{i - \beta} + \pi_{\sigma_i}) \right)^+.
\end{equation*}
This allows us to represent problem~\eqref{eqn:vi-main-theta-x} as follows by
partitioning the feasible region of \(\pi\):
\begin{align*}
\eqref{eqn:vi-main-theta-x}
& = \min_{\substack{t \in \integers{}, \\ \beta + 1 \leq t \leq \abs{\mathcal{Z}}} }
  \min_{\pi \in \Pi_t} \Set{
    \left( \sum_{i = 1}^{\beta} \pi_{\sigma_i} +
  \sum_{i = \beta + 1}^t (a_{i - \beta} + \pi_{\sigma_i}) \right)^+ -
  \sum_{i = 1}^{\abs{\mathcal{Z}}} \pi_{\sigma_i} \hat{x}_{\sigma_i}
  }, \\
\mathrm{where} \quad \Pi_t
& := \Set{
  \pi \in \reals^n :
    \begin{aligned}
    & \text{the last occurrence of } a_{i - \beta} + \pi_{\sigma_i} \geq 0 \text{ is at } t \\
    & 0 \geq \pi_{\sigma_1} \geq \pi_{\sigma_2} \geq \cdots \geq \pi_{\sigma_{\abs{\mathcal{Z}}}} \\
    & \pi_i = 0, \quad \forall i \notin \mathcal{Z}
    \end{aligned}
  }, \quad \forall \beta + 1 \leq t \leq |\mathcal{Z}|.
\end{align*}
For each \(t\), we further identify the structure of an optimal \(\pi^* \in \Pi_t\).
\begin{proposition}
For all \(t \in \{\beta+1, \ldots, |\mathcal{Z}|\}\), there exists an optimal
solution \(\pi^*\) to the inner minimization problem
\begin{align*}
\min_{\pi \in \Pi_t} \Set{
  \left( \sum_{i = 1}^{\beta} \pi_{\sigma_i} +
  \sum_{i = \beta + 1}^t (a_{i - \beta} + \pi_{\sigma_i}) \right)^+ -
  \sum_{i = 1}^{\abs{\mathcal{Z}}} \pi_{\sigma_i} \hat{x}_{\sigma_i}}
\end{align*}
that has the following structure:
\begin{enumerate}
\item If \(b_{t - \beta} - \beta a_{t - \beta} \geq 0\), then
    $\pi^*_{\sigma_1} = \pi^*_{\sigma_2} = \cdots = \pi^*_{\sigma_{\abs{\mathcal{Z}}}} = -a_{t - \beta}$.
\item If \(b_{t - \beta} - \beta a_{t - \beta} < 0\), then \(\pi^*\) takes one of the following two forms:
\begin{enumerate}
\item[(A)] Let \(\tilde{t} := \beta - \frac{b_{t - \beta}}{a_{t - \beta}}\). Then,
    \begin{align*}
      0 = \pi^*_{\sigma_1} = \cdots = \pi^*_{\sigma_{\floor{\tilde{t}}}} \geq
      \pi^*_{\sigma_{\ceil{\tilde{t}}}} = -a_{t - \beta} (\ceil{\tilde{t}} - \tilde{t}) >
      \pi^*_{\sigma_{\ceil{\tilde{t}}} + 1} = \cdots = \pi^*_{\sigma_{\abs{\mathcal{Z}}}} = -a_{t - \beta}.
    \end{align*}
\item[(B)] If there exists an integer
    \(t^{\prime} \in \left( \beta - \frac{b_{t - \beta + 1}}{a_{t - \beta + 1}},
      \beta - \frac{b_{t - \beta}}{a_{t - \beta}} \right) \cap \mathbb{Z}_+\){\color{black},} then\\[0.2cm]
    $\displaystyle 0 = \pi^*_{\sigma_1} = \cdots = \pi^*_{\sigma_{t^{\prime}}} > \pi^*_{\sigma_{t^{\prime} + 1}}
    = \cdots = \pi^*_{\sigma_{\abs{\mathcal{Z}}}} = - \frac{f(t - \beta)}{(t - t^{\prime})}$.\qed
\end{enumerate}
\end{enumerate}
\end{proposition}
Accordingly, for all \(t \in \{\beta + 1, \ldots, \abs{\mathcal{Z}}\}\), if
\(b_{t - \beta} - \beta a_{t - \beta} \geq 0\), then we define
\begin{align*}
  \Pi^*_t := \left\{\pi \in \reals^n:
  \pi_{\sigma_1} = \pi_{\sigma_2} = \cdots = \pi_{\sigma_{\abs{\mathcal{Z}}}} = - a_{t - \beta}, \ \pi_i = 0,  \forall i \notin \mathcal{Z}\right\};
\end{align*}
and if \(b_{t - \beta} - \beta a_{t - \beta} < 0\), then we define
\begin{align*}
  \Pi^*_t := \left\{\pi \in \reals^n:
    \begin{aligned}
      & \pi_{\sigma_1} = \pi_{\sigma_2} = \cdots = \pi_{\sigma_{\floor{\tilde{t}}}} = 0, \\
      & \pi_{\sigma_{\ceil{\tilde{t}}}} = -a_{t - \beta} (\ceil{\tilde{t}} - \tilde{t}), \\
      & \pi_{\sigma_{\ceil{\tilde{t}} + 1}} = \cdots = \pi_{\sigma_{\abs{\mathcal{Z}}}} = -a_{t - \beta}\\
      & \pi_i = 0, \quad \forall i \notin \mathcal{Z}
    \end{aligned}\right\}
        \cup \left\{ \pi \in \reals^n:
    \begin{aligned}
      & \pi_{\sigma_1} = \cdots = \pi_{\sigma_{t^{\prime}}} = 0, \\
      & \pi_{\sigma_{t^{\prime} + 1}} = \cdots \pi_{\sigma_{\abs{\mathcal{Z}}}}
      = - \frac{f(t - \beta)}{(t - t^{\prime})}, \\
      & \pi_i = 0, \quad \forall i \notin \mathcal{Z} \\
      & t^{\prime} \in \left( \beta - \frac{b_{t - \beta + 1}}{a_{t - \beta + 1}}, \beta - \frac{b_{t - \beta}}{a_{t - \beta}} \right)\cap \mathbb{Z}_+
    \end{aligned}\right\}.
\end{align*}
Note that if the interval
\(( \beta - b_{t - \beta + 1}/a_{t - \beta + 1}, \beta - b_{t - \beta}/a_{t - \beta} )\)
does not contain an integer, then the second component of the above union equals
\(\varnothing\). Let \(\displaystyle \Pi^* := \bigcup_t \Pi^*_t\), then
Proposition~\ref{prop:vi-opt-struct} shows that \(\Pi_t\) can be replaced by
\(\Pi_t^*\) without loss of optimality. In other words,
\begin{align*}
\eqref{eqn:vi-main-theta-x}
& = \min_{\substack{t \in \integers{}, \\ \beta + 1 \leq t \leq \abs{\mathcal{Z}}} }
  \min_{\pi \in \Pi^*_t} \Set{
    \left( \sum_{i = 1}^{\beta} \pi_{\sigma_i} +
  \sum_{i = \beta + 1}^t (a_{i - \beta} + \pi_{\sigma_i}) \right)^+ -
  \sum_{i = 1}^{\abs{\mathcal{Z}}} \pi_{\sigma_i} \hat{x}_{\sigma_i}
  } \\
& = \min_{\pi \in \Pi^*} \Set{
    \max_{T \in \{0\}\cup [\abs{\mathcal{Z}}]} \Set{
        f(\posp{T - \beta}) + \sum_{i = 1}^T \pi_{\sigma_i}
    } - \sum_{i = 1}^{\abs{\mathcal{Z}}} \pi_{\sigma_i} \hat{x}_{\sigma_i} }.
\end{align*}
Furthermore, Proposition~\ref{prop:vi-min-max} shows that, for all \(\pi \in \Pi^*\),
\edits{
\begin{align*}
  \max_{T \in \{0\}\cup [\abs{\mathcal{Z}}]} \Set{
  f \left( \posp{ T - \beta } \right) + \sum_{i=1}^{T} \pi_{\sigma_i}}
  = \min_{(\ell, \rho) \in L_1 \cup L_2}
  \max_{T \in \{0\}\cup [\abs{\mathcal{Z}}]} \left\{ \hfcn_{(\ell, \rho)}(T) + \sum_{i=1}^{T} \pi_{\sigma_i}\right\}.
\end{align*}
}
It follows that, for all \(\pi \in \Pi^*\),
\begin{align*}
  \max_{T \in \{0\}\cup [\abs{\mathcal{Z}}]} \Set{
  f \left( \posp{ T - \beta } \right) + \sum_{i=1}^{T} \pi_{\sigma_i}} - \sum_{i = 1}^{\abs{\mathcal{Z}}} \pi_{\sigma_i} \hat{x}_{\sigma_i}
  & = \min_{(\ell, \rho) \in L_1 \cup L_2}
  \max_{x \in \binaries^n} \Set{\hfcn_{(\ell, \rho)}(x^{\top}\xi) + \pi^{\top}x} - \pi^{\top} \hat{x} \\
  & = \min_{(\ell, \rho) \in L_1 \cup L_2} \max_{(\theta, x) \in \seti_{(\ell, \rho)}} \Set{\theta + \pi^{\top}x} - \pi^{\top} \hat{x},
\end{align*}
where the first equality is because \(\pi\) satisfies the optimality conditions
in Proposition \ref{prop:vi-opt-conds-1}. Thus, the optimal value of
problem~\eqref{eqn:vi-main-theta-x} satisfies
\begin{align*}
\hat{\theta}
& > \min_{\pi \in \Pi^*} 
\min_{(\ell, \rho) \in L_1 \cup L_2} \Set{
\max_{(\theta, x) \in \seti_{(\ell, \rho)}} \Set{\theta + \pi^{\top}x} - \pi^{\top} \hat{x}}.
\end{align*}
That is, there exist \((\ell, \rho) \in L_1 \cup L_2\) and \(\pi' \in \Pi^*\) such that
\begin{align*}
\hat{\theta} + \hat{x}^{\top} \pi' > \theta + x^{\top} \pi',
\quad \forall (\theta, x) \in \seti_{(\ell, \rho)}.
\end{align*}
Therefore, \((1, \pi')\) separates \((\hat{\theta}, \hat{x})\) from the set
\(\seti_{(\ell, \rho)}\), implying
\((\hat{\theta}, \hat{x}) \not \in \Co{} (\seti_{(\ell, \rho)})\). This finishes
the proof.
\end{proof}

When applying the BD algorithm to solve (\DRCCbC{}), we add a Benders
feasibility cut \eqref{eqn:reform-feasi-base} if the incumbent solution
\((\hat{x}, \hat{\gamma}, \hat{z})\) is infeasible to \((\dsubp_j)\). In the
meantime, we add more inequalities by separating
\((\hat{x}, \hat{\gamma}, \hat{z})\) from \(\Co(\seti{}^j_{i^*})\) with
\begin{align*}
  \seti{}^j_{i^*} := \Set{(x, \gamma, z_j) \in \binaries^{n} \times \reals^2:
  z_j + \gamma \leq \left(\posp{x^{\top}\hat{\xi}^j_{i^*} - v_{i^*} + 1}\right)^{1/p}},
\end{align*}
where \(i^*\) is identified in Proposition \ref{prop:reform-closed-form-sol}.
Since \((\cdot)^{1/p}\) is strictly concave and increasing, \(\seti{}^j_{i^*}\)
is isomorphic to \(\seti\) with \(z_j + \gamma\) and \(v_{i^*}-1\) playing the
roles of \(\theta\) and \(\beta\), respectively. As a consequence, separating
from \(\Co(\seti{}^j_{i^*})\) is efficient. In addition, with regard to any \((\ell, \rho)\),
a valid inequality takes the following form:
\begin{align}
  & - z_j - \gamma
  \geq \Phi^j(\varnothing) +
    \sum_{k = 1}^n \left[ \Phi^j(\mathcal{T}_k) - \Phi^j(\mathcal{T}_{k - 1}) \right] x_{\sigma_k},
    \label{eqn:cut-single}\\
  \mathrm{where} \quad
  & \Phi^j(x) := \frac{\ell^{1/p}}{\ell+\rho}\min\left\{
    -x^{\top}\hat{\xi}^j_{i^*} - v_{i^*}-\rho+1, \ 0\right\}, \nonumber{} \\
  \quad
  & \Phi^j(\mathcal{T}_k) - \Phi^j(\mathcal{T}_{k - 1}) =
  \begin{cases}
  - \frac{\ell^{1/p}}{\ell+\rho} & \text{ if } \Phi^j(\mathcal{T}_{k - 1}) < 0 \\
  \Phi^j(\mathcal{T}_k) & \text{ if } \Phi^j(\mathcal{T}_{k - 1}) = 0 \text{ and } \Phi^j(\mathcal{T}_k) < 0 \\
  0 & \text{ if } \Phi^j(\mathcal{T}_{k}) = 0,
  \end{cases} \nonumber{}
\end{align}
and \(\sigma\) is a permutation of \([n]\) such that
\(\hat{x}_{\sigma_1} \geq \hat{x}_{\sigma_2} \geq \cdots \geq \hat{x}_{\sigma_n}\).
In our implementation, we perform this separation via the lazy callback.

\subsection{Proofs of Propositions~\ref{prop:vi-opt-conds-1},~\ref{prop:vi-opt-struct},
  and~\ref{prop:vi-min-max}}%
\label{sec:proofs-of-propositions}

\setcounter{prop}{7}
\begin{proposition}%
\label{prop:vi-opt-conds-1}
There exists an optimal solution \(\pi\) to
formulation~\eqref{eqn:vi-main-theta-x} satisfying the following:
\begin{enumerate}
\item \(\pi_{\sigma_1} \geq \pi_{\sigma_2} \geq \cdots \geq \pi_{\sigma_{\abs{\mathcal{Z}}}}\);
\item \(\pi \leq 0\), \(\pi_{\sigma_{\beta+1}} \geq -a_1\), and \(\pi_i = 0\) for all \(i \not \in \mathcal{Z}\).
\end{enumerate}
\end{proposition}
\begin{proof}
First, for any \(\pi \in \reals^n\) let \(\set{\sigma'_i: i \in \mathcal{Z}}\),
possibly different from \(\set{\sigma_i: i \in \mathcal{Z}}\), be a permutation
of \(\mathcal{Z} \equiv \{i \in [n]: \xi_i = 1\}\) such that
\(\pi_{\sigma'_1} \geq \cdots \geq \pi_{\sigma'_{|\mathcal{Z}|}}\). Then, we
recast formulation~\eqref{eqn:vi-main-theta-x} as
\begin{align}
& \min_{\pi \in \reals^{n}} \Set{
      \max \Set{
          \theta + \pi^{\top} x : (\theta, x) \in \seti{}
      } - \pi^{\top} \hat{x}
} \nonumber \\
= \ & \min_{\pi \in \reals^{n}} \Set{
  \max_{x \in \binaries^n} \Set{
    f \left( \posp{ x^{\top}\xi - \beta } \right) + \sum_{i\in \mathcal{Z}} \pi_{\sigma'_i} x_{\sigma'_i} +
    \sum_{i \not \in \mathcal{Z}} \pi_i x_i} - \pi^{\top} \hat{x}
} \nonumber \\
= \ & \min_{\pi \in \reals^{n}} \left\{
  \max_{T \in \{0\} \cup [|\mathcal{Z}|]}\max_{x^{\top}\xi = T} \Set{
    f \left( \posp{ x^{\top}\xi - \beta } \right) + \sum_{i\in \mathcal{Z}} \pi_{\sigma'_i} x_{\sigma'_i} +
    \sum_{i \not \in \mathcal{Z}} \pi_i x_i} - \pi^{\top} \hat{x}
\right\} \nonumber \\
= \ & \min_{\pi \in \reals^{n}} \left\{
  \max_{T \in \{0\} \cup [|\mathcal{Z}|]} \Set{
    f \left( \posp{ T - \beta } \right) + \sum_{i=1}^{T} \pi_{\sigma'_i}} +
    \sum_{i \not \in \mathcal{Z}} \posp{\pi_i} - \pi^{\top} \hat{x}
\right\}. \label{eqn:vi-main-f-sum}
\end{align}
We prove the following optimality conditions:
\begin{description}
\item[Condition 1.] Suppose that \(\pi \in \reals^n\) is such that
    \(\pi_{\sigma_j} < \pi_{\sigma_k}\) for some \(j, k \in [\abs{\mathcal{Z}}]\) and \(j < k\). We show swapping \(\pi_{\sigma_j}\) and \(\pi_{\sigma_k}\)
    yields a lower objective value of \eqref{eqn:vi-main-f-sum}. Specifically, define \(\pi'\) as
    \begin{gather*}
    \pi'_{\sigma_i} :=
    \begin{cases}
    \pi_{\sigma_{k}} & \mbox{if } i = j, \\
    \pi_{\sigma_{j}} & \mbox{if } i = k, \\
    \pi_{\sigma_i} & \mbox{\ow{}} \\
    \end{cases}
    \quad{} \forall i \in \mathcal{Z}, \\
    \pi'_i := \pi_i \quad \forall i \not \in \mathcal{Z}.
    \end{gather*}
    Since \(\pi, \pi'\) differ only in ordering, the optimal value of the
    inner maximization problem of \eqref{eqn:vi-main-f-sum} remains the same after swapping.
    Then, the difference between their objective values equals
    \begin{align*}
      - \hat{x}^{\top}\pi - \left( -\hat{x}^{\top}\pi' \right)
      & = - ( \pi_{\sigma_j} \hat{x}_{\sigma_j} + \pi_{\sigma_k} \hat{x}_{\sigma_k} ) +
        ( \pi'_{\sigma_j} \hat{x}_{\sigma_j} + \pi'_{\sigma_k} \hat{x}_{\sigma_k} ) \\
      & = (\hat{x}_{\sigma_j} - \hat{x}_{\sigma_k}) (\pi_{\sigma_k} - \pi_{\sigma_j}) \geq 0.
    \end{align*}
    Therefore, there exists an optimal
    \(\pi\) that shares the same ordering with \(\hat{x}\) on \(\mathcal{Z}\).

\item[Condition 2.]\label{it:vi-main-cond2} To prove that the entries of an optimal \(\pi\) on \([n] \setminus \mathcal{Z}\)
      are all zero, we show that replacing any nonzero entry \(\pi_j\),
      \(j \notin \mathcal{Z}\), with \(0\) lowers the objective value.
      Specifically, define \(\pi'\) to be the \(\pi\) with \(\pi_j\) replaced by zero.
      This yields a difference in objective value
      \begin{align*}
        \left(\posp{\pi_j} - \pi_j \hat{x}_j\right) - 0
        = \pi_j(1 - \hat{x}_j) \Ind{\pi_j > 0} -
        \pi_j \hat{x}_j \Ind{\pi_j \leq 0 }
        \geq 0.
      \end{align*}
      To prove \(\pi \leq 0\) without loss of optimality, it suffices to show that \(\pi\) is
      nonpositive on \(\mathcal{Z}\). Suppose that \(\pi\) has strictly positive entries on
      \(\mathcal{Z}\) and let \(\pi_{\sigma_j}\) be the last such entry, i.e., either \(\pi_{\sigma_j} >
      0, \pi_{\sigma_{j+1}} \leq 0\) or \(\pi_{\sigma_{|\mathcal{Z}|}} > 0\) with \(j = |\mathcal{Z}|\).
      We show that replacing \(\pi_{\sigma_j}\) with \(0\) lowers the objective value. Specifically,
      let \(\pi'\) be the \(\pi\) with \(\pi_{\sigma_j}\) replaced by \(0\).
      Then, \(\pi\) and \(\pi'\) share the same ordering. In addition, both \(\pi_{\sigma_i}\)
      and \(\pi'_{\sigma_i}\) are nonnegative for all \(i \in [j]\). As a result, since \(f\) is
      increasing, a maximizer \(T^*\) of the inner maximization problem of~\eqref{eqn:vi-main-f-sum}
      must be at least \(j\) with respect to both \(\pi\) and \(\pi'\). It follows that
      \(\pi\) and \(\pi'\) yield the same maximizer \(T^*\). Therefore, the difference in
      objective value equals
      \begin{align*}
        \pi_{\sigma_j}(1 - \hat{x}_{\sigma_j}) - 0 \geq 0,
      \end{align*}
      implying that the objective value evaluated at \(\pi'\) is smaller.
      \edits{To prove \(\pi_{\sigma_{\beta+1}} \geq -a_1\), we note that if
      \(\pi_{\sigma_{\beta + 1}} + a_1 < 0\) then
      \(\pi_{\sigma_{\beta + i}} + a_i < 0\) for all \(i \in [\abs{\mathcal{Z}} - \beta]\),
      because \(\pi_{\sigma_{\beta + i}}\) and \(a_i\) are decreasing in $i$. Since
      \(\pi \leq 0\), an optimal \(T^*\) to the inner maximization problem of \eqref{eqn:vi-main-f-sum} is
      \(0\). It follows that increasing all \(\pi_{\sigma_i}\), \(1 \leq i \leq \beta + 1\) to
      \(-a_1\) does not increase the objective value of \eqref{eqn:vi-main-f-sum}. Hence, we have
      \(\pi_{\sigma_{\beta + 1}} + a_1 \geq 0\) without loss of optimality.}
\end{description}
\end{proof}

\begin{proposition}%
\label{prop:vi-opt-struct}
For all \(t \in \{\beta+1, \ldots, |\mathcal{Z}|\}\), there exists an optimal
solution \(\pi^*\) to the inner minimization problem
\begin{align}
\min_{\pi \in \Pi_t} \Set{
\left( \sum_{i = 1}^{\beta} \pi_{\sigma_i} +
  \sum_{i = \beta + 1}^t (a_{i - \beta} + \pi_{\sigma_i}) \right)^+ -
  \sum_{i = 1}^{\abs{\mathcal{Z}}} \pi_{\sigma_i} \hat{x}_{\sigma_i}
} \label{eqn:vi-main-6}
\end{align}
that has the following structure:
\begin{enumerate}
\item If \(b_{t - \beta} - \beta a_{t - \beta} \geq 0\), then
$\pi^*_{\sigma_1} = \pi^*_{\sigma_2} = \cdots = \pi^*_{\sigma_{\abs{\mathcal{Z}}}} = -a_{t - \beta}$.
\item If \(b_{t - \beta} - \beta a_{t - \beta} < 0\), then \(\pi^*\) takes one of the following two forms:
\begin{enumerate}
\item[(A)] Let \(\tilde{t} := \beta - \frac{b_{t - \beta}}{a_{t - \beta}}\). Then,
\begin{align*}
0 = \pi^*_{\sigma_1} = \cdots = \pi^*_{\sigma_{\floor{\tilde{t}}}} \geq
\pi^*_{\sigma_{\ceil{\tilde{t}}}} = -a_{t - \beta} (\ceil{\tilde{t}} - \tilde{t}) >
\pi^*_{\sigma_{\ceil{\tilde{t}}} + 1} = \cdots = \pi^*_{\sigma_{\abs{\mathcal{Z}}}} = -a_{t - \beta}.
\end{align*}
\item[(B)] If there exists an integer \(t^{\prime} \in \left( \beta - \frac{b_{t - \beta + 1}}{a_{t - \beta + 1}},
      \beta - \frac{b_{t - \beta}}{a_{t - \beta}} \right) \cap \mathbb{Z}_+\){\color{black},} then\\[0.2cm]
$\displaystyle 0 = \pi^*_{\sigma_1} = \cdots = \pi^*_{\sigma_{t^{\prime}}}
> \pi^*_{\sigma_{t^{\prime} + 1}} = \cdots = \pi^*_{\sigma_{\abs{\mathcal{Z}}}} = - \frac{f(t - \beta)}{(t - t^{\prime})}$.
\end{enumerate}
\end{enumerate}
\end{proposition}

\begin{proof}
Fix \(t \in \{\beta+1, \ldots, |\mathcal{Z}|\}\) and we discuss the following two cases.

First, if \(b_{t - \beta} - \beta a_{t - \beta} \geq 0\) then
\begin{align*}
\sum_{i = 1}^{\beta} \pi_{\sigma_i} + \sum_{i = \beta + 1}^t (a_{i - \beta} + \pi_{\sigma_i})
& \geq t\pi_{\sigma_t} + \sum_{i = \beta + 1}^t a_{i - \beta} \\
& \geq - t a_{t - \beta} + a_{t - \beta} (t - \beta) + b_{t - \beta} \\
& = b_{t - \beta} - \beta a_{t - \beta} \geq 0,
\end{align*}
where the first inequality is because \(\pi_{\sigma_i} \geq \pi_{\sigma_t}\) for
all \(i \in [t]\) and the second inequality is because
\(\pi_{\sigma_t} \geq - a_{t-\beta}\) by the definition of \(t\) and
\(\sum_{i=\beta+1}^t a_{i - \beta} = \sum_{i=\beta+1}^t (f(i-\beta) - f(i-\beta-1))
= f(t - \beta) = a_{t - \beta} (t - \beta) + b_{t - \beta}\).
Then, we recast \eqref{eqn:vi-main-6} as
\begin{align*}
\mathrm{\eqref{eqn:vi-main-6}} = \min_{\pi \in \Pi_t} \; & \Set{
\sum_{i = \beta + 1}^t a_{i - \beta} +
\sum_{i = 1}^t \pi_{\sigma_i} (1 - \hat{x}_{\sigma_i}) -
\sum_{i = t + 1}^{\abs{\mathcal{Z}}} \pi_{\sigma_i} \hat{x}_{\sigma_i}
}.
\end{align*}
Since \(\hat{x} \in [0, 1]^n\), decreasing \(\{\pi_{\sigma_i}: i \in [t]\}\) and
increasing \(\{\pi_{\sigma_i}: t+1 \leq i \leq |\mathcal{Z}|\}\) improve the
objective value. It follows that an optimal solution \(\pi^*\) takes the
following form:
\[
   \pi^*_{\sigma_1} = \cdots = \pi^*_{\sigma_t} = - a_{t - \beta} =
   \pi^*_{\sigma_{t + 1}} = \cdots =
   \pi^*_{\sigma_{\abs{\mathcal{Z}}}}.
\]

Second, suppose that \(b_{t - \beta} - \beta a_{t - \beta} < 0\). For \(\pi \in \Pi_t\) we define
\begin{align*}
S(\pi) := \sum_{i = 1}^{\beta} \pi_{\sigma_i} +
          \sum_{i = \beta + 1}^t (a_{i - \beta} + \pi_{\sigma_i}).
\end{align*}
We show \(S(\pi) = 0\) at optimality of \eqref{eqn:vi-main-6}. To see that, we
discuss the following two cases.
\begin{enumerate}
\item If \(S(\pi) > 0\), then there exists an \(i \in \integers{}, \beta + 1
      \leq i \leq \abs{\mathcal{Z}}\) such that \(a_{i - \beta} + \pi_{\sigma_i}
      > 0\). Let \(\underline{t}\) be its last occurrence:
    \begin{align*}
      \underline{t} := \max \Set{
      \beta + 1 \leq i \leq \abs{\mathcal{Z}} : a_{i - \beta} + \pi_{\sigma_i} > 0
      }.
    \end{align*}
    Note that \(\underline{t} \leq t\) by definition and
    \(a_{i - \beta} + \pi_{\sigma_i} = 0\) for all \(\underline{t} < i \leq t\).
    Then,
    \begin{align*}
      \max \Set{S(\pi), 0} - \sum_{i = 1}^{\abs{\mathcal{Z}}} \pi_{\sigma_i} \hat{x}_{\sigma_i}
      & = \sum_{i = 1}^{\beta} \pi_{\sigma_i} + \sum_{i = \beta + 1}^{\underline{t}} (a_{i - \beta} + \pi_{\sigma_i}) -
        \sum_{i = 1}^{\abs{\mathcal{Z}}} \pi_{\sigma_i} \hat{x}_{\sigma_i} \\
      & = \sum_{i = \beta + 1}^{\underline{t}} a_{i - \beta} +
        \sum_{i = 1}^{\underline{t} - 1} \pi_{\sigma_i} (1 - \hat{x}_{\sigma_i}) +
        \pi_{\sigma_{\underline{t}}}(1 - \hat{x}_{\sigma_{\underline{t}}}) -
        \sum_{i = \underline{t} + 1}^{\abs{\mathcal{Z}}} \pi_{\sigma_i} \hat{x}_{\sigma_i}.
    \end{align*}
    Therefore, as long as \(S(\pi) > 0\), decreasing
    \(\pi_{\sigma_{\underline{t}}}\) lowers the objective value
    of~\eqref{eqn:vi-main-6}.

\item If \(S(\pi) < 0\), then there exists an \(i \in [\beta]\) such that
    \(\pi_{\sigma_i} < 0\). Let \(\bar{t}\) be its first occurrence:
    \begin{align*}
      \bar{t} := \min \Set{i \in [\beta]: \pi_{\sigma_i} < 0}.
    \end{align*}
    Then, by noting that
    \(\pi_{\sigma_i} = 0\) for all \(i \in [\bar{t} - 1]\), we have
    \begin{align*}
      \max \Set{S(\pi), 0} - \sum_{i = 1}^{\abs{\mathcal{Z}}} \pi_{\sigma_i} \hat{x}_{\sigma_i}
      & = - \pi_{\sigma_{\bar{t}}} \hat{x}_{\sigma_{\bar{t}}} -
        \sum_{i = \bar{t} + 1}^{\abs{\mathcal{Z}}} \pi_{\sigma_i} \hat{x}_{\sigma_i}.
    \end{align*}
    Therefore, as long as \(S(\pi) < 0\), increasing \(\pi_{\sigma_{\bar{t}}}\)
    lowers the objective value of \eqref{eqn:vi-main-6}.
\end{enumerate}

In view of the optimality condition \(S(\pi) = 0\), we reformulate
\eqref{eqn:vi-main-6} as the following linear program:
\begin{align}
\eqref{eqn:vi-main-6} =
\min_{\pi} ~
& - \sum_{i = 1}^{\abs{\mathcal{Z}}} \pi_{\sigma_i} \hat{x}_{\sigma_i}, \nonumber \\
\st{} ~
& 0 \geq \pi_{\sigma_1}, \nonumber \\
& \pi_{\sigma_i} \geq \pi_{\sigma_{i + 1}}, ~ \forall i \in [\abs{\mathcal{Z}}-1], \nonumber \\
& \pi_{\sigma_t} + a_{t - \beta} \geq 0, \nonumber \\
& \pi_{\sigma_{t + 1}} + a_{t - \beta + 1} \leq 0, \label{eqn:vi-next-constr} \\
& \sum_{i = 1}^t \pi_{\sigma_i} + \sum_{i = \beta + 1}^t a_{i - \beta} = 0, \label{lp-note-1}
\end{align}
where constraint~\eqref{lp-note-1} states \(S(\pi) = 0\) and the remaining
constraints come from the definition of \(\Pi_t\). \edits{Note that constraint~\eqref{eqn:vi-next-constr} does not exist in the case $t = |\mathcal{Z}|$. Although the remaining proof focuses on $t < |\mathcal{Z}|$, it can be directly adapted to the case $t = |\mathcal{Z}|$.} In what
follows, we analyze the basic feasible solutions of the above linear program to
identify the structure of an optimal solution \(\pi^*\). There are
\(\abs{\mathcal{Z}}\) decision variables, 1 equality constraint, and
\(\abs{\mathcal{Z}} + 2\) inequality constraints. By definition of \(t\), the
constraint \(\pi_{\sigma_{t + 1}} + a_{t - \beta + 1} \leq 0\) must hold
strictly. In addition, variables
\(\pi_{\sigma_{t+2}}, \ldots, \pi_{\sigma_{|\mathcal{Z}|}}\) are only
constrained by their ordering and that
\(\pi_{\sigma_{t+2}} \leq \pi_{\sigma_{t+1}}\), and their coefficients in the
objective function are all negative. As a consequence,
\(\pi_{\sigma_{t+1}} = \pi_{\sigma_{t+2}} = \cdots = \pi_{\sigma_{\abs{\mathcal{Z}}}}\)
at optimality. Hence, we only need to consider those basic feasible solutions,
where either (A) \(\pi_{\sigma_t} + a_{t - \beta} = 0\) and two among the
constraints \(0 \geq \pi_{\sigma_1} \geq \cdots \geq \pi_{\sigma_{t+1}}\) are
satisfied strictly (or weakly in case of degeneracy), or (B)
\(\pi_{\sigma_t} + a_{t - \beta} > 0\) and one of the constraints
\(0 \geq \pi_{\sigma_1} \geq \cdots \geq \pi_{\sigma_{t+1}}\) is satisfied
strictly (or weakly in case of degeneracy). We discuss cases (A) and (B) to
finish the proof.

\begin{enumerate}
\item[(A)] Since \(\pi_{\sigma_t} + a_{t - \beta} = 0\), there exist
    \(0 \leq t_1 < t_2 \leq t\) such that
    \begin{align*}
      0 = \pi_{\sigma_1} = \cdots = \pi_{\sigma_{t_1}} \geq \pi_{\sigma_{t_1 + 1}} = \cdots
      = \pi_{\sigma_{t_2}} \geq \pi_{\sigma_{t_2 + 1}} = \cdots = \pi_{\sigma_{\abs{\mathcal{Z}}}} = -a_{t - \beta}.
    \end{align*}
    By constraint~\eqref{lp-note-1}, we have:
    \begin{align*}
      (t_2 - t_1) \pi_{\sigma_{t_2}}
      & = - \sum_{i=t_2 + 1}^t (-a_{t-\beta}) - \sum_{i = \beta + 1}^t a_{i - \beta}
        = - t_2 a_{t - \beta} + (\beta a_{t-\beta} - b_{t-\beta}),
    \end{align*}
    where the last equality is because
    \(\sum_{i=\beta+1}^t a_{i - \beta} = \sum_{i=\beta+1}^t (f(i-\beta) - f(i-\beta-1)) = f(t - \beta) = a_{t - \beta} (t - \beta) + b_{t - \beta}\).
    Let \(\tilde{t} := \beta - \frac{b_{t - \beta}}{a_{t - \beta}}\).
    Then, a valid choice of \(\pi_{\sigma_{t_2}}\) should satisfy
    \(-a_{t-\beta} \leq \pi_{\sigma_{t_2}} \leq 0\), which, together with the
    above equality, implies
    \begin{align*}
      & - (t_2 - t_1) a_{t-\beta} \leq (\tilde{t} - t_2) a_{t - \beta} \leq 0 \quad \Leftrightarrow \quad t_1 \leq \tilde{t} \leq t_2.
    \end{align*}
    We assume that \(\tilde{t}\) is fractional and will return to the case of
    integer-valued \(\tilde{t}\) in the end. We show \(t_1 = \floor{\tilde{t}}\)
    and \(t_2 = \ceil{\tilde{t}}\) at optimality. Indeed, if
    \(t_1 \leq \lfloor \tilde{t} \rfloor - 1\) then increasing \(t_1\) by 1
    improves the objective value. Specifically, let \(\pi'\) be the \(\pi\) with
    \(t_1\) replaced by \(t_1 + 1\). Then,
    \(\pi'_{\sigma_{t_2}} \leq \pi_{\sigma_{t_2}}\) because
    \(\sum_{i=1}^t \pi_{\sigma_i}\) equals a constant due to constraint
    \eqref{lp-note-1}. In addition,
    \begin{align}
      (t_2 - t_1)\pi_{\sigma_{t_2}}
      & = \sum_{i=1}^t \pi_{\sigma_i} - \sum_{i = t_2 + 1}^t \pi_{\sigma_i} \tag{by construction of \(\pi\)} \\
      & = -\sum_{i=\beta+1}^t a_{i-\beta} - \sum_{i = t_2 + 1}^t \pi_{\sigma_i} \tag{by constraint \eqref{lp-note-1}} \\
      & = \sum_{i=1}^t \pi'_{\sigma_i} - \sum_{i = t_2 + 1}^t \pi'_{\sigma_i} \tag{by construction of \(\pi'\)} \\
      & = (t_2 - t_1 - 1)\pi'_{\sigma_{t_2}}. \nonumber{}
    \end{align}
    Hence, the change in objective value equals
    \begin{align*}
      - \pi_{\sigma_{t_2}} \sum_{i = t_1 + 1}^{t_2} \hat{x}_{\sigma_i} +
      \pi'_{\sigma_{t_2}} \sum_{i = t_1 + 2}^{t_2} \hat{x}_{\sigma_i}
      & =
        - \pi_{\sigma_{t_2}} \hat{x}_{\sigma_{t_1 + 1}} - (\pi_{\sigma_{t_2}} - \pi'_{\sigma_{t_2}}) \sum_{i=t_1 + 2}^{t_2} \hat{x}_{\sigma_i} \\
      & \geq - \pi_{\sigma_{t_2}} \hat{x}_{\sigma_{t_1 + 1}} -
        (\pi_{\sigma_{t_2}} - \pi'_{\sigma_{t_2}}) \hat{x}_{\sigma_{t_1 + 1}} (t_2 - t_1 - 1) \\
      & = - \hat{x}_{\sigma_{t_1 + 1}} \left[ \pi_{\sigma_{t_2}} +
        (\pi_{\sigma_{t_2}} - \pi'_{\sigma_{t_2}}) (t_2 - t_1 -1) \right] \\
      & = 0,
    \end{align*}
    where the inequality is because \(\hat{x}_{\sigma_i}\) decreases in \(i\).
    Therefore, \(t_1 = \floor{\tilde{t}}\) at optimality. Likewise, if
    \(t_2 \geq \floor{\tilde{t}} + 1\) then decreasing \(t_2\) by \(1\) improves
    the objective value. Specifically, we let \(\pi''\) be the \(\pi\) with
    \(t_2\) replaced by \(t_2 - 1\). Then,
    \(\pi''_{\sigma_{t_2}} \geq \pi_{\sigma_{t_2}}\) because
    \(\sum_{i=1}^t \pi_{\sigma_i}\) equals a constant due to
    constraint~\eqref{lp-note-1}. In addition,
    \begin{align}
      (t_2 - t_1)\pi_{\sigma_{t_2}}
      & = \sum_{i=1}^t \pi_{\sigma_i} - \sum_{i = t_2 + 1}^t \pi_{\sigma_i} \tag{by construction of \(\pi\)} \\
      & = -\sum_{i=\beta+1}^t a_{i-\beta} - \sum_{i = t_2 + 1}^t \pi_{\sigma_i} \tag{by constraint~\eqref{lp-note-1}} \\
      & = \sum_{i=1}^t \pi''_{\sigma_i} - \sum_{i = t_2 + 1}^t \pi''_{\sigma_i} \tag{by construction of \(\pi''\)} \\
      & = (t_2 - t_1 - 1)\pi''_{\sigma_{t_2}} + (-a_{t-\beta}). \nonumber
    \end{align}
    Hence, the change in objective value equals
    \begin{align*}
      & \; - \pi_{\sigma_{t_2}} \sum_{i = t_1 + 1}^{t_2} \hat{x}_{\sigma_i} +
        a_{t - \beta} \sum_{i = t_2 + 1}^{\abs{\mathcal{Z}}} \hat{x}_{\sigma_i}
        + \pi''_{\sigma_{t_2}} \sum_{i = t_1 + 1}^{t_2 - 1} \hat{x}_{\sigma_i} -
        a_{t - \beta} \sum_{i = t_2}^{\abs{\mathcal{Z}}} \hat{x}_{\sigma_i} \\
      = & \; - (\pi_{\sigma_{t_2}} - \pi''_{\sigma_{t_2}}) \sum_{i = t_1 + 1}^{t_2 - 1} \hat{x}_{\sigma_i} -
          (\pi_{\sigma_{t_2}} + a_{t - \beta}) \hat{x}_{\sigma_{t_2}} \\
      \geq & \;  - \left[ (\pi_{\sigma_{t_2}} - \pi''_{\sigma_{t_2}}) (t_2 - t_1 - 1) +
             \pi_{\sigma_{t_2}} + a_{t - \beta} \right] \hat{x}_{\sigma_{t_2}} = 0.
    \end{align*}
    Therefore, \(t_2 = \ceil{\tilde{t}}\) at optimality. It follows from
    constraint~\eqref{lp-note-1} that at optimality
    \begin{align*}
      \pi^*_{\sigma_{t_2}} = -a_{t - \beta}(\ceil{\tilde{t}} - \tilde{t}).
    \end{align*}
    Finally, suppose that \(\tilde{t}\) is an integer. Then, following a similar
    analysis as above, we obtain that either (i) \(t_1 = \tilde{t} - 1\),
    \(t_2 = \tilde{t}\), \(\pi^*_{\sigma_{t_2}} = 0\) or (ii)
    \(t_1 = \tilde{t}\), \(t_2 = \tilde{t} + 1\),
    \(\pi^*_{\sigma_{t_2}} = -a_{t - \beta}\). Both cases coincide with the
    claim of this proposition.

\item[(B)] Since \(\pi_{\sigma_t} + a_{t - \beta} > 0\), there exists a
    \(0 \leq t^{\prime} \leq t\) such that
    \begin{align*}
      0 = \pi_{\sigma_1} = \cdots = \pi_{\sigma_{t^{\prime}}} \geq \pi_{\sigma_{t^{\prime} + 1}} = \cdots = \pi_{\sigma_{\abs{\mathcal{Z}}}}.
    \end{align*}
    If \(t^{\prime} = t\) then \(\sum_{i=1}^t \pi_{\sigma_i} = 0\), which violates
    constraint~\eqref{lp-note-1}. Hence, \(t^{\prime} \leq t - 1\) and a valid choice
    of \(\pi_{\sigma_t}\) satisfies
    \begin{align*}
      - a_{t - \beta + 1} >
      \pi_{\sigma_{t}} =
      - \frac{1}{(t - t^{\prime})}\sum_{i = \beta + 1}^t a_{i - \beta}
      > - a_{t - \beta}.
    \end{align*}
    Since
    \(\sum_{i=\beta+1}^t a_{i-\beta} = f(t - \beta) = (t-\beta)a_{t-\beta} + b_{t-\beta}\)
    and
    \(\sum_{i=\beta+1}^t a_{i-\beta} = f(t - \beta) = (t-\beta)a_{t-\beta+1} + b_{t-\beta+1}\),
    these inequalities are equivalent to
    \begin{align*}
      \beta - \frac{b_{t - \beta + 1}}{a_{t - \beta + 1}} < t^{\prime} < \beta - \frac{b_{t-\beta}}{a_{t-\beta}}.
    \end{align*}
    Accordingly, at optimality \(\pi^*_{\sigma_{t^{\prime} + 1}}\) equals
    \(\displaystyle - \frac{\sum_{i=\beta+1}^t a_{i-\beta}}{(t - t^{\prime})} \equiv - \frac{f(t - \beta)}{(t - t^{\prime})}\).
\end{enumerate}
\end{proof}

\begin{proposition}
\label{prop:vi-min-max}
For all \(\pi \in \Pi^*\), it holds that
\begin{align}
  \max_{T \in \{0\}\cup [\abs{\mathcal{Z}}]} \Set{
  f \left( \posp{ T - \beta } \right) + \sum_{i=1}^{T} \pi_{\sigma_i}}
  = \min_{(\ell, \rho) \in L_1 \cup L_2} \max_{T \in \{0\}\cup [\abs{\mathcal{Z}}]} \Set{ \hfcn_{(\ell, \rho)}(T) +
  \sum_{i=1}^{T} \pi_{\sigma_i}}. \label{eqn:vi-minmax}
\end{align}
\end{proposition}

\begin{proof}
Since
\(f ( (T - \beta)^+) = \min_{(\ell, \rho) \in L_1 \cup L_2}\hfcn_{(\ell, \rho)}(T)\)
by Lemma \ref{lem:vis-frepr}, we switch the order of maximization and
minimization to obtain
\begin{align*}
  \max_{T \in \{0\}\cup [\abs{\mathcal{Z}}]} \Set{
  f \left( \posp{ T - \beta } \right) + \sum_{i=1}^{T} \pi_{\sigma_i}}
  = & \max_{T \in \{0\}\cup [\abs{\mathcal{Z}}]} \Set{
  \min_{(\ell, \rho) \in L_1 \cup L_2}\hfcn_{(\ell, \rho)}(T) + \sum_{i=1}^{T} \pi_{\sigma_i}} \\
  \leq & \min_{(\ell, \rho) \in L_1 \cup L_2} \Set{\max_{T \in \{0\}\cup [\abs{\mathcal{Z}}]} \hfcn_{(\ell, \rho)}(T) +
  \sum_{i=1}^{T} \pi_{\sigma_i}}.
\end{align*}

It remains to prove the other direction of \eqref{eqn:vi-minmax}.
\(\pi \in \Pi^*\) implies that \(\pi \in \Pi^*_t\) for a
\(t \in \integers{}, \beta + 1 \leq t \leq \abs{\mathcal{Z}}\). We discuss the
following cases to finish the proof.
\begin{enumerate}
\item If \(b_{t - \beta} - \beta a_{t - \beta} \geq 0\), then by definition of
    \(\Pi^*_t\)
    \begin{align*}
      \pi_{\sigma_1} = \pi_{\sigma_2} = \cdots = \pi_{\sigma_{\abs{\mathcal{Z}}}} = -a_{t - \beta}.
    \end{align*}
    It follows that
    \begin{align*}
      \min_{(\ell, \rho) \in L_1 \cup L_2} \max_{T \in \{0\}\cup [\abs{\mathcal{Z}}]} \Set{\hfcn_{(\ell, \rho)}(T) +
      \sum_{i=1}^{T} \pi_{\sigma_i}}
      & \leq \max_{T \in \{0\}\cup [\abs{\mathcal{Z}}]} \Set{ \posp{\fcnl_{t - \beta}(T - \beta)} + \sum_{i = 1}^T \pi_{\sigma_i} } \\
      & = \max_{T \in \{0\}\cup [\abs{\mathcal{Z}}]} \Set{ a_{t - \beta} (T - \beta) + b_{t - \beta} - T a_{t - \beta} } \\
      & = \sum_{i = 1}^{\beta} (- a_{t - \beta}) + \sum_{i = \beta + 1}^t (a_{i - \beta} - a_{t - \beta}) \\
      & = \sum_{i = 1}^{\beta} \pi_{\sigma_i} + \sum_{i = \beta + 1}^t (a_{i - \beta} + \pi_{\sigma_i}) \\
      & = \max_{T \in \{0\}\cup [\abs{\mathcal{Z}}]} \Set{
        f \left( \posp{ T - \beta } \right) + \sum_{i=1}^{T} \pi_{\sigma_i}},
    \end{align*}
    where the inequality is because we select \((\ell, \rho) \in L_1\) with
    \(\ell = t - \beta\) and \(\rho = b_{t-\beta}/a_{t-\beta}\), the first
    equality is because \(b_{t-\beta} - \beta a_{t-\beta} \geq 0\), the second
    equality is because
    \(b_{t-\beta} - \beta a_{t - \beta} = b_{t-\beta} + (t-\beta) a_{t - \beta} - ta_{t - \beta} = \sum_{i=\beta+1}^t a_{i-\beta} - ta_{t - \beta}\),
    and the last equality is because \(\pi \in \Pi^*_t \subseteq \Pi_t\).
\item If \(b_{t - \beta} - \beta a_{t - \beta} < 0\) then we discuss the
    following two cases.
\begin{enumerate}
\item[(A)] If
    \(0 = \pi_{\sigma_1} = \cdots = \pi_{\sigma_{\floor{\tilde{t}}}} \geq \pi_{\sigma_{\ceil{\tilde{t}}}} = -a_{t - \beta} (\ceil{\tilde{t}} - \tilde{t}) > \pi_{\sigma_{\ceil{\tilde{t}}+1}} = \cdots = \pi_{\sigma_{\abs{\mathcal{Z}}}} = -a_{t - \beta}\)
    for \(\tilde{t} \equiv \beta - \frac{b_{t - \beta}}{a_{t - \beta}}\), then
    \begin{align*}
      \min_{(\ell, \rho) \in L_1 \cup L_2} \max_{T \in \{0\}\cup [\abs{\mathcal{Z}}]} \Set{ \hfcn_{(\ell, \rho)}(T) +
      \sum_{i=1}^{T} \pi_{\sigma_i} }
      & \leq \max_{T \in \{0\}\cup [\abs{\mathcal{Z}}]}
        \Set{ \posp{a_{t - \beta} (T - \beta) + b_{t - \beta}} + \sum_{i = 1}^T \pi_{\sigma_i} } \\
      & = 0 \\
      & = \max \Set{\sum_{i = 1}^{\beta} \pi_{\sigma_i} +
        \sum_{i = \beta + 1}^t (a_{i - \beta} + \pi_{\sigma_i}), 0} \\
      & = \max_{T \in \{0\}\cup [\abs{\mathcal{Z}}]} \Set{
        f \left( \posp{ T - \beta } \right) + \sum_{i=1}^{T} \pi_{\sigma_i}},
    \end{align*}
    where the inequality is because we select \((\ell, \rho) \in L_1\) with
    \(\ell = t - \beta\) and \(\rho = b_{t-\beta}/a_{t-\beta}\), the first
    equality is because
    \begin{align*}
      \posp{a_{t - \beta} (T - \beta) + b_{t - \beta}}
      & = \posp{Ta_{t-\beta} + b_{t-\beta} - \beta a_{t-\beta}} \\
      & = \posp{T - \tilde{t}}a_{t-\beta} \ = \ - \sum_{i=1}^T \pi_{\sigma_i}
    \end{align*}
    for all \(T \in \{0\} \cup [|\mathcal{Z}|]\), the second equality is because
    \(\sum_{i = \beta + 1}^t a_{i-\beta} = (t - \tilde{t})a_{t-\beta}\), and the
    last equality is because \(\pi \in \Pi^*_t \subseteq \Pi_t\).
\item[(B)] If
    \(\displaystyle 0 = \pi_{\sigma_1} = \cdots = \pi_{\sigma_{t^{\prime}}} > \pi_{\sigma_{t^{\prime} + 1}} = \cdots \pi_{\sigma_{\abs{\mathcal{Z}}}} = -\frac{f(t - \beta)}{(t - t^{\prime})}\)
    for a
    \(t^{\prime} \in \Bigl( \beta - \frac{b_{t - \beta + 1}}{a_{t - \beta + 1}}, \beta - \frac{b_{t - \beta}}{a_{t - \beta}} \Bigr) \cap \mathbb{Z}_+\),
    then
    \begin{align*}
      \min_{(\ell, \rho) \in L_1 \cup L_2} \max_{T \in \{0\}\cup [\abs{\mathcal{Z}}]} \Set{\hfcn_{(\ell, \rho)}(T) +
      \sum_{i=1}^{T} \pi_{\sigma_i}}
      & \leq \max_{T \in \{0\}\cup [\abs{\mathcal{Z}}]} \Set{ \hfcn_{(t - \beta, \beta - t^{\prime})}(T) + \sum_{i = 1}^T \pi_{\sigma_i} } \\
      & = \max_{T \in \{0\}\cup [\abs{\mathcal{Z}}]} \Set{
        \posp{\frac{f(t - \beta)}{t - t^{\prime}} (T - t^{\prime})} + \sum_{i = 1}^T \pi_{\sigma_i} } \\
      & = 0 \\
      & = \max \Set{\sum_{i = 1}^{\beta} \pi_{\sigma_i} +
        \sum_{i = \beta + 1}^t (a_{i - \beta} + \pi_{\sigma_i}), 0} \\
      & = \max_{T \in \{0\}\cup [\abs{\mathcal{Z}}]} \Set{
        f \left( \posp{ T - \beta } \right) + \sum_{i=1}^{T} \pi_{\sigma_i}},
    \end{align*}
    where the inequality is because we select \((\ell, \rho) \in L_2\) with
    \(\ell = t - \beta\) and \(\rho = \beta - t^{\prime}\), the second equality is due
    to the definition of \(\pi\), the third equality is because
    \(\sum_{i=\beta+1}^t a_{i-\beta} = f(t-\beta)\), and the last equality is
    because \(\pi \in \Pi^*_t \subseteq \Pi_t\).
\end{enumerate}
\end{enumerate}
\end{proof}


\subsection{Valid inequalities from cross scenarios}%
\label{sec:valid_ineqs_cross}
{\color{black}Using the fact that \(x^{\top}\xi\) is an integer,} we study valid inequalities arising from multiple scenarios \(j \in [N]\) and
multiple coverings \(i \in [I]\). To this end, we focus on a set \(\mathcal{S}\)
consisting of \(J\) ``base'' inequalities and other domain restrictions:
\begin{align}
  \setcross{} := \Set{ (x, \gamma, z) \in \binaries^n\times\reals^{1+J} \colon
  \begin{aligned}
  & - z_j - \gamma \geq d^j \left( x^{\top} \bar{\xi}^j \right) + d^j_0,
  \quad \forall j \in [J] \\
  & z_j \leq 0, \quad \forall j \in [J]
  \end{aligned}}, \label{eqn:setcross-base}
\end{align}
where \(d^j_0 \in \reals\) and \(d^j < 0\) represent constant parameters. The
base inequalities can come from two sources. First, the Benders feasibility cut
\eqref{eqn:reform-feasi-base} implies that
\begin{align*}
  - z_j - \gamma
  & \geq \phi^j(\varnothing) +
    \sum_{k = 1}^n \left[ \phi^j(\mathcal{T}_k) - \phi^j(\mathcal{T}_{k - 1}) \right] x_{\sigma_k} \\
  & = \phi^j(\varnothing) + \sum_{k:\phi^j(\mathcal{T}_{k - 1}) < 0} (\ones^{\top}\hat{\mu}_1)x_{\sigma_k} +
    \sum_{\substack{k:\phi^j(\mathcal{T}_{k - 1}) = 0,\\ \phi^j(\mathcal{T}_k) < 0}} \phi^j(\mathcal{T}_k) x_{\sigma_k} +
  \sum_{k:\phi^j(\mathcal{T}_k) = 0} 0 \cdot x_{\sigma_k} \\
  & \geq \phi^j(\varnothing) + (\ones^{\top}\hat{\mu}_1) x^{\top} \bar{\xi}^j,
\end{align*}
where \(\bar{\xi}^j\) denotes a binary vector, in which
\(\bar{\xi}^j_{\sigma_k}=1\) if and only if \(\phi^j(\mathcal{T}_k) < 0\), and
the last inequality is because
\(\phi^j(\mathcal{T}_k) \geq \ones^{\top}\hat{\mu}_1\) if
\(\phi^j(\mathcal{T}_{k - 1}) = 0\) and \(\phi^j(\mathcal{T}_k) < 0\). As a
result, \eqref{eqn:reform-feasi-base} generates a base inequality with
\(d^j_0 = \phi^j(\varnothing)\) and \(d^j = \ones^{\top}\hat{\mu}_1 < 0\).
Second, the single-scenario valid inequality \eqref{eqn:cut-single} implies
\begin{align*}
  - z_j - \gamma
  & \geq \Phi^j(\varnothing) +
    \sum_{k = 1}^n \left[ \Phi^j(\mathcal{T}_k) - \Phi^j(\mathcal{T}_{k - 1}) \right] x_{\sigma_k} \\
  & = \Phi^j(\varnothing) + \sum_{k:\Phi^j(\mathcal{T}_{k - 1}) < 0} \left(- \frac{\ell^{1/p}}{\ell+\rho}\right)x_{\sigma_k} +
    \sum_{\substack{k:\Phi^j(\mathcal{T}_{k - 1}) = 0,\\ \phi^j(\mathcal{T}_k) < 0}} \Phi^j(\mathcal{T}_k) x_{\sigma_k} +
  \sum_{k:\Phi^j(\mathcal{T}_k) = 0} 0 \cdot x_{\sigma_k} \\
  & \geq \Phi^j(\varnothing) + \left(- \frac{\ell^{1/p}}{\ell+\rho}\right) x^{\top} \bar{\xi}^j,
\end{align*}
where \(\bar{\xi}^j\) denotes a binary vector, in which
\(\bar{\xi}^j_{\sigma_k}=1\) if and only if \(\Phi^j(\mathcal{T}_k) < 0\). As a
result, \eqref{eqn:cut-single} generates a base inequality with
\(d^j_0 = \Phi^j(\varnothing)\) and \(d^j = -\ell^{1/p}/(\ell+\rho) < 0\). We
remark that the number of base inequalities, \(J\), need not to be the same as
the number of scenarios, \(N\). This is because a single scenario can generate
multiple base inequalities with regard to different coverings.

We employ the mixing scheme~\cite{guenluek-2001-mixin-mixed} to derive valid
inequalities for \(\setcross\). This scheme produces a new inequality by
combining existing ones. Specifically, consider the following mixing set \(\mathcal{MIX}\) consisting of \(J\) base inequalities:
\begin{align*}
\mathcal{MIX} := \Set{x \in \binaries^n\colon
\begin{aligned}
& f^j(x) + B g^j(x) \geq u^j, \quad \forall j \in [J] \\
& f^j(x) \geq 0, \quad \forall j \in [J]
\end{aligned}},
\end{align*}
where \(B \in \reals_{+}, u^j \in \reals\), and \(g^j(x) \in \integers{}\).
Then,~\cite{guenluek-2001-mixin-mixed} derived the following mixed inequality.
\begin{theorem}[Theorem~\(2\) in~\cite{guenluek-2001-mixin-mixed}]
\label{thm:mixing}
The following mixed inequality is valid for \(\mathcal{MIX}\):
\begin{gather*}
\max_{j \in [J]}\{f^j(x)\} \geq \sum_{j = 1}^J (\nu^{j} - \nu^{j-1}) (\tau^{j} - g^j(x)),
\end{gather*}
where \(\nu^{0} := 0\), \(\tau^j := \ceil{u^j / B}\), \(\nu^j := u^j - B(\tau^{j} - 1)\) for all \(j \in [J]\), and \(\nu^j\)'s are sorted so that \(\nu^j \geq \nu^{j - 1}\) for all \(j \in [J]\).\qed
\end{theorem}
To apply the mixing scheme, we divide both sides of the base inequalities in
\eqref{eqn:setcross-base} by \(-d^j\) to obtain
\[
\frac{z_j}{d^j} + x^{\top}\bar{\xi}^j \geq - \frac{d^j_0+\gamma}{d^j}, \quad \forall j \in [J].
\]
This would obey the structure of \(\mathcal{MIX}\) if \(-(d^j_0+\gamma)/d^j\)
(counterpart of \(u^j\)) were a constant, or equivalently, if variable
\(\gamma\) were a constant. Although this is not true in general, we observe
that, without loss of optimality, \(\gamma\) can take only a finite set of
values. \edits{It is worth mentioning that a similar result appeared in~\cite[Lemma~\(1\)]{ho-nguyen-2020-distr-robus}.}
\begin{proposition}%
\label{prop:reform-bin-encoding}
There exists an optimal solution \((x^*, \gamma^{*}, z^{*})\) to
(\DRCCbC{}), \edits{using the reformulation $Z$ defined
  in~\eqref{eqn:reform-mono},} such that
\begin{align*}
\gamma^* & = - \VaR{}_{1 - \epsilon}\left[-g(x^{*}, \hat{\xi})\right] \\
\edits{\mathrm{and} \quad z^*_j} & \edits{= \min \Set{0, \, g(x^*, \hat{\xi}^{j}) - \gamma^{*}}, \quad \forall j \in [N],}
\end{align*}
where \(g(x, \xi) = \min_{i \in [I]}\Bigl((x^{\top}\xi_i - v_i + 1)^+\Bigr)^{1/p}\).
\end{proposition}
\begin{proof}
\editsII{
By Proposition~\ref{prop:reform-mono} and the definition~\eqref{eqn:cvar-def} of~\CVaR{}, we recast the set \(Z\) as}
\begin{align*}
  Z = \Set{x \in \binaries^n: \frac{\delta}{\epsilon} +
  \min_{\gamma'} \Bigl\{ \gamma' + \frac{1}{N\epsilon} \sum_{j \in [N]} \posp{-g(x, \hat{\xi}^j) - \gamma'}\Bigr\} \leq 0}.
\end{align*}
Then, by the definition~\eqref{eqn:cvar-def} of \CVaR{}, \(\gamma'\)
equals \(\VaR{}_{1 - \epsilon}\Bigl[-g(x^{*}, \hat{\xi})\Bigr]\) at
optimality. The conclusion follows.
\end{proof}

By Proposition~\ref{prop:reform-bin-encoding}, we can restrict the value of
\(\gamma\) to be in a discrete set \(\{r \in \reals: r^p \in \{0\} \cup [n]\}\)
because \((x^*)^{\top}\hat{\xi}_i\) and \(v_i\) are integers for all
\(i \in [I]\). In addition, since \(g(x, \hat{\xi}) \leq g(\ones, \hat{\xi})\)
almost surely for all \(x \in \binaries^n\) and (\DRCCbC{}) requires that
\(\CVaR{}_{1 - \epsilon}\bigl[-g(x, \hat{\xi})\bigr] \leq -\delta/\epsilon\), we
can further restrict the value of \(\gamma\) to be that
\begin{align*}
  & \gamma \in \Set{r \colon
    \ceil[\Bigg]{\left( \frac{\delta}{\epsilon} \right)^{p}}^{1/p} \leq r \leq - \VaR{}_{1 - \epsilon}(-g(\ones{}, \hat{\xi})),
    \ r^{p} \in [n]}.
\end{align*}
For ease of exposition, we denote this set as \(\{r_k: k \in \Gamma\}\), where
\(\Gamma\) is the set of indices, and let \(r_1\) be the smallest element. It
follows that \(\gamma \geq r_1\) at optimality and the base inequalities
in~\eqref{eqn:setcross-base} imply
\[
\frac{z_j}{d^j} + x^{\top}\bar{\xi}^j \geq - \frac{d^j_0+r_1}{d^j}, \quad \forall j \in [J].
\]
Then, applying the mixing scheme in Theorem~\ref{thm:mixing} produces a mixed
inequality for \(\mathcal{S}\):
\begin{equation}
  \max_{j \in [J]} \Set{\frac{z_j}{d^j}}
  \geq \sum_{j = 1}^J (\nu^j_1 - \nu^{j - 1}_1) (\tau^j_1 - x^{\top}\bar{\xi}^j), \label{eqn:mix-weak}
\end{equation}
where \(\displaystyle \tau^j_1 := \ceil[\Big]{-\frac{r_1 + d^j_0}{d^j}}\),
\(\nu^0_1 := 0\), and
\(\displaystyle \nu^j_1 := - \frac{r_1 + d^j_0}{d^j} - (\tau^j_1 - 1)\) such
that \(\nu^j_1 \geq \nu^{j-1}_1\) for all \(j \in [J]\). Nevertheless, the mixed
inequality~\eqref{eqn:mix-weak} is potentially weak because it simply relaxes
\(\gamma\) to be \(r_1\) and ignores the possibility of \(\gamma\) taking other
values in \(\{r_k: k \in \Gamma\}\). To strengthen it, we notice that
\(\gamma\), which equals one and only one \(r_k\) for \(k \in \Gamma\), admits a
binary encoding:
\begin{align*}
\gamma = \sum_{k \in \Gamma} r_k \gamma_k,
\quad{} \sum_{k \in \Gamma} \gamma_k = 1,
\end{align*}
where \(\{\gamma_{k} \in \binaries: k \in \Gamma\}\) are (auxiliary) binary
variables that determine the value of \(\gamma\). Accordingly, we recast set
\(\mathcal{S}\) as
\begin{align*}
  \setcross{}_{\mathbb{B}} := \left\{ (x, \gamma, z) \in \binaries^{n + |\Gamma|}\times\reals^J_- \colon \
  \begin{aligned}
  & - z_j - \sum_{k \in \Gamma} r_k \gamma_k \geq d^j \left( x^{\top} \bar{\xi}^j \right) + d^j_0,
  \quad \forall j \in [J] \\
  & \sum_{k \in \Gamma} \gamma_k = 1
  \end{aligned}\right\}.
\end{align*}
As a consequence, the mixed inequality~\eqref{eqn:mix-weak} is valid for a
restriction of \(\setcross{}_{\mathbb{B}}\) when fixing \(\gamma_1 = 1\), i.e.,
\(\setcross{}_{\mathbb{B}} \cap \{(x, \gamma, z): \gamma_1 = 1\}\). Hence, we
can view~\eqref{eqn:mix-weak} as a \emph{seed} inequality and strengthen it by
\emph{lifting} variables \(\gamma_2, \ldots, \gamma_{|\Gamma|}\) back (see,
e.g.,~\cite{wolsey-2014-integ,gu-2000-sequen-indep} for more on the lifting
scheme).
\begin{theorem} \label{thm:lift}
The following lifted inequality is valid for \(\mathcal{S}_{\mathbb{B}}\):
\begin{gather}
  \max_{j \in [J]} \Set{\frac{z_j}{d^j}}
  \geq \sum_{j = 1}^J (\nu^j_1 - \nu^{j - 1}_1) (\tau^j_1 - x^{\top}\bar{\xi}^j) + \sum_{k \in \Gamma} \alpha_k \gamma_k, \label{eqn:cut-cross} \\
  \mathrm{where} \quad \alpha_k
  := \posp{\min_{j \in [J]} \min \set{\nu^j_k - \nu^J_1, 0} +
  \sum_{j = 1}^J (\nu^{j}_1 - \nu^{j - 1}_1)(\tau^{j}_k - \tau^j_1)}, \quad \forall k \in \Gamma, \nonumber
\end{gather}
\(\displaystyle \tau^j_k := \ceil[\Big]{-\frac{r_k + d^j_0}{d^j}}\), \(\nu^0_1 = 0\), and 
\(\displaystyle \nu^j_k := -\frac{r_k + d^j_0}{d^j} - (\tau^j_k - 1)\) such that \(\nu^j_1 \geq \nu^{j-1}_1\) for all \(j \in [J]\).
\end{theorem}
\begin{proof}
For all \(k \in \Gamma\), let
\(\mathcal{S}_k := \mathcal{S}_{\mathbb{B}} \cap \{(x, \gamma, z): \gamma = r_k\}\).
We simultaneously lift variables \(\gamma_2, \ldots, \gamma_{|\Gamma|}\) back
into inequality~\eqref{eqn:mix-weak}. To this end, we search for lifting
coefficients \(\alpha_2, \ldots, \alpha_{|\Gamma|}\) such that
\begin{align*}
  \max_{j \in [J]} \Set{\frac{z_j}{d^j}}
  & \geq \sum_{j = 1}^J (
  \nu^j_1 - \nu^{j - 1}_1)(\tau^j_1 - x^{\top} \bar{\xi}^j) + \sum_{k \in \Gamma} \alpha_k \gamma_k,
  \quad \forall (x, \gamma, z) \in \mathcal{S}_{\mathbb{B}},
\end{align*}
where \(\alpha_1 := 0\). Since \(\sum_{k \in \Gamma}\gamma_k = 1\), fixing
\(\gamma_k = 1\) automatically forces all other \(\gamma_k\)'s to be zero.
Hence, the lifted inequality is valid if and only if
\begin{align*}
  \max_{j \in [J]} \Set{\frac{z_j}{d^j}}
  & \geq \sum_{j = 1}^J (
  \nu^j_1 - \nu^{j - 1}_1)(\tau^j_1 - x^{\top} \bar{\xi}^j) +  \alpha_k \gamma_k,
  \quad \forall (x, \gamma, z) \in \mathcal{S}_k, \ \forall k \in \Gamma.
\end{align*}
If \(k = 1\), this validity requirement is valid because \(\alpha_1 = 0\) and
the inequality reduces to~\eqref{eqn:mix-weak}. If
\(k \in \Gamma \setminus \{1\}\), the validity requirement is valid if and
only if
\begin{align}
  \alpha_k \leq \min_{x, z} ~~
  & \max_{j \in [J]} \Set{\frac{z_j}{d^j}} - \sum_{j = 1}^J
    (\nu^j_1 - \nu^{j - 1}_1)(\tau^j_1 - x^{\top} \bar{\xi}^j), \nonumber \\
  \st{} ~~
  & \frac{z_j}{d^j} \geq -\frac{r_k + d^j_0}{d^j} - x^{\top} \bar{\xi}^j, \quad \forall j \in [J], \tag{Lifting} \label{eqn:lifting} \\
  & x \in \binaries^n, z \in \reals^J_-. \nonumber
\end{align}
For fixed \(x \in \binaries^n\), \(z_j \leq 0\) and \(d^j < 0\) imply
\begin{align*}
  \frac{z_j}{d^j}
  \geq \posp{-\frac{r_k + d^j_0}{d^j} - x^{\top} \bar{\xi}^j}
  & = \posp{\nu^j_k + \tau^j_k - x^{\top} \bar{\xi}^j - 1} \\
  & = \nu^j_k \posp{\tau^j_k - x^{\top}\bar{\xi}^j} + (1 - \nu^j_k)\posp{\tau^j_k - x^{\top}\bar{\xi}^j - 1}
\end{align*}
for all \(j \in [J]\). Let
\(\Upsilon_k := \max \Set{\tau^j_k - x^{\top}\bar{\xi}^j \colon j \in [J]}\),
\(\mathcal{J}_k\) be the corresponding set of maximizers, and
\(i \in \mathcal{J}_k\) be a maximizer with the largest \(\nu^j_k\), i.e.,
\begin{align*}
  \mathcal{J}_k
  & := \Set{j \in [J] \colon \tau^j_k - x^{\top}\bar{\xi}^j = \Upsilon_k} \\
  \mathrm{and} \quad i
  & \in \argmax \Set{\nu^j_k \colon j \in \mathcal{J}_k}.
\end{align*}
Then, it is easy to verify that
\begin{align*}
  \max_{j \in [J]} \Set{\frac{z_j}{d^j}}
  = \nu^i_k \posp{\tau^i_k - x^{\top} \bar{\xi}^i} +
  (1 - \nu^i_k) \posp{\tau^i_k - x^{\top} \bar{\xi}^i - 1}
\end{align*}
at optimality of the~\eqref{eqn:lifting} problem. Specifically, if
\(\Upsilon_k \leq 0\), then \((\tau^j_k - x^{\top}\bar{\xi}^j) \leq 0\) for all
\(j \in [J]\) and so \(\max_{j \in [J]} \Set{z_j/d^j} = 0\). This implies a
lower bound for the objective value of~\eqref{eqn:lifting}:
\begin{align*}
  \max_{j \in [J]} \Set{\frac{z_j}{d^j}} - \sum_{j = 1}^J (
  \nu^j_1 - \nu^{j - 1}_1)(\tau^j_1 - x^{\top} \bar{\xi}^j)
  & = \sum_{j = 1}^J (
    \nu^j_1 - \nu^{j - 1}_1)(x^{\top} \bar{\xi}^j - \tau^j_1) \\
  & \geq \sum_{j = 1}^J (\nu^j_1 - \nu^{j - 1}_1) (\tau^j_k - \tau^j_1).
\end{align*}
If \(\Upsilon_k \geq 1\) then
\(\max_{j \in [J]} \Set{z_j/d^j} = (\tau^i_k - x^{\top}\bar{\xi}^i - 1) + \nu^i_k\),
and the lower bound for the objective value of~\eqref{eqn:lifting} becomes
\begin{align*}
  & \max_{j \in [J]} \Set{\frac{z_j}{d^j}} - \sum_{j = 1}^J (
  \nu^j_1 - \nu^{j - 1}_1)(\tau^j_1 - x^{\top} \bar{\xi}^j) \\
  = \;
  & (\tau^i_k - x^{\top}\bar{\xi}^i - 1) + \nu^i_k - \sum_{j = 1}^J (\nu^j_1 - \nu^{j - 1}_1)(\tau^j_k - x^{\top} \bar{\xi}^j) -
    \sum_{j = 1}^J (\nu^j_1 - \nu^{j - 1}_1)(\tau^j_1 - \tau^j_k) \\
  \geq \;
  & (\tau^i_k - x^{\top}\bar{\xi}^i - 1) + \nu^i_k -
    \sum_{j \in \mathcal{J}_k} (\nu^j_1 - \nu^{j - 1}_1)(\tau^i_k - x^{\top} \bar{\xi}^i) -
    \sum_{j \not \in \mathcal{J}_k} (\nu^j_1 - \nu^{j - 1}_1)(\tau^i_k - x^{\top} \bar{\xi}^i - 1) \\
  & -\sum_{j = 1}^J (\nu^j_1 - \nu^{j - 1}_1) (\tau^j_1 - \tau^j_k) \\
  = \;
  & (\tau^i_k - x^{\top}\bar{\xi}^i - 1) (1 - \nu^J_1) +
    \nu^i_k - \sum_{j \in \mathcal{J}_k} (\nu^j_1 - \nu^{j - 1}_1) +
    \sum_{j = 1}^J (\nu^j_1 - \nu^{j - 1}_1)(\tau^j_k - \tau^j_1) \\
  \geq \;
  & \nu^i_k - \nu^J_1 + \sum_{j = 1}^J (\nu^{j}_1 - \nu^{j - 1}_1)(\tau^{j}_k - \tau^j_1),
\end{align*}
where the first inequality is because
\(\tau^j_k - x^{\top}\bar{\xi}^j \leq \tau^i_k - x^{\top}\bar{\xi}^i - 1\) if
\(j \notin \mathcal{J}_k\) and the second inequality is because
\(\tau^i_k - x^{\top}\bar{\xi}^i = \Upsilon_k \geq 1\), \(\nu^J_1 \leq 1\), and
\(\sum_{j \in \mathcal{J}_k} (\nu^j_1 - \nu^{j - 1}_1) \leq \sum_{j = 1}^J (\nu^j_1 - \nu^{j - 1}_1) = \nu^J_1\).
The two cases of the value of \(\Upsilon_k\) together suggests the following
valid choice of \(\alpha_k\):
\begin{align*}
  \alpha_k = \min_{j \in [J]} \min \set{\nu^j_k - \nu^J_1, 0} +
  \sum_{j = 1}^J (\nu^{j}_1 - \nu^{j - 1}_1)(\tau^{j}_k - \tau^j_1).
\end{align*}
But \(\alpha_k = 0\) is also a valid choice because, in that case, the lifted
inequality reduces to \eqref{eqn:mix-weak}, which is valid for
\(\mathcal{S}_{\mathbb{B}} \supseteq \mathcal{S}_k\). Therefore, we can take the
maximum of these two valid choices. The conclusion follows.
\end{proof}
We remark that, thanks to constraint \(\sum_{k \in \Gamma}\gamma_k = 1\), the
lifted inequality in Theorem \ref{thm:lift} is sequence-\emph{independent}
(cf.~\cite{gu-2000-sequen-indep}). In our implementation, we generate the base
inequalities in~\eqref{eqn:setcross-base} from the single-scenario valid
inequality~\eqref{eqn:cut-single}. We add the lifted inequalities via the lazy
callback.


\section{Extension to Knapsack Chance Constraint}%
\label{sec:extension-knaps}
\edits{
In this section, we extend the reformulation and valid inequalities for (\DRCCbC{}) to knapsack chance constraints with
random, binary-valued coefficients. Specifically, consider the following
ambiguous knapsack chance constraint
\begin{align}
  \inf_{\mathbb{P} \in \mathcal{P}} \mathbb{P}\Set{
  x^{\top}\tilde{\xi}_i \leq v_i, \ \forall i \in [I]}
  \geq 1 - \epsilon, \tag{KCC} \label{eq:knapsack-drcc}
\end{align}
where the Wasserstein ambiguity set $\mathcal{P}$ is defined in~\eqref{eq:intro-wass-ambiguity},
$\tilde{\xi} \in \Xi = \{0, 1\}^{I \times n}$, and $v_i \in [n-1]$ for all
$i \in [I]$. In this case, we can follow the proof of Proposition~\ref{prop:reform-mono} to show that
\begin{gather*}
\min_{\xi: \ x^{\top} \xi_i \geq v_i + 1} \norm{\xi - \hat{\xi}^j}_p =
\left( \posp{v_i + 1 - x^{\top} \hat{\xi}^j_i} \right)^{1/p} +
\chi \left( x^{\top} \ones{} \geq v_i + 1 \right),
\end{gather*}
where \(\chi(\cdot)\) denotes a characteristic function, which takes value \(0\)
if \(x^{\top}\ones{} \geq v_i + 1\) and takes value \(+\infty\) otherwise.
Hence, if we assume $v_1 \leq v_2 \leq \cdots \leq v_I$ \WLOG{} then (KCC) is equivalent to \(\bigcup_{k=0}^I Z_k\), where
$Z_0 := \{x \in \{0, 1\}^n: x^{\top}\ones{} \leq v_1\}$ and, for all
$k \in [I]$, $Z_k$ resembles the reformulation $Z$ for (DRC) in Proposition~\ref{prop:reform-mono}. That is,
\begin{align*}
Z_k := \bigl\{x \in \binaries^n \colon v_k + 1 \leq x^{\top}\ones{} \leq v_{k+1}\bigr\} \cap \Set{x \colon
  \begin{aligned}
    & \exists \ \gamma \in \reals_+, z \in \reals^N_-:\\
    & \delta - \gamma \epsilon \leq \frac{1}{N} \sum_{j \in [N]} z_j, \\
    & z_j + \gamma \leq \min_{i \in [k]} \left( \posp{v_i + 1 - x^{\top} \hat{\xi}^j_i} \right)^{1/p},
    \forall j \in [N]
  \end{aligned}},
\end{align*}
where $v_{I+1} := n$. As a consequence, we can readily adapt the valid inequalities in Section~\ref{sec:valid_ineqs} to solve optimization problems involving (KCC). For example, under the complement $y := \ones{} - x$,
the final constraints in the definition of $Z_k$ can be rewritten as
\begin{align*}
z_j + \gamma \leq \left( \posp{y^{\top}\hat{\xi}^j_i - \ones^{\top}\hat{\xi}^j_i + v_i + 1} \right)^{1/p},
\end{align*}
which is precisely the setting studied in Section~\ref{sec:valid_ineqs}
if \((\ones^{\top}\hat{\xi}^j_i - v_i - 1) \geq 1\). Hence, the valid inequalities \eqref{eqn:cut-single} and \eqref{eqn:cut-cross} can be adapted to compute (KCC) more
effectively. On the other hand, when \(\beta' := (\ones^{\top}\hat{\xi}^j_i - v_i - 1) \leq 0\), it can be shown that the convex hull of the mixed-integer set
\begin{align*}
  \mathcal{X}' := \Set{
  (\theta, x) \in \reals \times \binaries^n \colon
  \theta \leq f(x^{\top}\xi - \beta')}
\end{align*}
coincides with the continuous relaxation of its piecewise linear reformulation, i.e.,
\begin{align*}
  \Co(\mathcal{X}') = \Set{
  (\theta, x) \in \reals \times [0,1]^n \colon
  \theta \leq a_{\ell} (x^{\top}\xi - \beta') + b_{\ell}, \,
  \forall \ell \in \set{1-\beta', \ldots, n-\beta'}}.
\end{align*}
Finally, we notice
that a special case of (KCC) is an ambiguous set packing chance constraint,
in which $v_i = 1$ for all $i \in [I]$. Hence, our reformulation and valid inequalities
also apply to set packing chance constraints with random set membership.


}

\section{Numerical Experiments}%
\label{sec:exps}
We demonstrate the out-of-sample performance of (\DRCCbC{}) and the
effectiveness of our two-stage reformulation and valid inequalities in the following four numerical experiments:
\begin{enumerate}[label={\arabic*.}]
\item Section~\ref{sec:exps-oos} reports the out-of-sample performance of (\DRCCbC{}) with various radii $\delta$ and data sizes $N$ when the entries of $\tilde{\xi}$ are i.i.d.;
\item Section~\ref{sec:out-of-sample-case-study} examines correlated $\tilde{\xi}$ entries, using a context of medical facility location under natural disasters;
\item Section~\ref{sec:value-support} evaluates the value of incorporating the support information $\Xi = \binaries{}^{I\times n}$ into (\DRCCbC{}) as opposed to relaxing it to be $\Xi = \mathbb{R}^{I\times n}$;
\item Section~\ref{sec:exps-effi} demonstrates the strength of the valid inequalities derived in Section~\ref{sec:valid_ineqs}.
\end{enumerate}
\edits{All experiments in this section are conducted using the Python API of Gurobi
\(9.0.2\) on a single core of an Intel Xeon Gold \(6154\) Processor provided by
\href{https://arc-ts.umich.edu/greatlakes/configuration/}{the UM Great Lakes cluster.}}

\subsection{Independent uncertain parameters}%
\label{sec:exps-oos}

In this experiment, we generate
\emph{testing} data of \(\txi\) from the Bernoulli distribution with \iid{}
entries, and the probability of success \(q_i := \Probt{}\Set{\txi_{ij} = 1}\)
is uniformly sampled from the interval \([0.4, 0.8]\). In the out-of-sample
evaluation, this specific choice of \(\Probt{}\) facilitates the feasibility
checking of a given solution \(x \in \binaries^n\) because
\begin{align}
  \Probt{} \Set{x^{\top} \txi_i \geq v_i, \, \forall i \in [I]}
  & = \prod_{i = 1}^I \Probt{} \Set{x^{\top} \txi_i \geq v_i} \nonumber{}\\
  & = \prod_{i = 1}^I \sum_{s = v_i}^{n} \Probt{} \Set{ \sum_{k \colon x_k = 1} \txi_{ik} = s} \nonumber{}\\
  & = \prod_{i = 1}^I \sum_{s = v_i}^{n} \binom{\ones^{\top} x}{s} q_i^s (1 - q_i)^{(\ones^{\top} x) - s}
    \label{eqn:exps-oos}
\end{align}
admits a closed-form expression. In addition, we generate the \emph{training}
historical data \(\set{\hat{\xi}_j}_{j=1}^N\) by adding Gaussian noise to
\(\txi\):
\begin{align*}
\hat{\xi}_{ik} := \mbox{round}(\txi_{ik} + 0.25 \cdot \tilde{e}), \quad \forall i \in [I], k \in [n], 
\end{align*}
where \(\tilde{e}\) is a standard Gaussian random variable and
\(\mbox{round}(\cdot)\) rounds real numbers to \(\binaries\). For a given
solution \(x \in \binaries^n\), checking whether it is feasible for chance
constraint~\eqref{eqn:ccbc-constr} simplifies to computing its out-of-sample performance
using~\eqref{eqn:exps-oos} and comparing with the risk level \(1 - \epsilon\). We compare the out-of-sample performance of (\DRCCbC{}) with the sample average approximation formulation:
\begin{align*}
  \min ~~
  & c^{\top}x, \\
  \st{} ~~
  & \frac{1}{N} \sum_{j = 1}^N z_j \geq 1 - \epsilon, \tag{\SAA{}}\\
  & x^{\top} \hat{\xi}^j_i - v_i \geq v_m (z_j - 1), \quad \forall i \in [I], j \in [N],\\
  & x \in \binaries^n, z \in \binaries^N.
\end{align*}
Random test instances with \(n = 30, I = 10\) are generated in the
aforementioned way. For each instance, \(N \in \{100, 200, 300, 400, 500\}\)
training data are generated, the \RHS{} \(v_i, i \in [I]\) are all \(1\),
coefficients \(c\) in the objective function are uniformly sampled from integers
between \(1\) and \(100\), the radius \(\delta\) of the Wasserstein ball enumerates
\(\{0.05, 0.07, 0.09, 0.11, 0.13, 0.15, 0.17, 0.19, 0.21, 0.23, 0.25, 0.27\}\),
and the risk level \(\epsilon = 0.1\). For each parameter setting, we compute the
average of out-of-sample performance, as well as its \(90\%\) confidence interval, for the
returned solution over \(5\) different random instances and study how the radius \(\delta\) and the sample
size \(N\) affect the out-of-sample performance of (\DRCCbC{}) and
(\SAA{}).

\begin{figure}[!htbp]
\centering
\begin{subfigure}{0.48\textwidth}
\centering
\includegraphics[width=\textwidth]{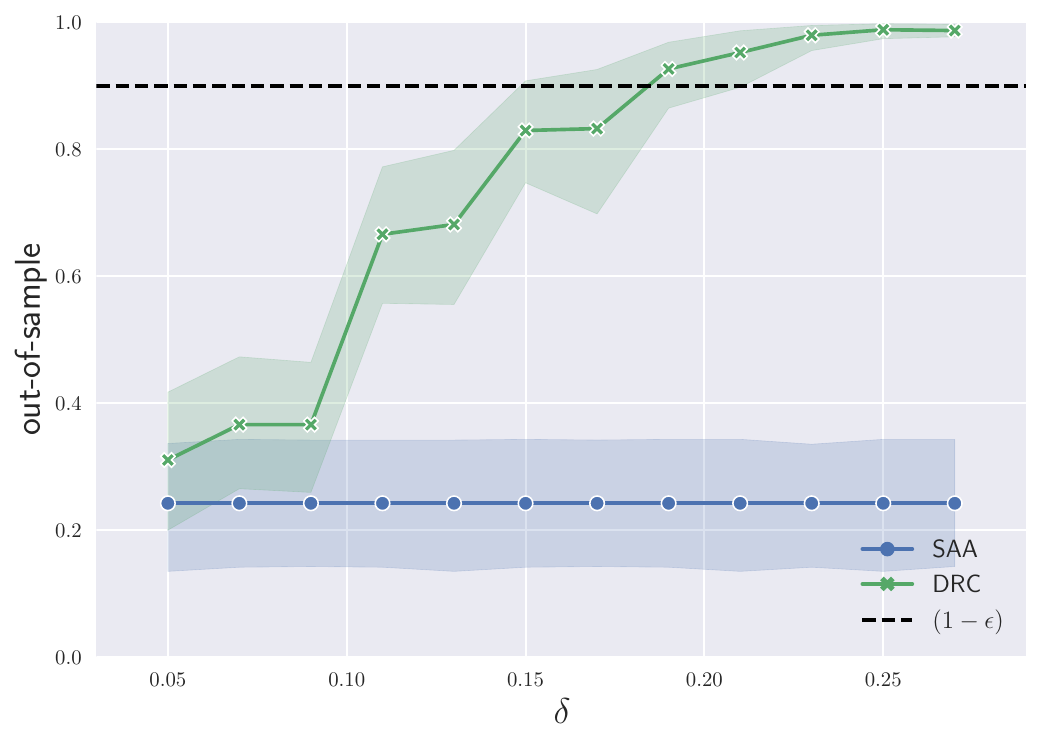}
\caption{\(\epsilon = 0.1, N = 20\)}
\label{fig:exps-delta-vs-oos-a}
\end{subfigure}
\hfill
\begin{subfigure}{0.48\textwidth}
\centering
\includegraphics[width=\textwidth]{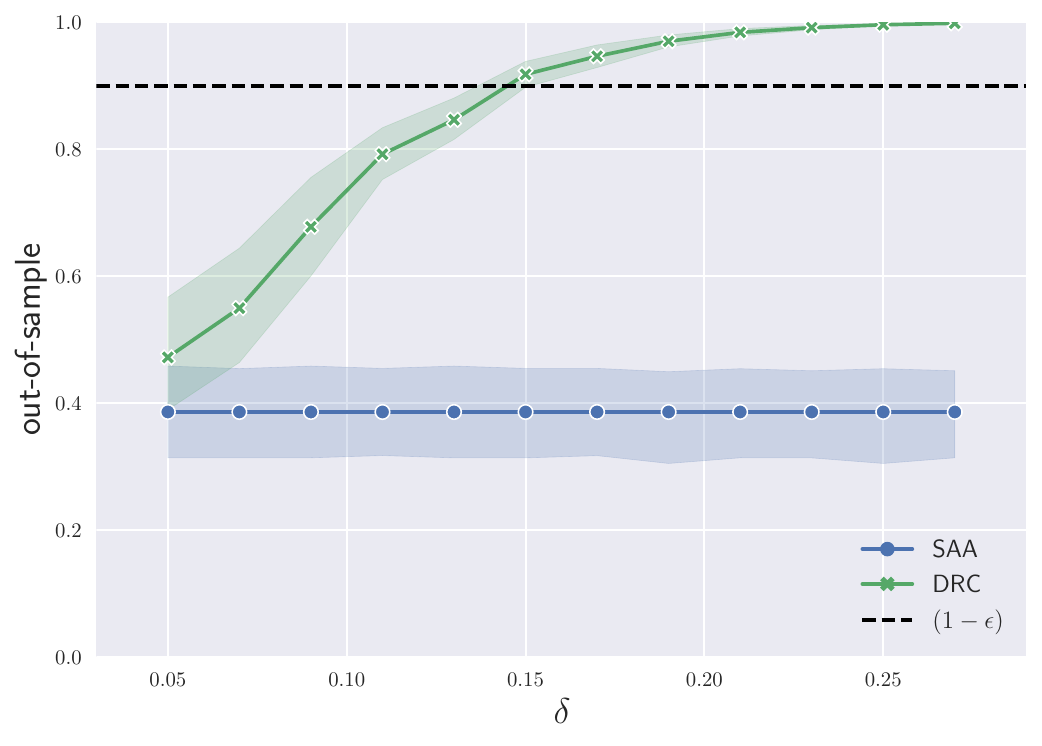}
\caption{\(\epsilon = 0.1, N = 300\)}
\end{subfigure}
\caption{Out-Of-Sample Performance of (\DRCCbC{}) and (\SAA{}) as a function of \(\delta\).}%
\label{fig:exps-delta-vs-oos}
\end{figure}
\subsubsection{Impact of the radius \texorpdfstring{\(\delta\)}{d} on out-of-sample performance}%
\label{sec:exps-oos-radius}

Fig.~\ref{fig:exps-delta-vs-oos} displays the impact of \(\delta\) on
out-of-sample performance of the (\DRCCbC{}) and (\SAA{}) optimal solutions. We notice that
the out-of-sample performance of (\DRCCbC{}) increases as \(\delta\) becomes larger. This
makes sense because the larger the radius \(\delta\) is, the more distributions
the Wasserstein ball \(\mathcal{P}\) includes, and accordingly the more likely
(\DRCCbC{}) generates a solution feasible with respect to \(\Probt{}\). The
(\SAA{}) model, however, yields a low out-of-sample performance. Specifically, Fig.~\ref{fig:exps-delta-vs-oos-a} shows that, with a proper
choice of \(\delta\), (\DRCCbC{}) is able to provide a reliable solution
satisfying the chance constraint even when the training data is very limited. In
contrast, (\SAA{}) does not return a feasible solution even if the size of
training data has a tenfold increase.

\begin{figure}[!htbp]
\centering
\begin{subfigure}{0.48\textwidth}
\centering
\includegraphics[width=\textwidth]{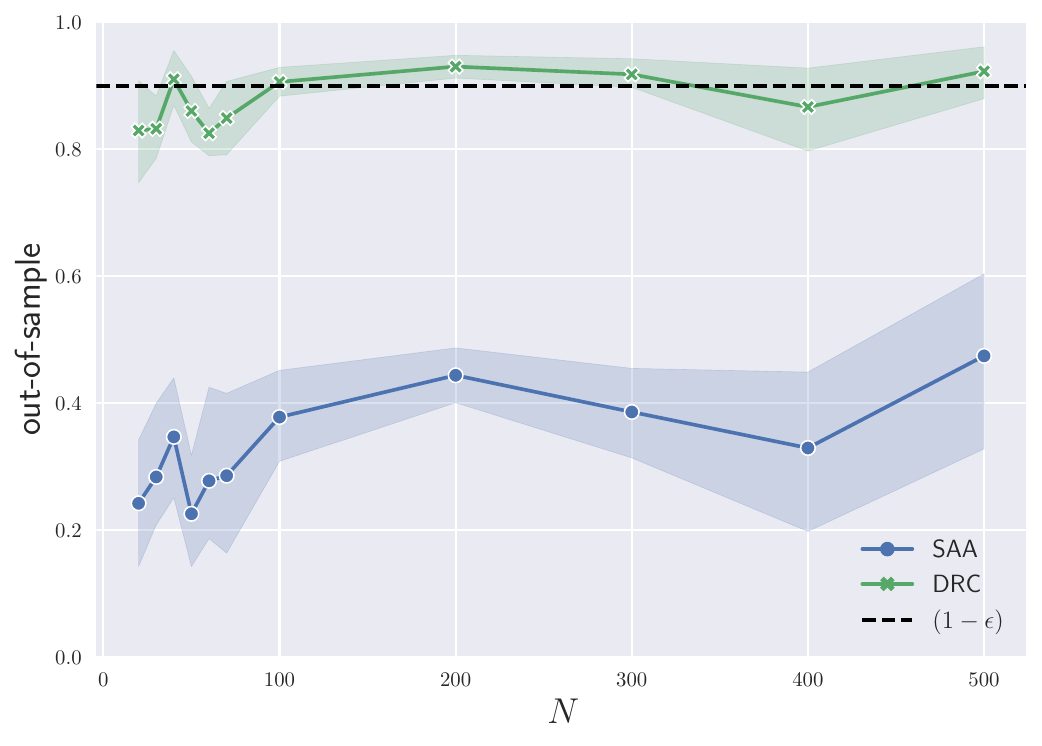}
\caption{\(\epsilon = 0.1, \delta = 0.15\)}
\end{subfigure}
\hfill
\begin{subfigure}{0.48\textwidth}
\centering
\includegraphics[width=\textwidth]{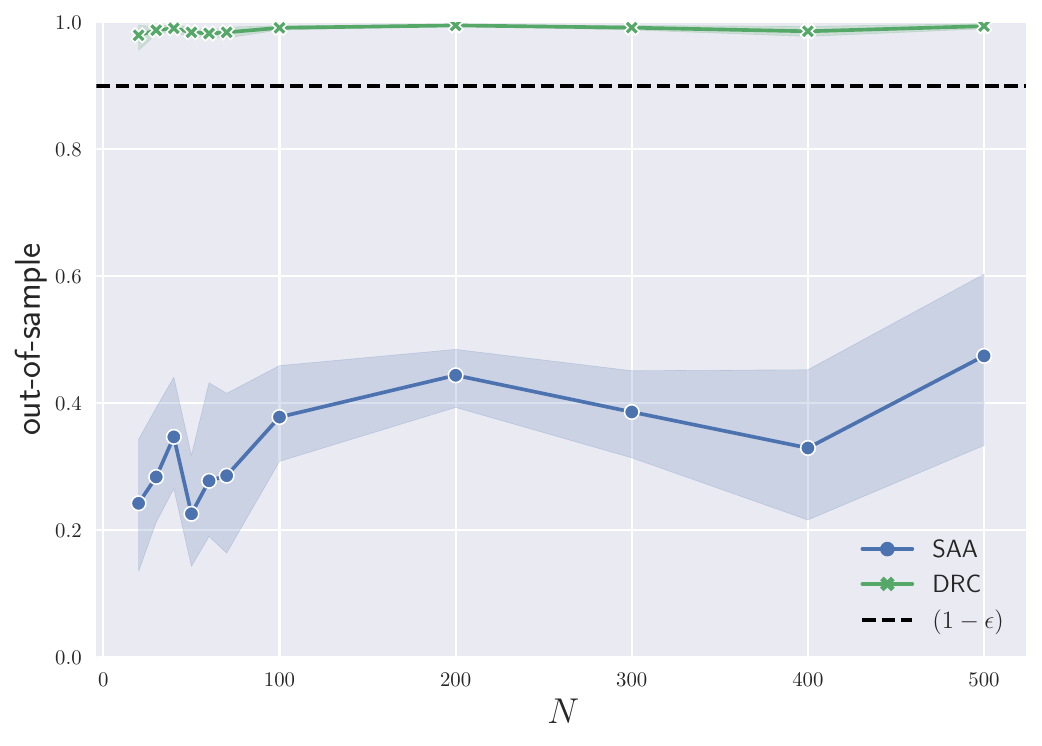}
\caption{\(\epsilon = 0.1, \delta = 0.23\)}
\end{subfigure}
\caption{Out-of-sample Performance of (\DRCCbC{}) and (\SAA{}) as a function of
  \(N\).}%
\label{fig:exps-sample-vs-oos}
\end{figure}
\subsubsection{Impact of the data size \texorpdfstring{\(N\)}{N} on out-of-sample performance}%
\label{sec:exps-oos-sample}

Fig.~\ref{fig:exps-sample-vs-oos} displays how \(N\) affects
out-of-sample performance of the (\DRCCbC{}) and (\SAA{}) optimal solutions. We observe that
the out-of-sample performance of both models shows an increasing trend as the data size
becomes larger. Specifically, (\DRCCbC{}) can return a high-quality solution with a proper
choice of radius \(\delta\) depending on much less data. In contrast, (\SAA{})
does not produce a feasible solution even when \(500\) training data are
provided. In addition, (\DRCCbC{}) is significantly more stable than (\SAA{}),
as (\DRCCbC{}) generates a much narrower \(90\%\) confidence interval (the
shaded area in Fig.~\ref{fig:exps-sample-vs-oos}) than (\SAA{}). This
demonstrates that (\DRCCbC{}) is capable of generating reliable and stable
solutions even with limited training data.

\begin{figure}[!htbp]
\centering
\includegraphics[width=\textwidth]{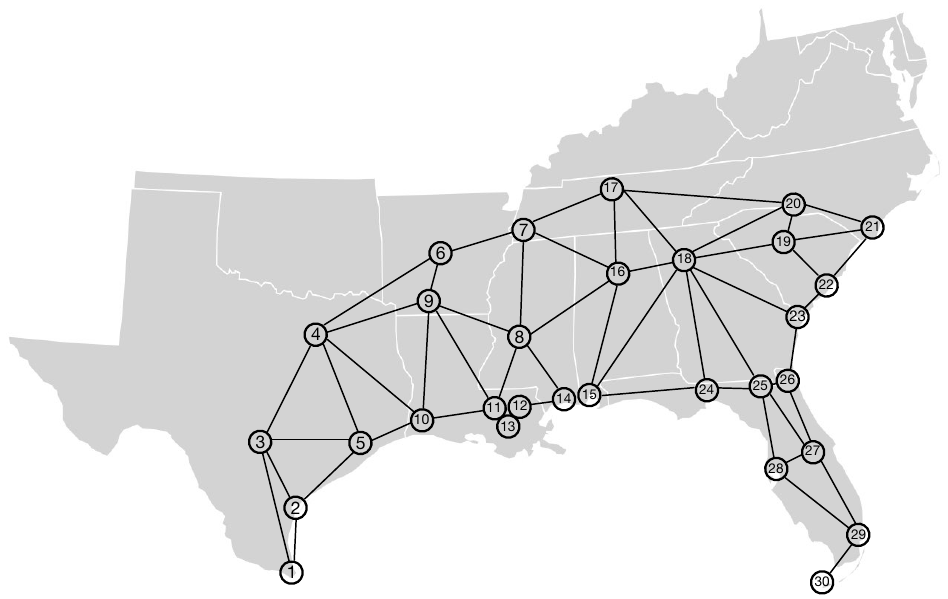}
\caption{Network of a medical facility location case study.}
\label{fig:exps-transport-net}
\end{figure}
\subsection{Dependent uncertain parameters}%
\label{sec:out-of-sample-case-study}
\edits{
In this experiment, we test (\DRCCbC{}) with dependent uncertain parameters based on a medical facility location problem with natural disasters. Specifically, we consider an undirected network with \(I\) nodes and a set \(\mathcal{E} \subseteq [I] \times [I]\) of roads connecting these nodes. As a special case of (\DRCCbC{}) with $n = I$, this problem establishes medical facilities among these $I$ nodes in order to cover the medical needs from the same set of nodes under coverage uncertainty. The binary random variables $\tilde{\xi}_{ij}$, $i,j \in [I]$ indicate whether a facility in node $j$ can cover the need from node $i$. To generate correlated $\tilde{\xi}_{ij}$, we consider natural disasters and let $\tilde{H}_j$ denote a binary random variable such that $\tilde{H}_j = 1$ indicates a disaster in node $j$. In addition, we partition the nodes into subsets $\{S_k\}_{k=1}^K$ according to geographic proximity and designate that if a node in a subset is affected by the disaster then all nodes in the same subset are affected, i.e., $\tilde{H}_j = 1$ for all $j \in S_k$ whenever $\tilde{H}_i = 1$ for an $i \in S_k$. Furthermore, we consider a random travel time $\tilde{T}_e$ on each road $e \in \mathcal{E}$ and assume that $\tilde{T}_e$ depends on the two nodes road $e$ connects: if neither of the nodes is affected by the disaster then $\tilde{T}_e$ takes a constant value $T^0_e$; and if either node is affected then $\tilde{T}_e$ follows a Gaussian distribution $N(4 \cdot T^0_e, 1)$. Accordingly, we define $\tilde{\xi}_{ij}$ as
\begin{align}
  \label{eq:exps-model}
  \tilde{\xi}_{ij} := (1 - \tilde{H}_j) \times \Ind{%
  \sum_{e \in \Path{}(j, i)} \tilde{T}_e \leq T_{\text{max}}},
\end{align}
where \(\Path{}(j, i)\) represents the shortest path from node \(j\) to
node \(i\) after the disaster and \(T_{\text{max}}\) denotes the longest time allowed to provide
medical service to a disaster node. That is, $\tilde{\xi}_{ij} = 1$ only when node $j$ is not affected by the disaster and the shortest path from $j$ to $i$ is shorter than $T_{\text{max}}$.
\begin{figure}[!htbp]
\centering
\begin{subfigure}{0.48\textwidth}
\centering
\includegraphics[width=\textwidth]{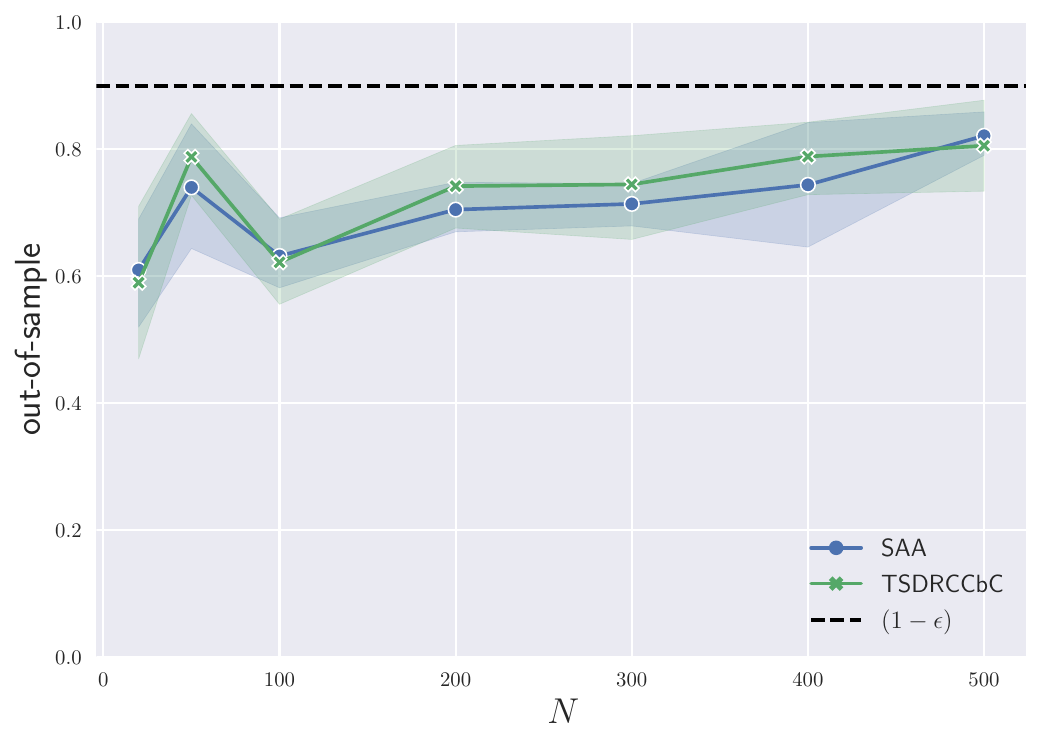}
\caption{\(\epsilon = 0.1, \delta = 0.03\)}
\end{subfigure}
\hfill
\begin{subfigure}{0.48\textwidth}
\centering
\includegraphics[width=\textwidth]{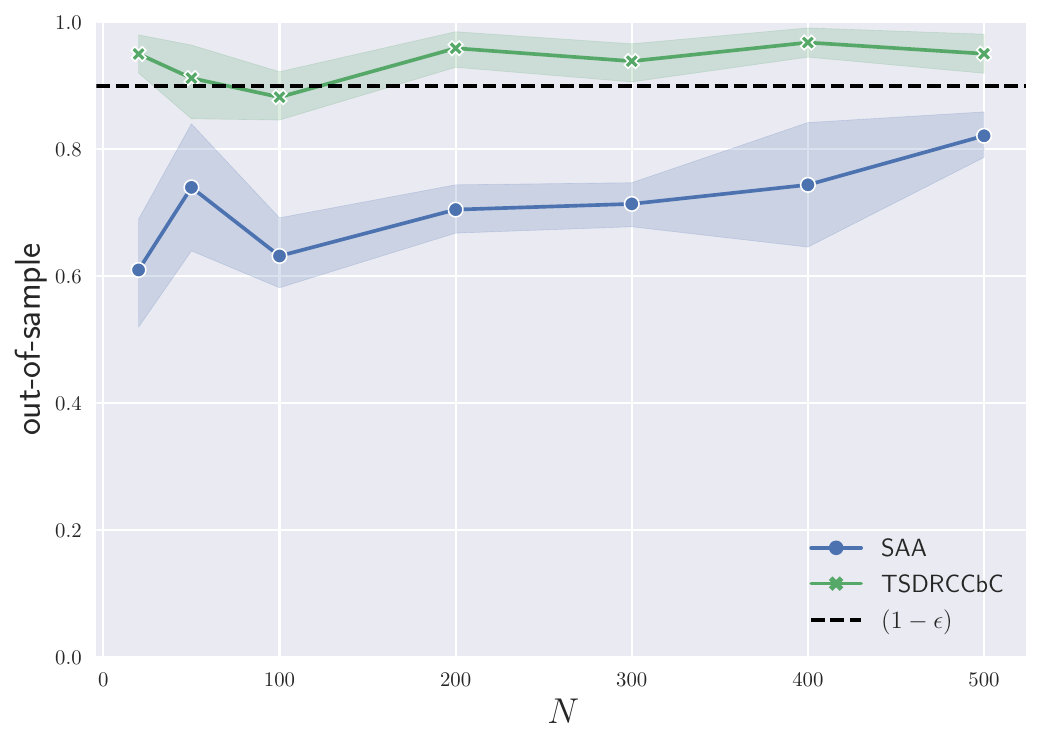}
\caption{\(\epsilon = 0.1, \delta = 0.15\)}
\end{subfigure}
\caption{Out-of-sample performance of (\DRCCbC{}) as a function of \(N\).}%
\label{fig:exps2-sample-vs-oos}
\end{figure}

Our case study is based on a network of southern US with $n = I = 30$ (cf.~\cite{shehadeh-2020-distr-robus, rawls-2010-pre-posit,velasquez-2020-prepos-disas}; see Figure~\ref{fig:exps-transport-net}). 
The partition $\{S_k\}_{k=1}^K$ of these nodes includes subsets \(\set{11, 12, 13}\), \(\set{14, 15}\), \(\set{19, 20}\), \(\set{22, 23}\), 
\(\set{24, 25, 26}\), \(\set{27, 28}\), and all the remaining subsets are singletons.

First, to generate $\tilde{H}_j$, we assign all nodes into three categories: \(\mathcal{C}_1 := \set{2, 5, 11, 13, 14, 15, 29, 30, 21, 22}\), \(\mathcal{C}_2 := \set{1, 3, 10, 12, 25, 26, 27, 28, 23}\), and \(\mathcal{C}_3 := [I] \setminus (\mathcal{C}_1 \cup \mathcal{C}_2)\). For each node $j$, we independently draw a Gaussian random number centered around the category index of $j$ with a standard deviation of $1$. For example, we sample from $N(2, 1)$ for node $1$ and from $N(1, 1)$ for node $2$. Then, we set $\tilde{H}_j = 1$ if node $j$ generates the smallest Gaussian random number, and we broadcast $\tilde{H}_j = 1$ in the subset $j$ belongs with. For example, if $j = 1$ then $\tilde{H}_1 = 1$ and $\tilde{H}_i = 0$ for all $i \in [I]\setminus\{1\}$; and if $j = 11$ then $\tilde{H}_{11} = \tilde{H}_{12} = \tilde{H}_{13} = 1$ and $\tilde{H}_i = 0$ for all $i \in [I]\setminus\{11, 12, 13\}$.

Second, we generate the random travel times $\tilde{T}_e$ for all $e \in \mathcal{E}$ based on $\{\tilde{H}_j\}_{j=1}^I$ as described above and the random coverage $\tilde{\xi}$ using \eqref{eq:exps-model} and $T_{\text{max}} := 12$. For each node $i \in [I]$, the coverage level $v_i$ is set to be 4 minus the category index of $i$. For example, $v_1 = 2$ and $v_2 = 3$.

\begin{figure}[!htbp]
\centering
\begin{subfigure}{0.48\textwidth}
\centering
\includegraphics[width=\textwidth]{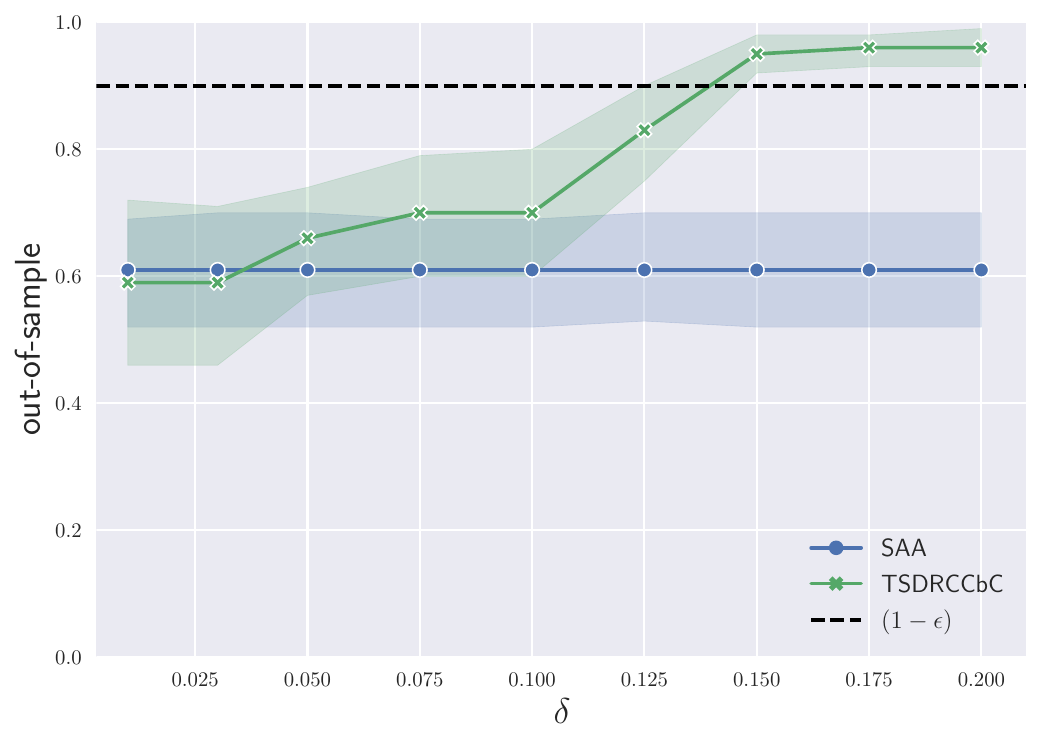}
\caption{\(\epsilon = 0.1, N = 20\)}
\end{subfigure}
\hfill
\begin{subfigure}{0.48\textwidth}
\centering
\includegraphics[width=\textwidth]{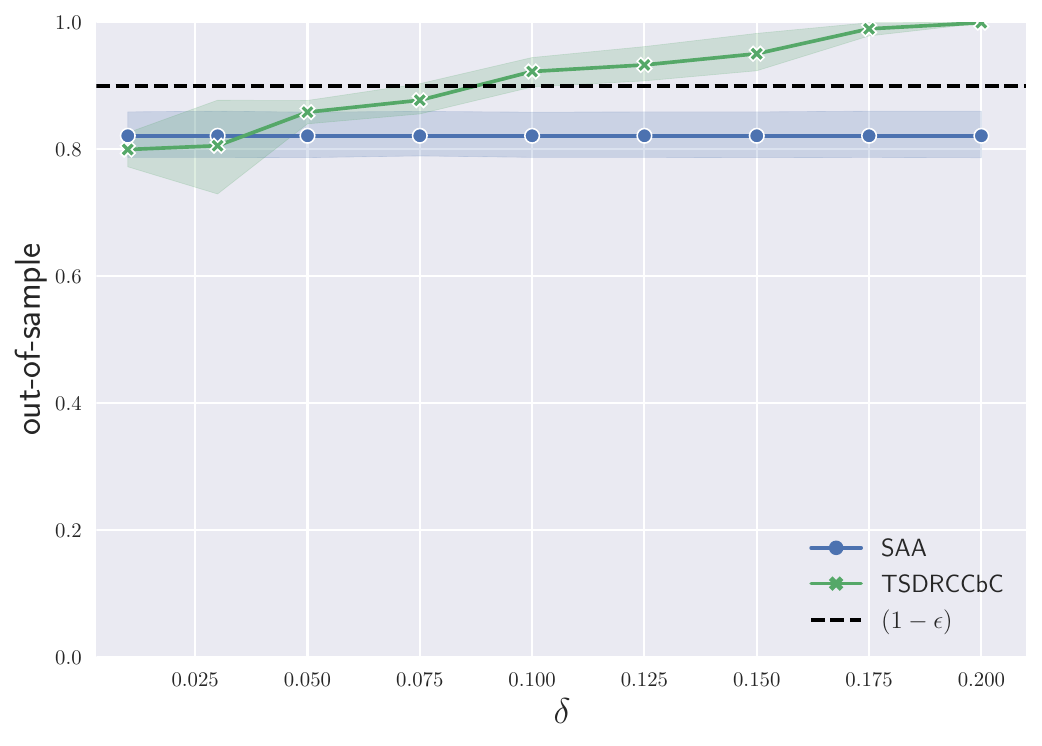}
\caption{\(\epsilon = 0.1, N = 500\)}
\end{subfigure}
\caption{Out-of-sample performance of (\DRCCbC{}) as a function of \(\delta\).}%
\label{fig:exps2-delta-vs-oos}
\end{figure}
Third, to evaluate the out-of-sample performance of (\DRCCbC{}), we generate \(N \in \{20, 50, 100, 200, 300,\) \(400, 500\}\) samples as training data,
the radius \(\delta\) of the Wasserstein ball are taken from
\(\{0.01, 0.03,\) \(0.05, 0.075, 0.1, 0.125, 0.15, 0.175, 0.2\}\), and the risk
level \(\epsilon = 0.1\). For each parameter setting, we generate \(5\) different
random instances, evaluate the out-of-sample performance using \(100,000\)
random coverage samples generated in the same way as the training data, and report the average performance along with the
\(90\%\) confidence interval in Figures~\ref{fig:exps2-sample-vs-oos}
and~\ref{fig:exps2-delta-vs-oos}. Figure~\ref{fig:exps2-sample-vs-oos}
depicts the out-of-sample performance as a function of the data size $N$. From this figure, we observe that when \(\delta\) is small, the
performance of (\DRCCbC{}) and (\SAA{}) are similar. As \(\delta\) increases,
\eg{}, when \(\delta = 0.15\), we observe a large improvement in the out-of-sample
performance of (\DRCCbC{}) as it requires feasible solutions to satisfy
chance constraints with respect to more distributions. In
Figure~\ref{fig:exps2-delta-vs-oos}(a), even with a small training data size, \eg{}, \(N = 20\), (\DRCCbC{}) can return a reliable solution when
\(\delta\) is properly chosen, say \(\delta = 0.15\). In contrast, (\SAA{})
did not produce a feasible solution in the sense of out-of-sample performance, even when \(N = 500\) as displayed 
in Figure~\ref{fig:exps2-delta-vs-oos}(b).
}
\subsection{Value of Incorporating Support Information}\label{sec:value-support}

\edits{A key difference between this paper and~\citet{xie-2019-distr-robus} lies in the support $\Xi$ of uncertainty parameters. We derive a reformulation $Z$ of (\DRCCbC{}) in \eqref{eqn:reform-mono} by assuming a binary-valued support $\Xi = \{0, 1\}^{I\times n}$. If we relax it to be $\Xi = \mathbb{R}^{I\times n}$, a different reformulation of the feasible region produced by chance constraint \eqref{eqn:drccbc-constr} follows from Theorem 1 and Proposition 1 of~\cite{xie-2019-distr-robus}:
\begin{align}
  Z_0 = \Set{
  x \in \binaries^n \colon
  \begin{aligned}
  & \exists \, \gamma \in \reals_+, \nu \in \reals_{+}, z \in \reals^N_- \colon \\
  & \delta \nu - \gamma\epsilon \leq \frac{1}{N} \sum_{j \in [N]} z_j, \\
  & \norm{x}_{p/(p-1)} \leq \nu, \\
  & z_j + \gamma \leq \posp{x^{\top} \hat{\xi}^j_i - v_i}, \forall i \in [I], j \in [N] \\
  \end{aligned}}. \label{eq:set-covering-of-xie}
\end{align}
Intuitively, as $Z$ relies on a more restricted support, it leads to a larger feasible region and accordingly (\DRCCbC{}) becomes less conservative. The following proposition formalizes this intuition.
\begin{proposition} \label{prop:comparison}
For fixed $\delta$ and $\epsilon$, it holds that \(Z_0 \subseteq Z\).
\end{proposition}
\begin{proof}
Since $v_i \geq 1$ for all $i \in [I]$, $x = 0$ is not feasible to (\DRCCbC{}). Then, we recast $Z_0$ in \eqref{eq:set-covering-of-xie} as follows by projecting out variable $\nu$:
\begin{align*}
  Z_0 = \Set{
  x \in \binaries^n \colon
  \begin{aligned}
  & \exists \, \gamma \in \reals_+, z \in \reals^N_- \colon \\
  & \delta - \gamma\epsilon \leq \frac{1}{N} \sum_{j \in [N]} z_j, \\
  & z_j + \gamma \leq \min_{i \in [I]}\frac{\posp{x^{\top}\hat{\xi}^j_i - v_i}}{\norm{x}_*}, \forall j \in [N] \\
  \end{aligned}},
\end{align*}
where $\norm{\cdot}_* := \norm{\cdot}_{p/(p-1)}$ represents the dual norm of $\norm{\cdot}_p$. Hence, it suffices to show
\[
\frac{\posp{x^{\top} \xi - v}}{\norm{x}_*}
\leq \left( \posp{x^{\top} \xi - v + 1} \right)^{1/p}
\]
for any \(0 \neq x \in \binaries^n\),
\(\xi \in \binaries^n\), and \(v \in [n]\). To this end, we discuss the following two cases.
\begin{enumerate}
\item When \(x^{\top} \xi \leq v\), we have
      \[
      \frac{\posp{x^{\top}\xi - v}}{\norm{x}_*} = 0
      \leq \left( \posp{x^{\top} \xi - v + 1} \right)^{1/p}.
      \]
\item When \(x^{\top} \xi \geq v + 1\), let \(\mathcal{T}\) be the index set of
      \(x \wedge \xi\), \ie{},
      \(\mathcal{T} := \set{i \in [n] \colon \xi_i = x_i = 1}\), then
      \(\abs{\mathcal{T}} \geq v + 1\). Let \(\mathcal{T}'\) be \(\mathcal{T}\) with
      arbitrary \(v\) elements removed and \(\xi'\) be the indicating vector of \(\mathcal{T}'\). That is, $\xi' \in \binaries{}^n$ and, for all $i \in [n]$, $\xi'_i = 1$ if and only if $i \in \mathcal{T}'$. Then, $x^{\top}\xi' = \ones{}^{\top}\xi' = x^{\top}\xi - v$ and $\norm{\xi'}_p = (\ones{}^{\top}\xi')^{1/p}$. We finish the proof by noticing that:
      \begin{align*}
      (x^{\top} \xi - v) &
      = x^{\top} \xi' \\
      & \leq \norm{\xi'}_p \norm{x}_* \\
      & = \left( \ones{}^{\top}\xi' \right)^{1/p} \norm{x}_* \\
      & = \left( x^{\top}\xi - v \right)^{1/p} \norm{x}_* \leq \left( x^{\top}\xi - v + 1 \right)^{1/p} \norm{x}_*,
      \end{align*}
      where the first inequality is by Cauchy–Schwarz inequality and the last inequality is because \(x^{\top}\xi - v \geq 0\).
\end{enumerate}
\end{proof}
}
\begin{table}[!htbp]
\centering
\resizebox{0.8\textwidth}{!}{%
\editsII{\begin{tabular}{rrrr|rrr|rrr}
\toprule
                             &        &       &            & \multicolumn{3}{c}{\texttt{bin}} & \multicolumn{3}{c}{\texttt{cont}}           \\
\(n\)                        & \(I\)  & \(N\) & \(\delta\) & Optval                           & Time  & \OOS{} & Optval   & Time   & \OOS{} \\
\midrule
20                           & 10     & 50    & 0.01       & 136.0                            & 1.31  & 0.82   & 270.4    & 8.13   & 0.98   \\
                             &        &       & 0.05       & 159.4                            & 0.94  & 0.88   & 557.4    & 7.02   & 0.98   \\
                             &        &       & 0.1        & 203.4                            & 0.62  & 0.90   & INF      & 0.55   & N/A    \\
                             &        & 100   & 0.01       & 115.4                            & 2.38  & 0.88   & 239.8    & 15.58  & 0.98   \\
                             &        &       & 0.05       & 128.8                            & 1.39  & 0.91   & 552.4    & 18.00  & 0.99   \\
                             &        &       & 0.1        & 210.6                            & 1.16  & 0.94   & 1,020.0  & 1.66   & 1.00   \\
                             & 25     & 50    & 0.01       & 200.2                            & 2.32  & 0.84   & 365.2    & 15.99  & 0.98   \\
                             &        &       & 0.05       & 216.2                            & 1.50  & 0.87   & 615.7    & 9.96   & 0.99   \\
                             &        &       & 0.1        & 258.0                            & 0.99  & 0.89   & INF      & 0.67   & N/A    \\
                             &        & 100   & 0.01       & 213.8                            & 5.50  & 0.85   & 409.2    & 49.20  & 0.98   \\
                             &        &       & 0.05       & 233.8                            & 2.98  & 0.90   & 841.7    & 8.17   & 0.99   \\
                             &        &       & 0.1        & 304.8                            & 1.83  & 0.95   & INF      & 1.16   & N/A    \\
30                           & 10     & 50    & 0.01       & 150.2                            & 11.24 & 0.83   & 236.6    & 109.86 & 0.96   \\
                             &        &       & 0.05       & 161.6                            & 5.83  & 0.87   & 404.0    & 109.56 & 0.99   \\
                             &        &       & 0.1        & 204.2                            & 1.63  & 0.93   & 752.6    & 44.02  & 0.99   \\
                             &        & 100   & 0.01       & 147.0                            & 13.62 & 0.90   & 255.2    & 191.69 & 0.97   \\
                             &        &       & 0.05       & 162.6                            & 5.58  & 0.91   & 459.2    & 108.07 & 0.99   \\
                             &        &       & 0.1        & 223.6                            & 2.32  & 0.95   & 959.2    & 33.22  & 1.00   \\
                             & 25     & 50    & 0.01       & 205.6                            & 19.16 & 0.85   & 322.6    & 177.63 & 0.96   \\
                             &        &       & 0.05       & 226.2                            & 8.46  & 0.88   & 544.8    & 96.73  & 0.99   \\
                             &        &       & 0.1        & 252.8                            & 3.33  & 0.91   & 1,253.6  & 7.09   & 0.99   \\
                             &        & 100   & 0.01       & 173.2                            & 23.38 & 0.88   & 305.8    & 588.23 & 0.98   \\
                             &        &       & 0.05       & 199.4                            & 10.16 & 0.91   & 579.8    & 282.40 & 0.99   \\
                             &        &       & 0.1        & 261.4                            & 3.68  & 0.96   & 1,201.3  & 24.23  & 1.00   \\
\midrule
\multicolumn{4}{c|}{Average} & 197.84 & 5.47  & 0.89       & 578.41                           & 79.53 & 0.99                                \\
\bottomrule
\end{tabular} 
}}
\caption{Comparison between optimizing over \(Z_0\) and \(Z\)}\label{tab:vs-weijun-bench}
\end{table}
Next, we demonstrate Proposition~\ref{prop:comparison}
  numerically using a set of random test instances, generated in the
  same way as in Section~\ref{sec:exps-oos}, with various
  \(n, I, N, \delta\) and the same risk level \(\epsilon = 0.1\). For
  each problem size, we solve our two-stage reformulation in
  Proposition~\ref{prop:reform-two-stage} (denoted as \texttt{bin})
  and~\citet{xie-2019-distr-robus}'s reformulation
  $\min_{x \in Z_0}\{c^{\top}x\}$ (denoted as \texttt{cont}). We solve
  five random instances for each problem size and report the average
  optimal value, \editsII{out-of-sample constraint satisfaction probability (\OOS{}),} and CPU time in
  Table~\ref{tab:vs-weijun-bench}. An ``INF'' is placed in the column
  ``Optval'' if all five random instances are infeasible, and ``N/A''
  represents ``not available''. From Table~\ref{tab:vs-weijun-bench},
  we observe that for the same problem
  size,~\citet{xie-2019-distr-robus}'s reformulation always returns a
  higher optimal value than ours. For example, for
  \(n = 30, I = 25, N = 100, \delta = 0.01\), \texttt{bin} returns an
  optimal value of \(173.2\), while \texttt{cont}'s optimal value
  \(305.8\) is almost twice as large. Additionally, for
  \(n = 20, I = 25, N = 100, \delta = 0.1\), \texttt{cont} is
  infeasible in all five instances while \texttt{bin} always remains
  feasible and its average optimal value \(304.8\) is even lower than
  \texttt{cont}'s with a smaller radius \(\delta = 0.01\). \editsII{Besides,
  from Table~\ref{tab:vs-weijun-bench} we notice that the
  out-of-sample constraint satisfaction probability of~\texttt{cont} is nearly \(1.00\) across all instances, while that of~\texttt{bin} is closer to the target probability (i.e., 0.90).} These computational
  results align with the fact \(Z_0 \subseteq Z\) and demonstrate that
  incorporating support information can make (\DRCCbC{}) significantly
  less conservative.

\subsection{Strength of the two-stage reformulation and valid inequalities}%
\label{sec:exps-effi}

We demonstrate the strength of our two-stage reformulation, and the
single- and cross-scenario valid inequalities. Random test instances
with \(n \in \{60, 80\}\), \(I \in \{70, 90\}\) are generated, each of
which is paired with \(N \in \set{50, 100, 200}\),
\(\delta \in \{0.05, 0.1, 0.2, 0.3\}\), and
\(\epsilon \in \{0.05, 0.1\}\). The numerical results are reported in
Tables~\ref{tab:exps-effi-50}--\ref{tab:exps-effi-200}, where we use
\texttt{\(2\)-Stg}, \texttt{+single}, and \texttt{+cross} to denote
using the two-stage reformulation in
Proposition~\ref{prop:reform-two-stage} only, two-stage reformulation
with single-scenario inequalities, and two-stage reformulation with
both single- and cross-scenario inequalities, respectively.  In
addition, we report average CPU time (in seconds) and average
optimality gap over \(5\) random test instances, after they are
terminated due to either a proof of optimality/infeasibility (INF) or
a timelimit of \(3,600\) seconds, whichever occurs first. The number
in the parenthesis following the gap, if displayed, denotes the number
of instances we fail to prove optimality/INF within the
timelimit. \edits{Finally, an ``INF'' is placed in the column ``Gap
  (\%)'' if all five random instances are infeasible.}

\edits{From Tables~\ref{tab:exps-effi-50}--\ref{tab:exps-effi-200}, we observe that \texttt{\(2\)-Stg} generally took
shorter time to solve the easier instances than \texttt{+single}, while the
strength of single-scenario inequalities began to reveal in the harder
instances. For instance, when
\(n = 60, I = 90, \delta = 0.05, \epsilon = 0.05, N = 50\), \texttt{\(2\)-Stg}
only solved \(2\) out of \(5\) random instances and took \(2333.7\) seconds on
average, whereas \texttt{+single} solved all \(5\) instances and only took
\(70.1\) seconds on average. In Table~\ref{tab:exps-effi-50}, single-scenario inequalities on
average shortened the CPU time by approximately \(20\%\) and shrunk the
optimality gap significantly (around \(48\%\)). Specifically, they helped a lot on
problems where the ratios of \(\delta / \epsilon\) is larger, \eg{},
when \((\delta, \epsilon) = (0.05, 0.05)\) or \((\delta, \epsilon) = (0.1, 0.1)\). This
demonstrates the effectiveness of the single-scenario inequalities. In
Table~\ref{tab:exps-effi-200}, when the training data size \(N\) became larger (\(N = 200\)), the speedup
brought by single-scenario inequalities remained considerable, even though
its average CPU time and final gap were on par with \texttt{\(2\)-Stg}. For
instance, when \(n = 60, I = 70, \delta = 0.1, \epsilon = 0.1, N = 200\),
\texttt{\(2\)-Stg} solved \(4\) out of \(5\) instances and spent \(1731.0\)
seconds on average, while \texttt{+single} solved all the instances and took only
\(45\%\) of the time on average. By comparing the performance of \texttt{+cross}
and \texttt{+single}, we observe that cross-scenario inequalities can further
improve the computational performance. For example, in Table~\ref{tab:exps-effi-50}, where
\(n=80, I=90, \delta=0.05, \epsilon=0.05, N=50\), \texttt{+cross} solved \(4\)
instances to optimality within the timelimit, while \texttt{+single} solved only
\(2\), let alone \texttt{\(2\)-Stg}, which failed to solve any of the \(5\)
instances. As for the most challenging instances, in which all three approaches
failed to prove optimality within the timelimit, \texttt{+cross} proved
significantly smaller optimality gaps than \texttt{\(2\)-Stg} and
\texttt{+single}. This demonstrates the strength of the cross-scenario
inequalities on capturing the intersections of the feasible regions arising from
multiple scenarios or coverings.}
\begin{table}[!htbp]
\begin{center}
\caption{\label{tab:exps-effi-50} Benchmark between different models on synthetic
  data (\(N = 50\))}
\vspace{-3mm}
\begin{tabular*}{\textwidth}{l@{\extracolsep{\fill}}lll|rl|rl|rl}
\toprule
    &     &          &            & \multicolumn{2}{c}{\texttt{\(2\)-Stg}} & \multicolumn{2}{c}{\texttt{+single}} & \multicolumn{2}{c}{\texttt{+cross}}            \\
$n$ & $I$ & $\delta$ & $\epsilon$ & Time     & Gap (\%)                    & Time   & Gap (\%)                    & Time   & Gap (\%) \\
\midrule
60  & 70  & 0.05     & 0.05       & 1498.0   & 2          (1)              & 63.5   & 0                           & 38.0   & 0        \\
    &     &          & 0.10       & 3057.4   & 30.1       (4)              & 3009.7 & 19.5 (4)                    & 1447.3 & 0.6 (1)  \\
    &     & 0.10     & 0.05       & 5.3      & 0                           & 11.8   & 0                           & 8.7    & 0        \\
    &     &          & 0.10       & 836.3    & 2.9        (1)              & 722.8  & 0                           & 188.8  & 0        \\
    &     & 0.20     & 0.05       & 1.6      & INF                         & 9.1    & INF                         & 6.4    & INF      \\
    &     &          & 0.10       & 5.7      & 0                           & 12.1   & 0                           & 8.6    & 0        \\
    &     & 0.30     & 0.05       & 1.1      & INF                         & 7.4    & INF                         & 0.0    & INF      \\
    &     &          & 0.10       & 3.8      & 0                           & 11.7   & 0                           & 8.7    & 0        \\
    & 90  & 0.05     & 0.05       & 2333.7   & 8.2        (3)              & 70.1   & 0                           & 31.0   & 0        \\
    &     &          & 0.10       & 3600.3   & 44         (5)              & 3600.0 & 30.3 (5)                    & 1829.0 & 0        \\
    &     & 0.10     & 0.05       & 10.4     & 0                           & 14.2   & 0                           & 10.6   & 0        \\
    &     &          & 0.10       & 3328.5   & 9          (4)              & 905.8  & 0                           & 184.5  & 0        \\
    &     & 0.20     & 0.05       & 1.8      & INF                         & 9.0    & INF                         & 6.6    & INF      \\
    &     &          & 0.10       & 9.1      & 0                           & 15.1   & 0                           & 11.0   & 0        \\
    &     & 0.30     & 0.05       & 1.2      & INF                         & 7.2    & INF                         & 0.0    & INF      \\
    &     &          & 0.10       & 4.5      & 0                           & 11.9   & 0                           & 9.2    & 0        \\
80  & 70  & 0.05     & 0.05       & 3600.0   & 42.7       (5)              & 3088.9 & 10.8 (4)                    & 2333.8 & 6.6 (2)  \\
    &     &          & 0.10       & 3600.0   & 65.2       (5)              & 3600.1 & 51.2 (5)                    & 3600.0 & 23.4 (5) \\
    &     & 0.10     & 0.05       & 27.1     & 0                           & 27.9   & 0                           & 18.6   & 0        \\
    &     &          & 0.10       & 3600.2   & 35         (5)              & 3431.0 & 10.9 (4)                    & 2622.2 & 7.1 (3)  \\
    &     & 0.20     & 0.05       & 7.2      & 0                           & 17.8   & 0                           & 12.7   & 0        \\
    &     &          & 0.10       & 23.7     & 0                           & 31.4   & 0                           & 20.2   & 0        \\
    &     & 0.30     & 0.05       & 1.7      & INF                         & 14.1   & INF                         & 0.0    & INF      \\
    &     &          & 0.10       & 7.5      & 0                           & 20.6   & 0                           & 13.9   & 0        \\
    & 90  & 0.05     & 0.05       & 3600.1   & 39.6       (5)              & 3126.5 & 5.2 (3)                     & 2203.9 & 1.1 (1)  \\
    &     &          & 0.10       & 3600.0   & 64.1       (5)              & 3600.1 & 52.1 (5)                    & 3600.0 & 22.5 (5) \\
    &     & 0.10     & 0.05       & 57.0     & 0                           & 30.3   & 0                           & 19.3   & 0        \\
    &     &          & 0.10       & 3600.1   & 35.5       (5)              & 3600.0 & 16.1 (5)                    & 3255.6 & 8.2 (4)  \\
    &     & 0.20     & 0.05       & 6.5      & 0                           & 17.0   & 0                           & 12.6   & 0        \\
    &     &          & 0.10       & 44.0     & 0                           & 32.6   & 0                           & 21.6   & 0        \\
    &     & 0.30     & 0.05       & 2.0      & INF                         & 14.9   & INF                         & 0.0    & INF      \\
    &     &          & 0.10       & 9.2      & 0                           & 19.5   & 0                           & 14.7   & 0        \\
\midrule
\multicolumn{4}{c|}{Average$^{\ast}$} & 1140.2   & 14.6                        & 911.1  & 7.5                         & 673.0  & 2.7      \\
\bottomrule
\end{tabular*}
\begin{flushleft}
\footnotesize{$^{\ast}$ INF instances are excluded when calculating average gap.}
\end{flushleft}

\end{center}
\end{table}

\begin{table}[!htbp]
\centering
\caption{\label{tab:exps-effi-100} Benchmark between different models on synthetic
  data (\(N = 100\))}
\vspace{-3mm}
\begin{tabular*}{\textwidth}{l@{\extracolsep{\fill}}lll|rl|rl|rl}
\toprule
    &     &          &            & \multicolumn{2}{c}{\texttt{\(2\)-Stg}} & \multicolumn{2}{c}{\texttt{+single}} & \multicolumn{2}{c}{\texttt{+cross}}            \\
$n$ & $I$ & $\delta$ & $\epsilon$ & Time     & Gap (\%)                    & Time   & Gap (\%)                    & Time   & Gap (\%) \\
\midrule
60  & 70  & 0.05     & 0.05       & 968.7    & 2 (1)                       & 671.5  & 0                           & 247.8  & 0        \\
    &     &          & 0.10       & 3600.0   & 43.6 (5)                    & 3600.0 & 29.9 (5)                    & 2103.2 & 6.5 (2)  \\
    &     & 0.10     & 0.05       & 10.8     & 0                           & 25.6   & 0                           & 18.7   & 0        \\
    &     &          & 0.10       & 1903.4   & 2.9  (2)                    & 990.7  & 0.7 (1)                     & 264.7  & 0        \\
    &     & 0.20     & 0.05       & 2.5      & INF                         & 18.4   & INF                         & 13.9   & INF      \\
    &     &          & 0.10       & 13.5     & 0                           & 32.2   & 0                           & 18.5   & 0        \\
    &     & 0.30     & 0.05       & 2.0      & INF                         & 14.8   & INF                         & 0.0    & INF      \\
    &     &          & 0.10       & 8.0      & 0                           & 26.6   & 0                           & 17.4   & 0        \\
    & 90  & 0.05     & 0.05       & 2696.9   & 8 (3)                       & 1159.9 & 2.1 (1)                     & 307.0  & 0        \\
    &     &          & 0.10       & 3600.0   & 48.1 (5)                    & 3600.1 & 35 (5)                      & 2594.6 & 8.5 (3)  \\
    &     & 0.10     & 0.05       & 18.0     & 0                           & 34.0   & 0                           & 22.3   & 0        \\
    &     &          & 0.10       & 2440.3   & 4.8  (3)                    & 1717.0 & 1.7 (1)                     & 544.9  & 0.7 (1)  \\
    &     & 0.20     & 0.05       & 2.7      & INF                         & 18.1   & INF                         & 13.9   & INF      \\
    &     &          & 0.10       & 24.0     & 0                           & 53.4   & 0                           & 27.4   & 0        \\
    &     & 0.30     & 0.05       & 2.1      & INF                         & 15.0   & INF                         & 0.0    & INF      \\
    &     &          & 0.10       & 11.6     & 0                           & 33.1   & 0                           & 20.9   & 0        \\
80  & 70  & 0.05     & 0.05       & 3600.0   & 30.6 (5)                    & 3305.9 & 9.8 (4)                     & 2528.1 & 6.1 (3)  \\
    &     &          & 0.10       & 3600.0   & 65   (5)                    & 3600.2 & 51.1 (5)                    & 3600.0 & 26.1 (5) \\
    &     & 0.10     & 0.05       & 225.6    & 0                           & 102.0  & 0                           & 53.8   & 0        \\
    &     &          & 0.10       & 3600.1   & 29.8 (5)                    & 3600.0 & 16.6 (5)                    & 2941.8 & 9.2 (3)  \\
    &     & 0.20     & 0.05       & 7.4      & 0                           & 33.4   & 0                           & 24.6   & 0        \\
    &     &          & 0.10       & 415.2    & 0                           & 234.7  & 0                           & 60.2   & 0        \\
    &     & 0.30     & 0.05       & 2.9      & INF                         & 29.3   & INF                         & 0.0    & INF      \\
    &     &          & 0.10       & 23.8     & 0                           & 60.8   & 0                           & 35.5   & 0        \\
    & 90  & 0.05     & 0.05       & 3600.0   & 30.8 (5)                    & 3579.4 & 10.4 (4)                    & 3197.5 & 7.4 (4)  \\
    &     &          & 0.10       & 3600.3   & 65.5 (5)                    & 3601.0 & 52 (5)                      & 3600.3 & 27.1 (5) \\
    &     & 0.10     & 0.05       & 46.6     & 0                           & 50.8   & 0                           & 35.1   & 0        \\
    &     &          & 0.10       & 3600.0   & 27.1 (5)                    & 3600.0 & 15.2 (5)                    & 3241.2 & 7.5 (3)  \\
    &     & 0.20     & 0.05       & 7.2      & 0                           & 31.9   & 0                           & 23.9   & 0        \\
    &     &          & 0.10       & 95.5     & 0                           & 70.5   & 0                           & 38.1   & 0        \\
    &     & 0.30     & 0.05       & 3.1      & INF                         & 28.7   & INF                         & 0.0    & INF      \\
    &     &          & 0.10       & 24.4     & 0                           & 59.2   & 0                           & 34.1   & 0        \\
\midrule
\multicolumn{4}{c|}{Average$^{\ast}$} & 1179.9   & 13.8                        & 1062.4 & 8.6                         & 800.1  & 3.8      \\
\bottomrule
\end{tabular*}
\begin{flushleft}
\footnotesize{$^{\ast}$ INF instances are excluded when calculating average gap.}
\end{flushleft}

\end{table}

\begin{table}[!htbp]
\centering
\caption{\label{tab:exps-effi-200} Benchmark between different models on synthetic
  data (\(N = 200\))}
\vspace{-3mm}
\begin{tabular*}{\textwidth}{l@{\extracolsep{\fill}}lll|rl|rl|rl}
\toprule
    &     &          &            & \multicolumn{2}{c}{\texttt{\(2\)-Stg}} & \multicolumn{2}{c}{\texttt{+single}} & \multicolumn{2}{c}{\texttt{+cross}}            \\
$n$ & $I$ & $\delta$ & $\epsilon$ & Time     & Gap (\%)                    & Time   & Gap (\%)                    & Time   & Gap (\%) \\
\midrule
60  & 70  & 0.05     & 0.05       & 432.4    & 0                           & 326.5  & 0                           & 81.6   & 0        \\
    &     &          & 0.10       & 3600.0   & 44.7 (5)                    & 3600.0 & 32.8 (5)                    & 3428.5 & 7.7 (4)  \\
    &     & 0.10     & 0.05       & 14.2     & 0                           & 47.8   & 0                           & 34.2   & 0        \\
    &     &          & 0.10       & 1731.0   & 0.6 (1)                     & 745.4  & 0                           & 152.7  & 0        \\
    &     & 0.20     & 0.05       & 4.3      & INF                         & 35.5   & INF                         & 31.5   & INF      \\
    &     &          & 0.10       & 25.2     & 0                           & 67.3   & 0                           & 38.1   & 0        \\
    &     & 0.30     & 0.05       & 3.4      & INF                         & 30.0   & INF                         & 0.0    & INF      \\
    &     &          & 0.10       & 13.4     & 0                           & 50.1   & 0                           & 39.5   & 0        \\
    & 90  & 0.05     & 0.05       & 2044.0   & 2.1 (2)                     & 809.0  & 0                           & 105.8  & 0        \\
    &     &          & 0.10       & 3600.0   & 48.8 (5)                    & 3600.9 & 36.4 (5)                    & 3566.4 & 13 (4)   \\
    &     & 0.10     & 0.05       & 22.9     & 0                           & 63.3   & 0                           & 41.1   & 0        \\
    &     &          & 0.10       & 2930.7   & 6.3 (4)                     & 2457.6 & 2.8 (3)                     & 461.4  & 0        \\
    &     & 0.20     & 0.05       & 5.0      & INF                         & 36.6   & INF                         & 26.9   & INF      \\
    &     &          & 0.10       & 48.5     & 0                           & 95.8   & 0                           & 50.0   & 0        \\
    &     & 0.30     & 0.05       & 3.7      & INF                         & 29.1   & INF                         & 0.0    & INF      \\
    &     &          & 0.10       & 23.3     & 0                           & 74.4   & 0                           & 46.3   & 0        \\
80  & 70  & 0.05     & 0.05       & 3600.0   & 31.3 (5)                    & 3600.0 & 21.2 (5)                    & 3600.0 & 12.9 (5) \\
    &     &          & 0.10       & 3600.1   & 75.2  (5)                   & 3600.4 & 54.3 (5)                    & 3600.3 & 35.4 (5) \\
    &     & 0.10     & 0.05       & 193.8    & 0                           & 165.2  & 0                           & 78.3   & 0        \\
    &     &          & 0.10       & 3600.0   & 31.3 (5)                    & 3600.2 & 22.1 (5)                    & 3600.0 & 12.9 (5) \\
    &     & 0.20     & 0.05       & 15.2     & 0                           & 71.2   & 0                           & 56.2   & 0        \\
    &     &          & 0.10       & 180.7    & 0                           & 200.7  & 0                           & 97.5   & 0        \\
    &     & 0.30     & 0.05       & 5.8      & INF                         & 58.6   & INF                         & 0.0    & INF      \\
    &     &          & 0.10       & 46.0     & 0                           & 119.7  & 0                           & 83.9   & 0        \\
    & 90  & 0.05     & 0.05       & 3600.0   & 23.3 (5)                    & 3600.0 & 14.5 (5)                    & 2141.8 & 3.2 (2)  \\
    &     &          & 0.10       & 3600.2   & 70.2  (5)                   & 3601.2 & 51.7 (5)                    & 3600.0 & 30.1 (5) \\
    &     & 0.10     & 0.05       & 110.8    & 0                           & 163.7  & 0                           & 98.2   & 0        \\
    &     &          & 0.10       & 3600.0   & 30.3 (5)                    & 3600.0 & 17.6 (5)                    & 3600.0 & 6.3 (5)  \\
    &     & 0.20     & 0.05       & 14.3     & 0                           & 69.9   & 0                           & 47.9   & 0        \\
    &     &          & 0.10       & 856.0    & 0.1                         & 787.7  & 0                           & 136.4  & 0        \\
    &     & 0.30     & 0.05       & 6.0      & INF                         & 56.6   & INF                         & 0.0    & INF      \\
    &     &          & 0.10       & 123.2    & 0                           & 278.2  & 0                           & 110.7  & 0        \\
\midrule
\multicolumn{4}{c|}{Average$^{\ast}$} & 1176.7   & 9.1                         & 1113.8 & 9.7                         & 904.9  & 4.7      \\
\bottomrule
\end{tabular*}
\begin{flushleft}
\footnotesize{$^{\ast}$ INF instances are excluded when calculating average gap.}
\end{flushleft}

\end{table}


\clearpage
\begin{appendices}
\section{Proof of Proposition~\ref{reform:np-hard-sep}} \label{sec:appendix}
\edits{
\begin{proof}[Proof]\label{appendix:np-hard-sep-pf}
Consider a graph \(\mathcal{G} := (\mathcal{V}, \mathcal{E})\) with vertex set
\(\mathcal{V}\) and edge set \(\mathcal{E}\), on which the classic NP-hard
vertex cover problem has the following binary linear formulation:
\begin{align*}
  \min ~~
  & \sum_{u \in \mathcal{V}} x_u, \tag{VC}\\
  \st ~~
  & x^{\top}\xi_{u, v} \geq 1, \forall (u, v) \in \mathcal{E}, \\
  & x_u \in \binaries, \forall u \in \mathcal{V},
\end{align*}
where binary variables \(x_u\) indicate whether node \(u \in \mathcal{V}\)
is part of the vertex cover and \(\xi_{u,v}\) is a binary vector with two
nonzero entries: \(\xi_{u,v} = e_u + e_v\) and
\(x^{\top}\xi_{u,v} = x_u + x_v\) for all \(x \in \binaries^{\abs{\mathcal{V}}}\).
In particular, a vertex cover \(x\) can cover every edge twice
if and only if all nodes are in the cover, \ie{}, \(x = \ones\). We provide
a polynomial reduction from (VC) to the following instance
of~\eqref{eqn:reform-nphard} to finish the proof:
\begin{align}
  \min_{x \in \binaries^{\abs{\mathcal{V}'}}} ~~
  & L(x) := \frac{1}{\abs{\mathcal{V}'}} \sum_{u \in \mathcal{V}} x_u +
    \frac{1}{\abs{\mathcal{V}'}} x_{w'} + x_{w} -
    \min_{(u, v) \in \mathcal{E}'} \left( x^{\top} \xi_{u,v} \right)^{1/p}, %
    \label{eqn:rebut-nphard}
\end{align}
where we add two new nodes \(w\) and \(w'\) and augment the graph \(\mathcal{G}\) to obtain
\(\mathcal{G}' := (\mathcal{V}', \mathcal{E}')\),
\(\mathcal{V}' := \mathcal{V} \cup \set{w, w'}\), and
\(\mathcal{E}' := \mathcal{E} \cup \set{(w, w')}\). Since the optimal value of
(VC) is bounded above by \(\abs{\mathcal{V}}\), we are only interested in
whether there is a vertex cover of size less than or equal to \(K \in \integers_+, 1 \leq K \leq \abs{\mathcal{V}} - 1\). On the one hand, suppose that there
exists a vertex cover \(x\) with size less than or equal to \(K\), then
together with \(x_{w'} := 1\) and \(x_w := 0\), \(x^{\prime} := (x, x_{w'}, x_w)\) form a
feasible solution to~\eqref{eqn:rebut-nphard} with objective value less than or equal to
\((K + 1) / \abs{\mathcal{V}'} - 1\) because
\begin{align*}
  L(x^{\prime}) = \frac{1}{\abs{\mathcal{V}'}} \sum_{u \in \mathcal{V}} x_u +
    \frac{1}{\abs{\mathcal{V}'}} x_{w'} + x_{w} -
  \min_{(u, v) \in \mathcal{E}'} \left( x^{\top} \xi_{u,v}\right)^{1/p}
  \leq \frac{1}{\abs{\mathcal{V}'}}(K + 1) - 1.
\end{align*}
On the other hand, suppose that there is a
\(y := (y_0, y_{w'}, y_{w}) \in \binaries^{\abs{\mathcal{V}'}}\) such that \(L(y) \leq (K + 1) / \abs{\mathcal{V}'} - 1\). We discuss the following three cases, with regard to the coverage number \(C(y) := \min_{(u, v) \in \mathcal{E}'} \left(y^{\top} \xi_{u, v}\right)^{1/p}\), to show that there exists a vertex cover with size at most $K$.
\begin{enumerate}
\item \(C(y) = 0\) is impossible because
      \begin{align*}
        C(y) =
        \min_{(u, v) \in \mathcal{E}'} \left(y^{\top}\xi_{u, v} \right)^{1/p}
        & \geq  \frac{1}{\abs{\mathcal{V}'}} \sum_{u \in \mathcal{V}} y_u +
        \frac{1}{\abs{\mathcal{V}'}} y_{w'} + y_{w} -
        \frac{1}{\abs{\mathcal{V}'}} (K + 1) + 1  \\
        & \geq -\frac{1}{\abs{\mathcal{V}'}} (K + 1) + 1
          \geq 1 - \frac{\abs{\mathcal{V}} + 1}{\abs{\mathcal{V}} + 2} > 0.
      \end{align*}
\item \(C(y) = 1\): suppose that \(y_w = 1\) and \(y_{w'}\) equals zero or one. Then an alternative solution \(y' := (y_0, y'_{w'}, y'_w)\) with \(y'_{w'} = 1\) and \(y'_w = 0\) to formulation \eqref{eqn:rebut-nphard} satisfies:
      \begin{align*}
        C(y') & = \min \Set{
        \min_{(u, v) \in \mathcal{E}} \left(y_0^{\top}\xi_{u, v}  \right)^{1/p},
        (y'_{w'} + y'_{w})^{1/p}
        } = 1 \quad \text{and} \\
        L(y') & \leq L(y) \leq \frac{1}{\abs{\mathcal{V}'}} (K + 1) - 1.
      \end{align*}
      It follows that \(y_0\) is a vertex cover of \(\mathcal{G}\), and we may assume that \(y_{w'} = 1, y_w = 0\) in formulation \eqref{eqn:rebut-nphard} without loss of optimality. Hence,
      \begin{align*}
        L(y) = \frac{1}{\abs{\mathcal{V}'}} \sum_{u \in \mathcal{V}} (y_0)_u +
        \frac{1}{\abs{\mathcal{V}'}} y_{w'} + y_{w} -
        C(y)
        = \frac{1}{\abs{\mathcal{V}'}} \left(
        \sum_{u \in \mathcal{V}} (y_0)_u + 1 \right) - 1
        \leq \frac{1}{\abs{\mathcal{V}'}} (K + 1) - 1,
      \end{align*}
      which implies that \(y_0\) is a vertex cover of size at most \(K\).
\item \(C(y) = 2^{1/p}\) is also impossible. Indeed, since \(C(y) = 2^{1/p}\), every edge of $\mathcal{G}'$ is covered twice, implying that
      \(y_{w'} = w_{w} = (y_0)_u = 1\) for all \(u \in \mathcal{V}\). Then,
      \begin{align*}
        L(y) = \frac{1}{\abs{\mathcal{V}'}} ( \abs{\mathcal{V}} + 1) + 1 - 2^{1/p}
        \leq \frac{1}{\abs{\mathcal{V}'}} (K + 1) - 1
        \leq \frac{1}{\abs{\mathcal{V}'}} \abs{\mathcal{V}} - 1.
      \end{align*}
      Simplifying the two ends of the above inequalities gives us
      \(2 + \abs{\mathcal{V}'}^{-1} \leq 2^{1/p}\), which is impossible for
      \(p \geq 1\). Therefore, \(C(y)\) cannot be \(2^{1/p}\).
\end{enumerate}
To sum up, if there is a feasible solution $y$ to formulation \eqref{eqn:rebut-nphard} with $L(y) \leq (K + 1) / \abs{\mathcal{V}'} - 1$, then there is a vertex cover of \(\mathcal{G}\) with size at most \(K\).
\end{proof}
}

\end{appendices}

\printbibliography{}

\end{document}